\documentclass[12pt]{amsart}
\usepackage{amssymb}
\usepackage{euscript}
\usepackage[hypertex]{hyperref}
\def\today{March 6, 2013}

\catcode`\@=11
\def\@evenfoot{\rule{0pt}{20pt}[\today] \hfill}
\def\@oddfoot{\rule{0pt}{20pt}\hfill [\today]}
\catcode`\@=13
\newtheorem{theorem}{Theorem}[subsection]

\newtheorem{proposition}[theorem]{Proposition}
\newtheorem{lemma}[theorem]{Lemma}

\newtheorem*{theoremC}{Theorem~C}
\newtheorem*{propositionD}{Proposition~D}

\theoremstyle{definition}
\newtheorem*{definitionA}{Definition~A}
\newtheorem*{definitionB}{Definition~B}
\newtheorem*{Example}{Example}
\newtheorem{definition}[theorem]{Definition}

\newtheorem*{Remark}{Remark}

\def\KK{{\it KK}}\def\Binftyop{{\EuScript B}_\infty}
\def\otexp#1#2{{#1}^{\ot #2}}\def\Ainftyop{{\EuScript A}_\infty}
\def\Associative{{\EuScript A}{\it ss}}\def\CC{C}
\def\N{{\mathbb N}}
\def\Fuk#1#2{
\vbox to 0pt{\vss
\hbox{\raisebox{.15em}{\mbox {\scriptsize
$\Big( \hskip -.4em
\def\arraystretch{.7}
\begin{array}{c}
#1 \cr #2  
\end{array} \hskip -.4em
\Big)$}}}}
}
\def\id{{1\!\!1}}
\def\me{m}\def\en{n}
\def\FuK(#1,#2){\Fuk#1#2}
\def\FrFuK(#1,#2){\Fr\Fuk#1#2}

\def\baloon#1{
{
\unitlength=.35pt
\begin{picture}(28.00,54.00)(0.00,0.00)
\thicklines
\put(14.00,14.00){\makebox(0.00,0.00){\scriptsize $#1$}}
\put(14.00,14.00){\circle{28.00}}
\thinlines
\put(14.00,28.00){\vector(0,1){30.00}}
\end{picture}}
}
\def\ball#1{
{
\unitlength=.35pt
\begin{picture}(28.00,54.00)(0.00,-49)
\thicklines
\put(14.00,-14.00){\makebox(0.00,0.00){\scriptsize $#1$}}
\put(14.00,-14.00){\circle{28.00}}
\thinlines
\put(14.00,-28.00){\vector(0,-1){30.00}}
\end{picture}}
}
\def\slp#1#2{
\ \begin{array}{|c|}
\vphantom{3}#1\cr\vphantom{3}#2    
\end{array}\ 
}
\def\EP{{\EuScript P}}\def\J{{\EuScript J}}
\def\SP{{\EuScript {SP}}}\def\EK{{\EuScript K}}
\def\ol#1{\overline{#1}}\def\Dec{{\it Dec}}

\def\lT{{\it lT}}\def\pT{{\it p\hskip -.1em T \hskip -.1em}}
\def\zT{{\it zT}}\def\bT{{\it d \hskip -.1emT}}
\def\free{{\sf F}}\def\vert{{\it Vert}}
\def\Tonks{{\it Ton}}
\def\epi{\twoheadrightarrow}
\def\po{partial order}
\def\Fr{\free(\Xi)}\def\rada#1#2{#1,\ldots,#2}
\def\except{\hskip .2em\raisebox{-.16em}{\rule{.8pt}{1em}}  \hskip .2em}
\def\sfP{{\sf P}}\def\Rada#1#2#3{#1_{#2},\dots,#1_{#3}}\def\ot{\otimes}
\def\raiseB#1{\raisebox{.5em}{$#1$}}
\def\lboxtimes{\raisebox{-.1em}{$\boxtimes$}}

\def\JEDNA{
{
\unitlength=0.015em
\begin{picture}(190.00,100.00)(0.00,0.00)
\put(0.00,20.00){\line(1,0){190.00}}
\put(0.00,40.00){\line(1,0){190.00}}
\put(0.00,60.00){\line(1,0){190.00}}
\put(0.00,80.00){\line(1,0){190.00}}
\thicklines
\qbezier(50.00,40.00)(70.00,0.00)(70.00,0.00)
\qbezier(160.00,80.00)(170.00,90.00)(180.00,100.00)
\qbezier(160.00,80.00)(150.00,90.00)(140.00,100.00)
\put(160.00,0.00){\line(0,1){80.00}}
\put(70.00,60.00){\line(0,1){40.00}}
\qbezier(110.00,20.00)(100.00,10.00)(90.00,0.00)
\put(70.00,60.00){\line(1,-1){60.00}}
\put(10.00,0.00){\line(1,1){60.00}}
\end{picture}}
}

\def\DVA{
{
\unitlength=0.015em
\begin{picture}(190.00,100.00)(0.00,0.00)
\put(0.00,20.00){\line(1,0){190.00}}
\put(0.00,40.00){\line(1,0){190.00}}
\put(0.00,60.00){\line(1,0){190.00}}
\put(0.00,80.00){\line(1,0){190.00}}
\thicklines
\qbezier(90.00,40.00)(80.00,20.00)(70.00,0.00)
\qbezier(30.00,20.00)(40.00,10.00)(50.00,0.00)
\qbezier(160.00,80.00)(170.00,90.00)(180.00,100.00)
\qbezier(160.00,80.00)(150.00,90.00)(140.00,100.00)
\put(160.00,0.00){\line(0,1){80.00}}
\put(70.00,60.00){\line(0,1){40.00}}
\put(70.00,60.00){\line(1,-1){60.00}}
\put(10.00,0.00){\line(1,1){60.00}}
\end{picture}}
}

\def\TRI{
{
\unitlength=0.015em
\begin{picture}(230.00,100.00)(0.00,0.00)
\put(0.00,20.00){\line(1,0){230.00}}
\put(0.00,40.00){\line(1,0){230.00}}
\put(0.00,60.00){\line(1,0){230.00}}
\put(0.00,80.00){\line(1,0){230.00}}
\thicklines
\qbezier(200.00,60.00)(210.00,80.00)(220.00,100.00)
\qbezier(200.00,60.00)(190.00,80.00)(180.00,100.00)
\qbezier(150.00,20.00)(140.00,10.00)(130.00,0.00)
\qbezier(50.00,40.00)(70.00,20.00)(90.00,0.00)
\put(200.00,0.00){\line(0,1){60.00}}
\put(90.00,100.00){\line(0,-1){20.00}}
\put(90.00,80.00){\line(1,-1){80.00}}
\put(10.00,0.00){\line(1,1){80.00}}
\end{picture}}
}

\def\CTYRI{
{
\unitlength=0.015em
\begin{picture}(230.00,100.00)(0.00,0.00)
\put(0.00,20.00){\line(1,0){230.00}}
\put(0.00,40.00){\line(1,0){230.00}}
\put(0.00,60.00){\line(1,0){230.00}}
\put(0.00,80.00){\line(1,0){230.00}}
\thicklines
\qbezier(130.00,40.00)(110.00,20.00)(90.00,0.00)
\qbezier(30.00,20.00)(40.00,10.00)(50.00,0.00)
\qbezier(200.00,60.00)(210.00,80.00)(220.00,100.00)
\qbezier(200.00,60.00)(190.00,80.00)(180.00,100.00)
\put(200.00,0.00){\line(0,1){60.00}}
\put(90.00,100.00){\line(0,-1){20.00}}
\put(90.00,80.00){\line(1,-1){80.00}}
\put(10.00,0.00){\line(1,1){80.00}}
\end{picture}}
}

\def\PET{
{
\unitlength=.017em
\begin{picture}(150.00,80.00)(0.00,0.00)
\put(0.00,60.00){\line(1,0){150.00}}
\put(0.00,40.00){\line(1,0){150.00}}
\put(0.00,20.00){\line(1,0){150.00}}
\thicklines
\qbezier(120.00,60.00)(130.00,70.00)(140.00,80.00)
\qbezier(100.00,80.00)(110.00,70.00)(120.00,60.00)
\put(120.00,0.00){\line(0,1){60.00}}
\put(50.00,80.00){\line(0,-1){40.00}}
\qbezier(30.00,20.00)(40.00,0.00)(40.00,0.00)
\qbezier(70.00,20.00)(60.00,0.00)(60.00,0.00)
\qbezier(50.00,40.00)(70.00,20.00)(90.00,0.00)
\qbezier(10.00,0.00)(30.00,20.00)(50.00,40.00)
\end{picture}}
}

\def\SEST{
{
\unitlength=0.017em
\begin{picture}(190.00,80.00)(0.00,0.00)
\put(0.00,20.00){\line(1,0){190.00}}
\put(0.00,40.00){\line(1,0){190.00}}
\put(0.00,60.00){\line(1,0){190.00}}
\thicklines
\qbezier(160.00,40.00)(170.00,60.00)(180.00,80.00)
\qbezier(160.00,40.00)(150.00,60.00)(140.00,80.00)
\qbezier(110.00,20.00)(100.00,10.00)(90.00,0.00)
\qbezier(30.00,20.00)(40.00,10.00)(50.00,0.00)
\put(160.00,40.00){\line(0,-1){40.00}}
\put(70.00,80.00){\line(0,-1){20.00}}
\put(70.00,60.00){\line(1,-1){60.00}}
\put(10.00,0.00){\line(1,1){60.00}}
\end{picture}}
}

\def\SEDM{
{
\unitlength=.017em
\begin{picture}(190.00,80.00)(0.00,0.00)
\put(0.00,20.00){\line(1,0){190.00}}
\put(0.00,40.00){\line(1,0){190.00}}
\put(0.00,60.00){\line(1,0){190.00}}
\thicklines
\qbezier(160.00,60.00)(170.00,70.00)(180.00,80.00)
\qbezier(140.00,80.00)(150.00,70.00)(160.00,60.00)
\qbezier(110.00,20.00)(100.00,10.00)(90.00,0.00)
\qbezier(50.00,40.00)(60.00,20.00)(70.00,0.00)
\put(160.00,0.00){\line(0,1){60.00}}
\put(70.00,60.00){\line(0,1){20.00}}
\put(70.00,60.00){\line(1,-1){60.00}}
\put(10.00,0.00){\line(1,1){60.00}}
\end{picture}}
}

\def\OSM{
{
\unitlength=0.017em
\begin{picture}(190.00,80.00)(0.00,0.00)
\put(0.00,20.00){\line(1,0){190.00}}
\put(0.00,40.00){\line(1,0){190.00}}
\put(0.00,60.00){\line(1,0){190.00}}
\thicklines
\qbezier(90.00,40.00)(80.00,20.00)(70.00,0.00)
\qbezier(30.00,20.00)(40.00,10.00)(50.00,0.00)
\qbezier(160.00,60.00)(170.00,70.00)(180.00,80.00)
\qbezier(140.00,80.00)(150.00,70.00)(160.00,60.00)
\put(160.00,0.00){\line(0,1){60.00}}
\put(70.00,60.00){\line(0,1){20.00}}
\put(70.00,60.00){\line(1,-1){60.00}}
\put(10.00,0.00){\line(1,1){60.00}}
\end{picture}}
}

\def\DEVET{
{
\unitlength=0.02em
\begin{picture}(150.00,60.00)(0.00,0.00)
\put(0.00,20.00){\line(1,0){150.00}}
\put(0.00,40.00){\line(1,0){150.00}}
\thicklines
\put(120.00,40.00){\line(0,-1){40.00}}
\put(50.00,60.00){\line(0,-1){20.00}}
\qbezier(120.00,40.00)(130.00,50.00)(140.00,60.00)
\qbezier(120.00,40.00)(110.00,50.00)(100.00,60.00)
\qbezier(60.00,0.00)(60.00,0.00)(70.00,20.00)
\qbezier(30.00,20.00)(40.00,0.00)(40.00,0.00)
\qbezier(50.00,40.00)(70.00,20.00)(90.00,0.00)
\qbezier(50.00,40.00)(30.00,20.00)(10.00,0.00)
\end{picture}}
}

\def\DESET
{{
\unitlength=.02em
\begin{picture}(100.00,60.00)(0.00,0.00)
\put(0.00,50.00){\line(1,0){100.00}}
\thicklines
\put(50.00,60.00){\line(0,-1){20.00}}
\qbezier(60.00,0.00)(60.00,0.00)(70.00,20.00)
\qbezier(30.00,20.00)(40.00,0.00)(40.00,0.00)
\qbezier(50.00,40.00)(70.00,20.00)(90.00,0.00)
\qbezier(50.00,40.00)(30.00,20.00)(10.00,0.00)
\end{picture}}
}
\def\JEDENACT
{{
\unitlength=.02em
\begin{picture}(100.00,60.00)(0.00,0.00)
\put(0.00,27.00){\line(1,0){100.00}}
\thicklines
\put(50.00,60.00){\line(0,-1){20.00}}
\qbezier(60.00,0.00)(60.00,0.00)(70.00,20.00)
\qbezier(30.00,20.00)(40.00,0.00)(40.00,0.00)
\qbezier(50.00,40.00)(70.00,20.00)(90.00,0.00)
\qbezier(50.00,40.00)(30.00,20.00)(10.00,0.00)
\end{picture}}
}
\def\DVANACT
{{
\unitlength=.02em
\begin{picture}(100.00,60.00)(0.00,0.00)
\put(0.00,40.00){\line(1,0){100.00}}
\thicklines
\put(50.00,60.00){\line(0,-1){20.00}}
\qbezier(60.00,0.00)(60.00,0.00)(70.00,20.00)
\qbezier(30.00,20.00)(40.00,0.00)(40.00,0.00)
\qbezier(50.00,40.00)(70.00,20.00)(90.00,0.00)
\qbezier(50.00,40.00)(30.00,20.00)(10.00,0.00)
\end{picture}}
}

\def\dvojiteypsilon{{
\unitlength=.3pt
\begin{picture}(24.00,30.00)(0.00,3.00)
\put(10.00,20.00){\line(0,-1){10.00}}
\bezier{20}(10.00,10.00)(15,5)(20.00,0.00)
\bezier{20}(10.00,10.00)(5,5)(0.00,0.00)
\bezier{20}(10.00,20.00)(15,25)(20.00,30.00)
\bezier{20}(0.00,30.00)(5,25)(10.00,20.00)
\end{picture}}}

\def\zeroatwo#1{%
{
\unitlength=.5pt
\begin{picture}(36.00,30.00)(-3.00,10.00)
\put(10.00,20.00){\makebox(0.00,0.00)[l]{$#1$}}
\put(20.00,10.00){\line(0,-1){5.00}}
\put(10.00,10.00){\line(0,-1){5.00}}
\put(30.00,30.00){\line(-1,0){30.00}}
\put(30.00,10.00){\line(0,1){20.00}}
\put(0.00,10.00){\line(1,0){30.00}}
\put(0.00,30.00){\line(0,-1){20.00}}
\end{picture}}
}

\def\twoazero#1{%
{
\unitlength=.5pt
\begin{picture}(36.00,30.00)(-3.00,0.00)
\put(20.00,25.00){\line(0,-1){5.00}}
\put(10.00,25.00){\line(0,-1){5.00}}
\put(10.00,10.00){\makebox(0.00,0.00)[l]{$#1$}}
\put(30.00,20.00){\line(-1,0){30.00}}
\put(30.00,0.00){\line(0,1){20.00}}
\put(0.00,0.00){\line(1,0){30.00}}
\put(0.00,20.00){\line(0,-1){20.00}}
\end{picture}}
}

\def\zeroathree#1{
{
\unitlength=.5pt
\begin{picture}(48.00,30.00)(-4.00,10.00)
\put(20.00,20.00){\makebox(0.00,0.00){$#1$}}
\put(0.00,30.00){\line(0,-1){20.00}}
\put(40.00,30.00){\line(-1,0){40.00}}
\put(40.00,10.00){\line(0,1){20.00}}
\put(0.00,10.00){\line(1,0){40.00}}
\put(30.00,10.00){\line(0,-1){5.00}}
\put(20.00,10.00){\line(0,-1){5.00}}
\put(10.00,10.00){\line(0,-1){5.00}}
\end{picture}}
}

\def\jednadva{{
\unitlength=.4pt
\begin{picture}(24.00,20.00)(-2.00,0.00)
\bezier{20}(10.00,10.00)(15.00,5.00)(20.00,0.00)
\bezier{20}(10.00,10.00)(5.00,5.00)(0.00,0.00)
\put(10.00,20.00){\line(0,-1){10.00}}
\end{picture}}
}

\def\fjednadva{{
\unitlength=.4pt
\begin{picture}(24.00,20.00)(-2.00,0.00)
\thicklines
\bezier{20}(10.00,10.00)(15.00,5.00)(20.00,0.00)
\bezier{20}(10.00,10.00)(5.00,5.00)(0.00,0.00)
\put(10.00,20.00){\line(0,-1){10.00}}
\end{picture}}
}

\def\jednactyri{{
\unitlength=.05pt
\begin{picture}(176.00,160.00)(-8.00,0.00)
\put(80.00,100.00){\line(0,1){60.00}}
\bezier{20}(80.00,80.00)(100.00,30.00)(110.00,0.00)
\bezier{20}(80.00,80.00)(60.00,30.00)(50.00,0.00)
\bezier{20}(80.00,80.00)(120.00,40.00)(160.00,0.00)
\bezier{20}(80.00,80.00)(40.00,40.00)(0.00,0.00)
\put(80.00,100.00){\line(0,-1){20.00}}
\end{picture}}
}

\def\ctyrijedna{{
\unitlength=.05pt
\begin{picture}(176.00,160.00)(-8.00,-160.00)
\put(80.00,-100.00){\line(0,-1){60.00}}
\bezier{20}(80.00,-80.00)(100.00,-30.00)(110.00,0.00)
\bezier{20}(80.00,-80.00)(60.00,-30.00)(50.00,0.00)
\bezier{20}(80.00,-80.00)(120.00,-40.00)(160.00,0.00)
\bezier{20}(80.00,-80.00)(40.00,-40.00)(0.00,0.00)
\put(80.00,-80.00){\line(0,-1){40.00}}
\end{picture}}
}

\def\dvajedna{{
\unitlength=.4pt
\begin{picture}(24.00,20.00)(-2.00,0.00)
\put(10.00,10.00){\line(0,-1){10.00}}
\bezier{20}(10.00,10.00)(15.00,15.00)(20.00,20.00)
\bezier{20}(0.00,20.00)(5.00,15.00)(10.00,10.00)
\end{picture}}
}

\def\fdvajedna{{
\unitlength=.4pt
\begin{picture}(24.00,20.00)(-2.00,0.00)
\thicklines
\put(10.00,10.00){\line(0,-1){10.00}}
\bezier{20}(10.00,10.00)(15.00,15.00)(20.00,20.00)
\bezier{20}(0.00,20.00)(5.00,15.00)(10.00,10.00)
\end{picture}}
}

\def\dvadva{{
\unitlength=.8pt
\begin{picture}(12.00,10.00)(-1.00,0.00)
\bezier{30}(0.00,0.00)(5.00,5.00)(10.00,10.00)
\bezier{30}(0.00,10.00)(5.00,5.00)(10.00,0.00)
\end{picture}}
}

\def\jednatri{{
\unitlength=.4pt
\begin{picture}(24.00,20.00)(-2.00,0.00)
\bezier{20}(10.00,10.00)(15.00,5.00)(20.00,0.00)
\bezier{20}(10.00,10.00)(5.00,5.00)(0.00,0.00)
\put(10.00,20.00){\line(0,-1){20.00}}
\end{picture}}
}

\def\trijedna{{
\unitlength=.4pt
\begin{picture}(24.00,20.00)(-2.00,-20.00)
\bezier{20}(10.00,-10.00)(15.00,-5.00)(20.00,0.00)
\bezier{20}(10.00,-10.00)(5.00,-5.00)(0.00,0.00)
\put(10.00,-20.00){\line(0,1){20.00}}
\end{picture}}
}

\def\dvatri{{
\unitlength=0.4pt
\begin{picture}(24.00,20.00)(-2.00,0.00)
\put(10.00,10.00){\line(0,-1){10.00}}
\bezier{30}(0.00,0.00)(10.00,10.00)(20.00,20.00)
\bezier{30}(0.00,20.00)(10.00,10.00)(20.00,0.00)
\end{picture}}
}

\def\tridva{{
\unitlength=.4pt
\begin{picture}(24.00,20.00)(-2.00,-20.00)
\put(10.00,-10.00){\line(0,1){10.00}}
\bezier{30}(0.00,0.00)(10.00,-10.00)(20.00,-20.00)
\bezier{30}(0.00,-20.00)(10.00,-10.00)(20.00,0.00)
\end{picture}}
}

\def\dvacarkatri{{
\unitlength=.2pt
\begin{picture}(40.00,50.00)(0.00,0.00)
\put(20.00,30.00){\line(0,-1){10.00}}
\put(20.00,0.00){\line(0,1){20}}
\bezier{20}(20.00,20.00)(30.00,10.00)(40.00,0.00)
\bezier{20}(20.00,20.00)(10.00,10.00)(0.00,0.00)
\bezier{20}(20.00,30.00)(30.00,40.00)(40.00,50.00)
\bezier{20}(0.00,50.00)(10.00,40.00)(20.00,30.00)
\end{picture}}
}

\def\gen#1#2{
\if #11
    \if #22 \jednadva \else \fi
\else
\fi
\if #12
    \if #22 \dvadva \else \fi
\else
\fi
\if #12
    \if #21 \dvajedna \else \fi
\else
\fi
\if #13
    \if #22 \tridva \else \fi
\else
\fi
\if #13
    \if #21 \trijedna \else \fi
\else
\fi
\if #12
    \if #23 \dvatri \else \fi
\else
\fi
\if #11
    \if #23 \jednatri \else \fi
\else
\fi
\if #11
    \if #24 \jednactyri \else \fi
\fi
\if #14
    \if #21 \ctyrijedna \else \fi
\fi
}
\def\fgen#1#2{
\if #11
    \if #22 \fjednadva \else \fi
\else
\fi
\if #12
    \if #22 \dvadva \else \fi
\else
\fi
\if #12
    \if #21 \fdvajedna \else \fi
\else
\fi
\if #13
    \if #22 \tridva \else \fi
\else
\fi
\if #13
    \if #21 \trijedna \else \fi
\else
\fi
\if #12
    \if #23 \dvatri \else \fi
\else
\fi
\if #11
    \if #23 \jednatri \else \fi
\else
\fi
\if #11
    \if #24 \jednactyri \else \fi
\fi
\if #14
    \if #21 \ctyrijedna \else \fi
\fi
}

\def\bZbbZ{
{
\unitlength=.27pt
\begin{picture}(48.00,30.00)(-4,0.00)
\bezier{34}(20.00,20.00)(30.00,10.00)(40.00,0.00)
\bezier{34}(20.00,20.00)(10.00,10.00)(0.00,0.00)
\bezier{20}(30.00,10.00)(25.00,5.00)(20.00,0.00)
\put(20.00,30.00){\line(0,-1){10.00}}
\end{picture}} 
}

\def\ZbbZb{{
\unitlength=.27pt
\begin{picture}(48.00,30.00)(-4,0.00)
\bezier{34}(20.00,20.00)(30.00,10.00)(40.00,0.00)
\bezier{34}(20.00,20.00)(10.00,10.00)(0.00,0.00)
\bezier{20}(10.00,10.00)(15.00,5.00)(20.00,0.00)
\put(20.00,30.00){\line(0,-1){10.00}}
\end{picture}}}

\def\dvabZbbZ{{
\unitlength=.2pt
\begin{picture}(48.00,40.00)(-4.00,-40.00)
\bezier{10}(30.00,-30.00)(25.00,-35.00)(20.00,-40.00)
\bezier{34}(0.00,0.00)(20.00,-20.00)(40.00,-40.00)
\bezier{34}(0.00,-40.00)(20.00,-20.00)(40.00,0.00)
\end{picture}}}

\def\dvaZbbZb{{
\unitlength=.2pt
\begin{picture}(48.00,40.00)(-4.00,-40.00)
\bezier{10}(10.00,-30.00)(15.00,-35.00)(20.00,-40.00)
\bezier{34}(0.00,0.00)(20.00,-20.00)(40.00,-40.00)
\bezier{34}(0.00,-40.00)(20.00,-20.00)(40.00,0.00)
\end{picture}}}

\def\asci{{
\unitlength=0.025em
\begin{picture}(110.00,60.00)(0.00,0.00)
\put(0.00,20.00){\line(1,0){110.00}}
\put(0.00,40.00){\line(1,0){110.00}}
\thicklines
\put(80.00,40.00){\line(0,-1){40.00}}
\put(30.00,60.00){\line(0,-1){60.00}}
\qbezier(80.00,40.00)(90.00,50.00)(100.00,60.00)
\qbezier(80.00,40.00)(70.00,50.00)(60.00,60.00)
\qbezier(30.00,20.00)(40.00,10.00)(50.00,0.00)
\qbezier(10.00,0.00)(20.00,10.00)(30.00,20.00)
\end{picture}}
}

\def\ascicut{{
\unitlength=0.025em
\begin{picture}(60.00,60.00)(0.00,0.00)
\put(0.00,40.00){\line(1,0){60}}
\thicklines
\put(30.00,60.00){\line(0,-1){60.00}}
\qbezier(30.00,20.00)(40.00,10.00)(50.00,0.00)
\qbezier(10.00,0.00)(20.00,10.00)(30.00,20.00)
\end{picture}}
}

\def\ascicutmod{{
\unitlength=0.025em
\begin{picture}(60.00,40.00)(0.00,-20.00)
\put(0.00,20.00){\line(1,0){60}}
\thicklines
\put(30.00,40.00){\line(0,-1){40.00}}
\qbezier(30.00,20.00)(40.00,10.00)(50.00,0.00)
\qbezier(10.00,0.00)(20.00,10.00)(30.00,20.00)
\end{picture}}
}

\def\asc{{
\unitlength=0.025em
\begin{picture}(150.00,60.00)(0.00,0.00)
\put(150.00,20.00){\line(-1,0){150.00}}
\put(0.00,40.00){\line(1,0){150.00}}
\thicklines
\put(120.00,40.00){\line(0,-1){40.00}}
\put(50.00,60.00){\line(0,-1){20.00}}
\qbezier(120.00,40.00)(130.00,50.00)(140.00,60.00)
\qbezier(100.00,60.00)(110.00,50.00)(120.00,40.00)
\qbezier(30.00,20.00)(40.00,10.00)(50.00,0.00)
\qbezier(50.00,40.00)(70.00,20.00)(90.00,0.00)
\qbezier(50.00,40.00)(30.00,20.00)(10.00,0.00)
\end{picture}}
}

\def\ascmodq{{
\unitlength=0.025em
\begin{picture}(150.00,60.00)(0.00,0.00)
\put(150.00,20.00){\line(-1,0){150.00}}
\put(0.00,40.00){\line(1,0){150.00}}
\thicklines
\put(120.00,40.00){\line(0,-1){40.00}}
\put(50.00,60.00){\line(0,-1){20.00}}
\qbezier(120.00,40.00)(130.00,50.00)(140.00,60.00)
\qbezier(100.00,60.00)(110.00,50.00)(120.00,40.00)
\qbezier(70.00,20.00)(60.00,10.00)(50.00,0.00)
\qbezier(50.00,40.00)(70.00,20.00)(90.00,0.00)
\qbezier(50.00,40.00)(30.00,20.00)(10.00,0.00)
\end{picture}}
}

\def\xyt{{
\unitlength=0.025em
\begin{picture}(110.00,60.00)(0.00,0.00)
\put(110.00,20.00){\line(-1,0){110.00}}
\put(0.00,40.00){\line(1,0){110.00}}
\thicklines
\put(30.00,60.00){\line(0,-1){40.00}}
\put(80.00,60.00){\line(0,-1){60.00}}
\qbezier(80.00,40.00)(90.00,50.00)(100.00,60.00)
\qbezier(60.00,60.00)(70.00,50.00)(80.00,40.00)
\qbezier(30.00,20.00)(40.00,10.00)(50.00,0.00)
\qbezier(10.00,0.00)(20.00,10.00)(30.00,20.00)
\end{picture}}
}

\def\tttt{{
\unitlength=0.025em
\begin{picture}(150.00,60.00)(0.00,0.00)
\put(150.00,20.00){\line(-1,0){150.00}}
\put(0.00,40.00){\line(1,0){150.00}}
\thicklines
\put(100.00,20.00){\line(0,-1){20.00}}
\put(30.00,60.00){\line(0,-1){40.00}}
\qbezier(80.00,40.00)(90.00,50.00)(100.00,60.00)
\qbezier(100.00,20.00)(120.00,40.00)(140.00,60.00)
\qbezier(60.00,60.00)(80.00,40.00)(100.00,20.00)
\qbezier(30.00,20.00)(40.00,10.00)(50.00,0.00)
\qbezier(10.00,0.00)(20.00,10.00)(30.00,20.00)
\end{picture}}
}

\def\ttttmod{{
\unitlength=0.025em
\begin{picture}(150.00,60.00)(0.00,0.00)
\put(150.00,20.00){\line(-1,0){150.00}}
\put(0.00,40.00){\line(1,0){150.00}}
\thicklines
\put(100.00,20.00){\line(0,-1){20.00}}
\put(30.00,60.00){\line(0,-1){40.00}}
\qbezier(120.00,40.00)(110.00,50.00)(100.00,60.00)
\qbezier(100.00,20.00)(120.00,40.00)(140.00,60.00)
\qbezier(60.00,60.00)(80.00,40.00)(100.00,20.00)
\qbezier(30.00,20.00)(40.00,10.00)(50.00,0.00)
\qbezier(10.00,0.00)(20.00,10.00)(30.00,20.00)
\end{picture}}
}

\def\uuuua{{
\unitlength=.02em
\begin{picture}(230.00,80.00)(0.00,0.00)
\put(0.00,60.00){\line(1,0){230.00}}
\put(0.00,40.00){\line(1,0){230.00}}
\put(0.00,20.00){\line(1,0){230.00}}
\thicklines
\put(180.00,40.00){\line(0,-1){40.00}}
\qbezier(180.00,40.00)(200.00,60.00)(220.00,80.00)
\qbezier(140.00,80.00)(160.00,60.00)(180.00,40.00)
\qbezier(30.00,20.00)(40.00,10.00)(50.00,0.00)
\put(70.00,60.00){\line(1,-1){60.00}}
\put(70.00,60.00){\line(0,1){20.00}}
\put(10.00,0.00){\line(1,1){60.00}}
\end{picture}}
}

\def\uuuuacut{{
\unitlength=.02em
\begin{picture}(150.00,80.00)(0.00,0.00)
\put(0.00,40.00){\line(1,0){140}}
\thicklines
\qbezier(30.00,20.00)(40.00,10.00)(50.00,0.00)
\put(70.00,60.00){\line(1,-1){60.00}}
\put(70.00,60.00){\line(0,1){20.00}}
\put(10.00,0.00){\line(1,1){60.00}}
\end{picture}}
}

\def\uuuuacutmod{{
\unitlength=.02em
\begin{picture}(150.00,80.00)(0.00,0.00)
\put(0.00,40.00){\line(1,0){140}}
\thicklines
\qbezier(102.00,22.00)(92.00,12.00)(82.00,2.00)
\put(70.00,60.00){\line(1,-1){60.00}}
\put(70.00,60.00){\line(0,1){20.00}}
\put(10.00,0.00){\line(1,1){60.00}}
\end{picture}}
}

\def\uuuuamod{{
\unitlength=.02em
\begin{picture}(230.00,80.00)(0.00,0.00)
\put(0.00,60.00){\line(1,0){230.00}}
\put(0.00,40.00){\line(1,0){230.00}}
\put(0.00,20.00){\line(1,0){230.00}}
\thicklines
\put(180.00,40.00){\line(0,-1){40.00}}
\qbezier(180.00,40.00)(200.00,60.00)(220.00,80.00)
\qbezier(140.00,80.00)(160.00,60.00)(180.00,40.00)
\qbezier(100.00,20.00)(90.00,10.00)(80.00,0.00)
\put(70.00,60.00){\line(1,-1){60.00}}
\put(70.00,60.00){\line(0,1){20.00}}
\put(10.00,0.00){\line(1,1){60.00}}
\end{picture}}
}

\def\iuert{{
\unitlength=0.02em
\begin{picture}(230.00,80.00)(0.00,0.00)
\put(0.00,60.00){\line(1,0){230.00}}
\put(0.00,40.00){\line(1,0){230.00}}
\put(0.00,20.00){\line(1,0){230.00}}
\thicklines
\qbezier(120.00,60.00)(130.00,70.00)(140.00,80.00)
\put(50.00,40.00){\line(0,1){40.00}}
\put(160.00,20.00){\line(0,-1){20.00}}
\put(160.00,20.00){\line(1,1){60.00}}
\put(100.00,80.00){\line(1,-1){60.00}}
\qbezier(50.00,40.00)(70.00,20.00)(90.00,0.00)
\qbezier(10.00,0.00)(30.00,20.00)(50.00,40.00)
\end{picture}}
}

\def\iuertmod{{
\unitlength=0.02em
\begin{picture}(230.00,80.00)(0.00,0.00)
\put(0.00,60.00){\line(1,0){230.00}}
\put(0.00,40.00){\line(1,0){230.00}}
\put(0.00,20.00){\line(1,0){230.00}}
\thicklines
\qbezier(190.00,60.00)(180.00,70.00)(170.00,80.00)
\put(50.00,40.00){\line(0,1){40.00}}
\put(160.00,20.00){\line(0,-1){20.00}}
\put(160.00,20.00){\line(1,1){60.00}}
\put(100.00,80.00){\line(1,-1){60.00}}
\qbezier(50.00,40.00)(70.00,20.00)(90.00,0.00)
\qbezier(10.00,0.00)(30.00,20.00)(50.00,40.00)
\end{picture}}
}

\def\uuuinvtot{{
\unitlength=0.02em
\begin{picture}(150.00,80.00)(-150,-80.00)
\put(0.00,-20.00){\line(-1,0){150.00}}
\put(0.00,-40.00){\line(-1,0){150.00}}
\put(0.00,-60.00){\line(-1,0){150.00}}
\thicklines
\put(-50.00,-80.00){\line(0,1){40.00}}
\put(-120.00,-60.00){\line(0,1){60.00}}
\qbezier(-120.00,-60.00)(-130.00,-70.00)(-140.00,-80.00)
\qbezier(-120.00,-60.00)(-110.00,-70.00)(-100.00,-80.00)
\qbezier(-50.00,-40.00)(-70.00,-20.00)(-90.00,0.00)
\qbezier(-30.00,-20.00)(-40.00,-10.00)(-50.00,0.00)
\qbezier(-10.00,0.00)(-30.00,-20.00)(-50.00,-40.00)
\end{picture}}
}

\def\uuuinvtotmod{{
\unitlength=0.02em
\begin{picture}(150.00,80.00)(-150,-80.00)
\put(0.00,-20.00){\line(-1,0){150.00}}
\put(0.00,-40.00){\line(-1,0){150.00}}
\put(0.00,-60.00){\line(-1,0){150.00}}
\thicklines
\put(-50.00,-80.00){\line(0,1){40.00}}
\put(-120.00,-60.00){\line(0,1){60.00}}
\qbezier(-120.00,-60.00)(-130.00,-70.00)(-140.00,-80.00)
\qbezier(-120.00,-60.00)(-110.00,-70.00)(-100.00,-80.00)
\qbezier(-50.00,-40.00)(-70.00,-20.00)(-90.00,0.00)
\qbezier(-70.00,-20.00)(-60.00,-10.00)(-50.00,0.00)
\qbezier(-10.00,0.00)(-30.00,-20.00)(-50.00,-40.00)
\end{picture}}
}

\def\uuuu{{
\unitlength=0.02em
\begin{picture}(150.00,80.00)(0.00,0.00)
\put(0.00,20.00){\line(1,0){150.00}}
\put(0.00,40.00){\line(1,0){150.00}}
\put(0.00,60.00){\line(1,0){150.00}}
\thicklines
\put(50.00,80.00){\line(0,-1){40.00}}
\put(120.00,60.00){\line(0,-1){60.00}}
\qbezier(120.00,60.00)(130.00,70.00)(140.00,80.00)
\qbezier(120.00,60.00)(110.00,70.00)(100.00,80.00)
\qbezier(50.00,40.00)(70.00,20.00)(90.00,0.00)
\qbezier(30.00,20.00)(40.00,10.00)(50.00,0.00)
\qbezier(10.00,0.00)(30.00,20.00)(50.00,40.00)
\end{picture}}
}

\def\uuuucut{{
\unitlength=0.02em
\begin{picture}(100.00,80.00)(0.00,0.00)
\put(0.00,60.00){\line(1,0){100.00}}
\thicklines
\put(50.00,80.00){\line(0,-1){40.00}}
\qbezier(50.00,40.00)(70.00,20.00)(90.00,0.00)
\qbezier(30.00,20.00)(40.00,10.00)(50.00,0.00)
\qbezier(10.00,0.00)(30.00,20.00)(50.00,40.00)
\end{picture}}
}

\def\painteduuuucut{{
\unitlength=0.02em
\begin{picture}(100.00,80.00)(0.00,0.00)
\thicklines
\put(50.00,80.00){\line(0,-1){40.00}}
\multiput(47,80.00)(1,0){7}{\line(0,-1){20}}
\put(50.00,80.00){\line(0,-1){20.00}}
\qbezier(50.00,40.00)(70.00,20.00)(90.00,0.00)
\qbezier(30.00,20.00)(40.00,10.00)(50.00,0.00)
\qbezier(10.00,0.00)(30.00,20.00)(50.00,40.00)
\end{picture}}
}

\def\painteduuuuacut{{
\unitlength=.02em
\begin{picture}(150.00,80.00)(0.00,0.00)
\thicklines
\multiput(67,80.00)(1,0){7}{\line(0,-1){20}}
\multiput(38.5,38.5)(.5,.5){8}{\qbezier(30.00,20.00)(40.00,10.00)(50.00,0.00)}
\multiput(38.5,42.5)(.5,-.5){8}{\qbezier(30.00,20.00)(20.00,10.00)(10.00,0.00)}
\qbezier(30.00,20.00)(40.00,10.00)(50.00,0.00)
\put(70.00,60.00){\line(1,-1){60.00}}
\qbezier(70,60)(80,50)(90,40)
\put(70.00,60.00){\line(0,1){20.00}}
\put(10.00,0.00){\line(1,1){60.00}}
\end{picture}}
}

\def\painteduuuuacutmmo{{
\unitlength=.015em
\begin{picture}(150.00,80.00)(0.00,0.00)
\thicklines
\multiput(67,80.00)(1,0){7}{\line(0,-1){20}}
\multiput(38.5,38.5)(.5,.5){8}{\qbezier(30.00,20.00)(40.00,10.00)(50.00,0.00)}
\multiput(38.5,42.5)(.5,-.5){8}{\qbezier(30.00,20.00)(20.00,10.00)(10.00,0.00)}
\qbezier(30.00,20.00)(40.00,10.00)(50.00,0.00)
\put(70.00,60.00){\line(1,-1){60.00}}
\qbezier(70,60)(80,50)(90,40)
\put(70.00,60.00){\line(0,1){20.00}}
\put(10.00,0.00){\line(1,1){60.00}}
\end{picture}}
}

\def\paintedffffcut{{
\unitlength=.02em
\begin{picture}(140.00,80.00)(0.00,0.00)
\thicklines
\multiput(67,80.00)(1,0){7}{\line(0,-1){20}}
\multiput(38.5,38.5)(.5,.5){8}{\qbezier(30.00,20.00)(50.00,0.00)(70.00,-20.00)}
\multiput(38.5,42.5)(.5,-.5){8}{\qbezier(30.00,20.00)(10,0)(-10,-20)}
\multiput(18.5,18.5)(.5,.5){8}{\qbezier(30.00,20.00)(40.00,10.00)(50.00,0.00)}
\qbezier(50.00,40.00)(70.00,20.00)(90.00,0.00)
\put(70.00,80.00){\line(0,-1){20.00}}
\put(70.00,60.00){\line(1,-1){60.00}}
\put(10.00,0.00){\line(1,1){60.00}}
\end{picture}}
}

\def\paintedffffcutred{{
\unitlength=.015em
\begin{picture}(140.00,80.00)(0.00,0.00)
\thicklines
\multiput(67,80.00)(1,0){7}{\line(0,-1){20}}
\multiput(38.5,38.5)(.5,.5){8}{\qbezier(30.00,20.00)(50.00,0.00)(70.00,-20.00)}
\multiput(38.5,42.5)(.5,-.5){8}{\qbezier(30.00,20.00)(10,0)(-10,-20)}
\multiput(18.5,18.5)(.5,.5){8}{\qbezier(30.00,20.00)(40.00,10.00)(50.00,0.00)}
\qbezier(50.00,40.00)(70.00,20.00)(90.00,0.00)
\put(70.00,80.00){\line(0,-1){20.00}}
\put(70.00,60.00){\line(1,-1){60.00}}
\put(10.00,0.00){\line(1,1){60.00}}
\end{picture}}
}

\def\paintedffffcutmod{{
\unitlength=.02em
\begin{picture}(140.00,80.00)(0.00,0.00)
\thicklines
\multiput(67,80.00)(1,0){7}{\line(0,-1){20}}
\multiput(38.5,38.5)(.5,.5){8}{\qbezier(30.00,20.00)(50.00,0.00)(70.00,-20.00)}
\multiput(38.5,42.5)(.5,-.5){8}{\qbezier(30.00,20.00)(10,0)(-10,-20)}
\multiput(22.5,18.5)(-.5,.5){8}{\qbezier(70.00,20.00)(60.00,10.00)(50.00,0.00)}
\qbezier(90.00,40.00)(70.00,20.00)(50.00,0.00)
\put(70.00,80.00){\line(0,-1){20.00}}
\put(70.00,60.00){\line(1,-1){60.00}}
\put(10.00,0.00){\line(1,1){60.00}}
\end{picture}}
}

\def\paintedffffcutmodred{{
\unitlength=.015em
\begin{picture}(140.00,80.00)(0.00,0.00)
\thicklines
\multiput(67,80.00)(1,0){7}{\line(0,-1){20}}
\multiput(38.5,38.5)(.5,.5){8}{\qbezier(30.00,20.00)(50.00,0.00)(70.00,-20.00)}
\multiput(38.5,42.5)(.5,-.5){8}{\qbezier(30.00,20.00)(10,0)(-10,-20)}
\multiput(22.5,18.5)(-.5,.5){8}{\qbezier(70.00,20.00)(60.00,10.00)(50.00,0.00)}
\qbezier(90.00,40.00)(70.00,20.00)(50.00,0.00)
\put(70.00,80.00){\line(0,-1){20.00}}
\put(70.00,60.00){\line(1,-1){60.00}}
\put(10.00,0.00){\line(1,1){60.00}}
\end{picture}}
}

\def\painteduuuuacutmod{{
\unitlength=.02em
\begin{picture}(150.00,80.00)(0.00,0.00)
\thicklines
\multiput(67,80.00)(1,0){7}{\line(0,-1){20}}
\multiput(38.5,38.5)(.5,.5){8}{\qbezier(30.00,20.00)(40.00,10.00)(50.00,0.00)}
\multiput(38.5,42.5)(.5,-.5){8}{\qbezier(30.00,20.00)(20.00,10.00)(10.00,0.00)}
\qbezier(102.00,22.00)(92.00,12.00)(82.00,2.00)
\put(70.00,60.00){\line(1,-1){60.00}}
\put(70.00,60.00){\line(0,1){20.00}}
\put(10.00,0.00){\line(1,1){60.00}}
\end{picture}}
}

\def\painteduuuuacutmodred{{
\unitlength=.015em
\begin{picture}(150.00,80.00)(0.00,0.00)
\thicklines
\multiput(67,80.00)(1,0){7}{\line(0,-1){20}}
\multiput(38.5,38.5)(.5,.5){8}{\qbezier(30.00,20.00)(40.00,10.00)(50.00,0.00)}
\multiput(38.5,42.5)(.5,-.5){8}{\qbezier(30.00,20.00)(20.00,10.00)(10.00,0.00)}
\qbezier(102.00,22.00)(92.00,12.00)(82.00,2.00)
\put(70.00,60.00){\line(1,-1){60.00}}
\put(70.00,60.00){\line(0,1){20.00}}
\put(10.00,0.00){\line(1,1){60.00}}
\end{picture}}
}

\def\painteduuuucutmod{{
\unitlength=0.02em
\begin{picture}(100.00,80.00)(0.00,0.00)
\thicklines
\multiput(47,80.00)(1,0){7}{\line(0,-1){20}}
\put(50.00,80.00){\line(0,-1){40.00}}
\qbezier(50.00,40.00)(70.00,20.00)(90.00,0.00)
\qbezier(60.00,22.00)(50.00,12.00)(40.00,2.00)
\qbezier(10.00,0.00)(30.00,20.00)(50.00,40.00)
\end{picture}}
}

\def\painteduuuucutcut{{
\unitlength=0.02em
\begin{picture}(100.00,80.00)(0.00,0.00)
\thicklines
\multiput(47,60.00)(1,0){7}{\line(0,-1){20}}
\put(50.00,60.00){\line(0,-1){20.00}}
\qbezier(50.00,40.00)(70.00,20.00)(90.00,0.00)
\qbezier(30.00,20.00)(40.00,10.00)(50.00,0.00)
\qbezier(10.00,0.00)(30.00,20.00)(50.00,40.00)
\end{picture}}
}

\def\painteduuuucutcutmod{{
\unitlength=0.02em
\begin{picture}(100.00,80.00)(0.00,0.00)
\thicklines
\put(-20,-20){
\multiput(67,80.00)(1,0){7}{\line(0,-1){20}}
\multiput(38.5,38.5)(.5,.5){8}{\qbezier(30.00,20.00)(40.00,10.00)(50.00,0.00)}
\multiput(38.5,42.5)(.5,-.5){8}{\qbezier(30.00,20.00)(20.00,10.00)(10.00,0.00)}
}
\put(50.00,60.00){\line(0,-1){20.00}}
\qbezier(50.00,40.00)(70.00,20.00)(90.00,0.00)
\qbezier(30.00,20.00)(40.00,10.00)(50.00,0.00)
\qbezier(10.00,0.00)(30.00,20.00)(50.00,40.00)
\end{picture}}
}

\def\paintedpoiucut{{
\unitlength=0.025em
\begin{picture}(100.00,60.00)(0.00,0.00)
\thicklines
\put(-20,-20){
\multiput(67,80.00)(.8,0){7}{\line(0,-1){40}}
\multiput(38.5,38.5)(.45,.45){8}{\qbezier(30.00,20.00)(40.00,10.00)(50.00,0.00)}
\multiput(38.7,42.3)(.45,-.45){8}{\qbezier(30.00,20.00)(20.00,10.00)(10.00,0.00)}
}
\put(50.00,60.00){\line(0,-1){60.00}}
\qbezier(50.00,40.00)(70.00,20.00)(90.00,0.00)
\qbezier(10.00,0.00)(30.00,20.00)(50.00,40.00)
\end{picture}}
}

\def\paintedpoiucutmmod{{
\unitlength=0.02em
\begin{picture}(100.00,60.00)(0.00,0.00)
\thicklines
\put(-20,-20){
\multiput(67,80.00)(.8,0){7}{\line(0,-1){40}}
\multiput(38.5,38.5)(.45,.45){8}{\qbezier(30.00,20.00)(40.00,10.00)(50.00,0.00)}
\multiput(38.7,42.3)(.45,-.45){8}{\qbezier(30.00,20.00)(20.00,10.00)(10.00,0.00)}
}
\put(50.00,60.00){\line(0,-1){60.00}}
\qbezier(50.00,40.00)(70.00,20.00)(90.00,0.00)
\qbezier(10.00,0.00)(30.00,20.00)(50.00,40.00)
\end{picture}}
}

\def\painteduuuucutcutmodmod{{
\unitlength=0.02em
\begin{picture}(100.00,80.00)(0.00,0.00)
\thicklines
\put(-20,-20){
\multiput(67,80.00)(1,0){7}{\line(0,-1){20}}
\multiput(38.5,38.5)(.5,.5){8}{\qbezier(30.00,20.00)(40.00,10.00)(50.00,0.00)}
\multiput(38.5,42.5)(.5,-.5){8}{\qbezier(30.00,20.00)(20.00,10.00)(10.00,0.00)}
}
\put(50.00,60.00){\line(0,-1){20.00}}
\qbezier(50.00,40.00)(70.00,20.00)(90.00,0.00)
\qbezier(70.00,20.00)(60.00,10.00)(50.00,0.00)
\qbezier(10.00,0.00)(30.00,20.00)(50.00,40.00)
\end{picture}}
}

\def\painteduuuucutcutmodmodmod{{
\unitlength=0.02em
\begin{picture}(100.00,80.00)(0.00,0.00)
\thicklines
\multiput(46,60.00)(1,0){8}{\line(0,-1){21}}
\put(50.00,60.00){\line(0,-1){20.00}}
\qbezier(50.00,40.00)(70.00,20.00)(90.00,0.00)
\qbezier(70.00,20.00)(60.00,10.00)(50.00,0.00)
\qbezier(10.00,0.00)(30.00,20.00)(50.00,40.00)
\end{picture}}
}
 
\def\paintedascicut{{
\unitlength=0.025em
\begin{picture}(60.00,60.00)(0.00,0.00)
\thicklines
\multiput(27,60.00)(1,0){7}{\line(0,-1){20}}
\put(30.00,60.00){\line(0,-1){60.00}}
\qbezier(30.00,20.00)(40.00,10.00)(50.00,0.00)
\qbezier(10.00,0.00)(20.00,10.00)(30.00,20.00)
\end{picture}}
}

\def\paindedascicutmod{{
\unitlength=0.025em
\begin{picture}(60.00,40.00)(0.00,-20.00)
\thicklines
\multiput(27,40.00)(1,0){7}{\line(0,-1){20}}
\put(30.00,40.00){\line(0,-1){40.00}}
\qbezier(30.00,20.00)(40.00,10.00)(50.00,0.00)
\qbezier(10.00,0.00)(20.00,10.00)(30.00,20.00)
\end{picture}}
}

\def\paindedascicutmodvetsi{{
\unitlength=0.035em
\begin{picture}(60.00,40.00)(0.00,-20.00)
\thicklines
\multiput(27.2,40.00)(.8,0){7}{\line(0,-1){20}}
\put(30.00,40.00){\line(0,-1){40.00}}
\qbezier(30.00,20.00)(40.00,10.00)(50.00,0.00)
\qbezier(10.00,0.00)(20.00,10.00)(30.00,20.00)
\end{picture}}
}

\def\uuuucutcut{{
\unitlength=0.02em
\begin{picture}(100.00,80.00)(0.00,0.00)
\put(0.00,40.00){\line(1,0){100.00}}
\thicklines
\put(50.00,60.00){\line(0,-1){20.00}}
\qbezier(50.00,40.00)(70.00,20.00)(90.00,0.00)
\qbezier(30.00,20.00)(40.00,10.00)(50.00,0.00)
\qbezier(10.00,0.00)(30.00,20.00)(50.00,40.00)
\end{picture}}
}

\def\uuuucutcutmod{{
\unitlength=0.02em
\begin{picture}(100.00,80.00)(0.00,0.00)
\put(0.00,20.00){\line(1,0){100.00}}
\thicklines
\put(50.00,60.00){\line(0,-1){20.00}}
\qbezier(50.00,40.00)(70.00,20.00)(90.00,0.00)
\qbezier(30.00,20.00)(40.00,10.00)(50.00,0.00)
\qbezier(10.00,0.00)(30.00,20.00)(50.00,40.00)
\end{picture}}
}

\def\uuuucutcutmodmod{{
\unitlength=0.02em
\begin{picture}(100.00,80.00)(0.00,0.00)
\put(0.00,20.00){\line(1,0){100.00}}
\thicklines
\put(50.00,60.00){\line(0,-1){20.00}}
\qbezier(50.00,40.00)(70.00,20.00)(90.00,0.00)
\qbezier(70.00,20.00)(60.00,10.00)(50.00,0.00)
\qbezier(10.00,0.00)(30.00,20.00)(50.00,40.00)
\end{picture}}
}

\def\uuuucutcutmodmodmod{{
\unitlength=0.02em
\begin{picture}(100.00,80.00)(0.00,0.00)
\put(0.00,40.00){\line(1,0){100.00}}
\thicklines
\put(50.00,60.00){\line(0,-1){20.00}}
\qbezier(50.00,40.00)(70.00,20.00)(90.00,0.00)
\qbezier(70.00,20.00)(60.00,10.00)(50.00,0.00)
\qbezier(10.00,0.00)(30.00,20.00)(50.00,40.00)
\end{picture}}
}

\def\uuuucutmod{{
\unitlength=0.02em
\begin{picture}(100.00,80.00)(0.00,0.00)
\put(0.00,60.00){\line(1,0){100.00}}
\thicklines
\put(50.00,80.00){\line(0,-1){40.00}}
\qbezier(50.00,40.00)(70.00,20.00)(90.00,0.00)
\qbezier(60.00,22.00)(50.00,12.00)(40.00,2.00)
\qbezier(10.00,0.00)(30.00,20.00)(50.00,40.00)
\end{picture}}
}

\def\uuuumod{{
\unitlength=0.02em
\begin{picture}(150.00,80.00)(0.00,0.00)
\put(0.00,20.00){\line(1,0){150.00}}
\put(0.00,40.00){\line(1,0){150.00}}
\put(0.00,60.00){\line(1,0){150.00}}
\thicklines
\put(50.00,80.00){\line(0,-1){40.00}}
\put(120.00,60.00){\line(0,-1){60.00}}
\qbezier(120.00,60.00)(130.00,70.00)(140.00,80.00)
\qbezier(120.00,60.00)(110.00,70.00)(100.00,80.00)
\qbezier(50.00,40.00)(70.00,20.00)(90.00,0.00)
\qbezier(70.00,20.00)(60.00,10.00)(50.00,0.00)
\qbezier(10.00,0.00)(30.00,20.00)(50.00,40.00)
\end{picture}}
}

\def\uytr{{
\unitlength=0.02em
\begin{picture}(260.00,80.00)(0.00,0.00)
\put(0.00,20.00){\line(1,0){260.00}}
\put(0.00,40.00){\line(1,0){260.00}}
\put(0.00,60.00){\line(1,0){260.00}}
\thicklines
\qbezier(180.00,40.00)(200.00,60.00)(220.00,80.00)
\put(200.00,20.00){\line(0,-1){20.00}}
\put(200.00,20.00){\line(5,6){50.00}}
\put(140.00,80.00){\line(1,-1){60.00}}
\put(70.00,60.00){\line(0,1){20.00}}
\put(70.00,60.00){\line(1,-1){60.00}}
\put(10.00,0.00){\line(1,1){60.00}}
\end{picture}}
}

\def\uytrmod{{
\unitlength=0.02em
\begin{picture}(260.00,80.00)(0.00,0.00)
\put(0.00,20.00){\line(1,0){260.00}}
\put(0.00,40.00){\line(1,0){260.00}}
\put(0.00,60.00){\line(1,0){260.00}}
\thicklines
\qbezier(220.00,40.00)(200.00,60.00)(180.00,80.00)
\put(200.00,20.00){\line(0,-1){20.00}}
\put(200.00,20.00){\line(5,6){50.00}}
\put(140.00,80.00){\line(1,-1){60.00}}
\put(70.00,60.00){\line(0,1){20.00}}
\put(70.00,60.00){\line(1,-1){60.00}}
\put(10.00,0.00){\line(1,1){60.00}}
\end{picture}}
}

\def\aa{
{
\unitlength=0.02em
\begin{picture}(160.00,80.00)(0.00,0.00)
\put(10.00,20.00){\line(1,0){140.00}}
\put(10.00,40.00){\line(1,0){140.00}}
\put(10.00,60.00){\line(1,0){140.00}}
\thicklines
\put(80.00,60.00){\line(0,1){20.00}}
\qbezier(60.00,40.00)(80.00,20.00)(100.00,0.00)
\qbezier(40.00,20.00)(50.00,10.00)(60.00,0.00)
\put(80.00,60.00){\line(1,-1){60.00}}
\put(80.00,60.00){\line(0,1){0.00}}
\put(20.00,0.00){\line(1,1){60.00}}
\end{picture}}
}

\def\aaup{
{
\unitlength=0.02em
\begin{picture}(160.00,80.00)(0.00,-80.00)
\put(10.00,-20.00){\line(1,0){140.00}}
\put(10.00,-40.00){\line(1,0){140.00}}
\put(10.00,-60.00){\line(1,0){140.00}}
\thicklines
\put(80.00,-60.00){\line(0,-1){20.00}}
\qbezier(60.00,-40.00)(80.00,-20.00)(100.00,0.00)
\qbezier(40.00,-20.00)(50.00,-10.00)(60.00,0.00)
\put(80.00,-60.00){\line(1,1){60.00}}
\put(80.00,-60.00){\line(0,-1){0.00}}
\put(20.00,0.00){\line(1,-1){60.00}}
\end{picture}}
}

\def\bb{{
\unitlength=0.02em
\begin{picture}(160.00,80.00)(0.00,0.00)
\put(10.00,20.00){\line(1,0){140.00}}
\put(10.00,40.00){\line(1,0){140.00}}
\put(10.00,60.00){\line(1,0){140.00}}
\thicklines
\put(80.00,60.00){\line(0,1){20.00}}
\put(80.00,60.00){\line(1,-1){60.00}}
\put(80.00,60.00){\line(0,1){0.00}}
\put(20.00,0.00){\line(1,1){60.00}}
\qbezier(120.00,20.00)(110.00,10.00)(100.00,0.00)
\qbezier(100.00,40.00)(80.00,20.00)(60.00,0.00)
\end{picture}}
}
\def\bbup{{
\unitlength=0.02em
\begin{picture}(160.00,80.00)(0.00,-80.00)
\put(10.00,-20.00){\line(1,0){140.00}}
\put(10.00,-40.00){\line(1,0){140.00}}
\put(10.00,-60.00){\line(1,0){140.00}}
\thicklines
\put(80.00,-60.00){\line(0,-1){20.00}}
\put(80.00,-60.00){\line(1,1){60.00}}
\put(80.00,-60.00){\line(0,-1){0.00}}
\put(20.00,0.00){\line(1,-1){60.00}}
\qbezier(120.00,-20.00)(110.00,-10.00)(100.00,0.00)
\qbezier(100.00,-40.00)(80.00,-20.00)(60.00,0.00)
\end{picture}}
}
\def\cc{
{
\unitlength=0.02em
\begin{picture}(160.00,80.00)(0.00,0.00)
\put(10.00,20.00){\line(1,0){140}}
\put(10.00,40.00){\line(1,0){140}}
\put(10.00,60.00){\line(1,0){140}}
\thicklines
\put(80.00,60.00){\line(0,1){20.00}}
\put(80.00,60.00){\line(1,-1){60.00}}
\put(80.00,60.00){\line(0,1){0.00}}
\put(20.00,0.00){\line(1,1){60.00}}
\qbezier(80.00,20.00)(90.00,10.00)(100.00,0.00)
\qbezier(100.00,40.00)(80.00,20.00)(60.00,0.00)
\end{picture}}
}
\def\ccup{
{
\unitlength=0.02em
\begin{picture}(160.00,80.00)(0.00,-80.00)
\put(10.00,-20.00){\line(1,0){140}}
\put(10.00,-40.00){\line(1,0){140}}
\put(10.00,-60.00){\line(1,0){140}}
\thicklines
\put(80.00,-60.00){\line(0,-1){20.00}}
\put(80.00,-60.00){\line(1,1){60.00}}
\put(80.00,-60.00){\line(0,-1){0.00}}
\put(20.00,0.00){\line(1,-1){60.00}}
\qbezier(80.00,-20.00)(90.00,-10.00)(100.00,0.00)
\qbezier(100.00,-40.00)(80.00,-20.00)(60.00,0.00)
\end{picture}}
}

\def\ffff{{
\unitlength=.02em
\begin{picture}(270.00,80.00)(0.00,0.00)
\put(0.00,60.00){\line(1,0){270.00}}
\put(0.00,40.00){\line(1,0){270.00}}
\put(0.00,20.00){\line(1,0){270.00}}
\thicklines
\qbezier(50.00,40.00)(70.00,20.00)(90.00,0.00)
\put(200.00,20.00){\line(0,-1){20.00}}
\put(200.00,20.00){\line(1,1){60.00}}
\put(140.00,80.00){\line(1,-1){60.00}}
\put(70.00,80.00){\line(0,-1){20.00}}
\put(70.00,60.00){\line(1,-1){60.00}}
\put(10.00,0.00){\line(1,1){60.00}}
\end{picture}}
}
\def\ffffcut{{
\unitlength=.02em
\begin{picture}(140.00,80.00)(0.00,0.00)
\put(0.00,20.00){\line(1,0){140.00}}
\thicklines
\qbezier(50.00,40.00)(70.00,20.00)(90.00,0.00)
\put(70.00,80.00){\line(0,-1){20.00}}
\put(70.00,60.00){\line(1,-1){60.00}}
\put(10.00,0.00){\line(1,1){60.00}}
\end{picture}}
}

\def\ffffcutmod{{
\unitlength=.02em
\begin{picture}(140.00,80.00)(0.00,0.00)
\put(0.00,20.00){\line(1,0){140.00}}
\thicklines
\qbezier(90.00,40.00)(70.00,20.00)(50.00,0.00)
\put(70.00,80.00){\line(0,-1){20.00}}
\put(70.00,60.00){\line(1,-1){60.00}}
\put(10.00,0.00){\line(1,1){60.00}}
\end{picture}}
}

\def\ffffmod{{
\unitlength=.02em
\begin{picture}(270.00,80.00)(0.00,0.00)
\put(0.00,60.00){\line(1,0){270.00}}
\put(0.00,40.00){\line(1,0){270.00}}
\put(0.00,20.00){\line(1,0){270.00}}
\thicklines
\qbezier(90.00,40.00)(70.00,20.00)(50.00,0.00)
\put(200.00,20.00){\line(0,-1){20.00}}
\put(200.00,20.00){\line(1,1){60.00}}
\put(140.00,80.00){\line(1,-1){60.00}}
\put(70.00,80.00){\line(0,-1){20.00}}
\put(70.00,60.00){\line(1,-1){60.00}}
\put(10.00,0.00){\line(1,1){60.00}}
\end{picture}}
}

\def\ee{
{
\unitlength=0.02em
\begin{picture}(160.00,80.00)(0.00,0.00)
\put(10.00,40.00){\line(1,0){140}}
\put(10.00,60.00){\line(1,0){140}}
\put(10.00,20.00){\line(1,0){140}}
\thicklines
\put(80.00,60.00){\line(0,1){20.00}}
\put(80.00,60.00){\line(1,-1){60.00}}
\put(20.00,0.00){\line(1,1){60.00}}
\qbezier(80.00,20.00)(70.00,30.00)(60.00,40.00)
\qbezier(80.00,20.00)(70.00,10.00)(60.00,0.00)
\qbezier(80.00,20.00)(90.00,10.00)(100.00,0.00)
\end{picture}}
}

\def\eeup{
{
\unitlength=0.02em
\begin{picture}(160.00,80.00)(0.00,-80.00)
\put(10.00,-40.00){\line(1,0){140}}
\put(10.00,-60.00){\line(1,0){140}}
\put(10.00,-20.00){\line(1,0){140}}
\thicklines
\put(80.00,-60.00){\line(0,-1){20.00}}
\put(80.00,-60.00){\line(1,1){60.00}}
\put(20.00,0.00){\line(1,-1){60.00}}
\qbezier(80.00,-20.00)(70.00,-30.00)(60.00,-40.00)
\qbezier(80.00,-20.00)(70.00,-10.00)(60.00,0.00)
\qbezier(80.00,-20.00)(90.00,-10.00)(100.00,0.00)
\end{picture}}
}

\def\ff{
{
\unitlength=0.02em
\begin{picture}(160.00,80.00)(0.00,0.00)
\put(10.00,20.00){\line(1,0){140}}
\put(10.00,40.00){\line(1,0){140}}
\put(10.00,60.00){\line(1,0){140}}
\thicklines
\put(80.00,60.00){\line(0,1){20.00}}
\put(80.00,60.00){\line(1,-1){60.00}}
\put(20.00,0.00){\line(1,1){60.00}}
\qbezier(100.00,40.00)(90.00,20.00)(80.00,0.00)
\qbezier(40.00,20.00)(50.00,10.00)(60.00,0.00)
\end{picture}}
}

\def\ffinv{
{
\unitlength=0.02em
\begin{picture}(160.00,80.00)(0.00,-80.00)
\put(10.00,-20.00){\line(1,0){140}}
\put(10.00,-40.00){\line(1,0){140}}
\put(10.00,-60.00){\line(1,0){140}}
\thicklines
\put(80.00,-60.00){\line(0,-1){20.00}}
\put(80.00,-60.00){\line(1,1){60.00}}
\put(20.00,0.00){\line(1,-1){60.00}}
\qbezier(100.00,-40.00)(90.00,-20.00)(80.00,0.00)
\qbezier(40.00,-20.00)(50.00,-10.00)(60.00,0.00)
\end{picture}}
}

\def\gg{
{
\unitlength=0.02em
\begin{picture}(160.00,80.00)(0.00,0.00)
\put(10.00,20.00){\line(1,0){140}}
\put(10.00,40.00){\line(1,0){140}}
\put(10.00,60.00){\line(1,0){140}}
\thicklines
\put(80.00,60.00){\line(0,1){20.00}}
\put(80.00,60.00){\line(1,-1){60.00}}
\put(80.00,60.00){\line(0,1){0.00}}
\put(20.00,0.00){\line(1,1){60.00}}
\qbezier(60.00,40.00)(70.00,20.00)(80.00,0.00)
\qbezier(120.00,20.00)(110.00,10.00)(100.00,0.00)
\end{picture}}
}

\def\gginv{
{
\unitlength=0.02em
\begin{picture}(160.00,80.00)(0.00,-80.00)
\put(10.00,-20.00){\line(1,0){140}}
\put(10.00,-40.00){\line(1,0){140}}
\put(10.00,-60.00){\line(1,0){140}}
\thicklines
\put(80.00,-60.00){\line(0,-1){20.00}}
\put(80.00,-60.00){\line(1,1){60.00}}
\put(80.00,-60.00){\line(0,-1){0.00}}
\put(20.00,0.00){\line(1,-1){60.00}}
\qbezier(60.00,-40.00)(70.00,-20.00)(80.00,0.00)
\qbezier(120.00,-20.00)(110.00,-10.00)(100.00,0.00)
\end{picture}}
}

\def\hh{{
\unitlength=0.025em
\begin{picture}(120.00,60.00)(0.00,0.00)
\put(10.00,20.00){\line(1,0){100}}
\put(10.00,40.00){\line(1,0){100}}
\thicklines
\put(60.00,60.00){\line(0,-1){20.00}}
\qbezier(60.00,40.00)(80.00,20.00)(100.00,0.00)
\qbezier(40.00,20.00)(50.00,10.00)(60.00,0.00)
\qbezier(40.00,20.00)(40.00,10.00)(40.00,0.00)
\qbezier(60.00,40.00)(40.00,20.00)(20.00,0.00)
\end{picture}}
}

\def\hhinv{{
\unitlength=0.025em
\begin{picture}(120.00,60.00)(0.00,-60.00)
\put(10.00,-20.00){\line(1,0){100}}
\put(10.00,-40.00){\line(1,0){100}}
\thicklines
\put(60.00,-60.00){\line(0,1){20.00}}
\qbezier(60.00,-40.00)(80.00,-20.00)(100.00,0.00)
\qbezier(40.00,-20.00)(50.00,-10.00)(60.00,0.00)
\qbezier(40.00,-20.00)(40.00,-10.00)(40.00,0.00)
\qbezier(60.00,-40.00)(40.00,-20.00)(20.00,0.00)
\end{picture}}
}

\def\ii{{
\unitlength=0.025em
\begin{picture}(120.00,60.00)(0.00,0.00)
\put(10.00,20.00){\line(1,0){100}}
\put(10.00,40.00){\line(1,0){100}}
\thicklines
\put(60.00,60.00){\line(0,-1){20.00}}
\qbezier(60.00,40.00)(80.00,20.00)(100.00,0.00)
\qbezier(60.00,40.00)(40.00,20.00)(20.00,0.00)
\qbezier(80.00,20.00)(80.00,10.00)(80.00,0.00)
\qbezier(80.00,20.00)(70.00,10.00)(60.00,0.00)
\end{picture}}
}

\def\iiinv{{
\unitlength=0.025em
\begin{picture}(120.00,60.00)(0.00,-60.00)
\put(10.00,-20.00){\line(1,0){100}}
\put(10.00,-40.00){\line(1,0){100}}
\thicklines
\put(60.00,-60.00){\line(0,1){20.00}}
\qbezier(60.00,-40.00)(80.00,-20.00)(100.00,0.00)
\qbezier(60.00,-40.00)(40.00,-20.00)(20.00,0.00)
\qbezier(80.00,-20.00)(80.00,-10.00)(80.00,0.00)
\qbezier(80.00,-20.00)(70.00,-10.00)(60.00,0.00)
\end{picture}}
}

\def\ssss{{
\unitlength=0.025em
\begin{picture}(190.00,60.00)(0.00,0.00)
\put(0.00,20.00){\line(1,0){190.00}}
\put(0.00,40.00){\line(1,0){190.00}}
\thicklines
\put(140.00,20.00){\line(0,-1){20.00}}
\put(50.00,60.00){\line(0,-1){20.00}}
\qbezier(140.00,20.00)(160.00,40.00)(180.00,60.00)
\qbezier(100.00,60.00)(120.00,40.00)(140.00,20.00)
\qbezier(30.00,20.00)(40.00,10.00)(50.00,0.00)
\qbezier(50.00,40.00)(70.00,20.00)(90.00,0.00)
\qbezier(10.00,0.00)(30.00,20.00)(50.00,40.00)
\end{picture}}
}

\def\ssssmod{{
\unitlength=0.025em
\begin{picture}(190.00,60.00)(0.00,0.00)
\put(0.00,20.00){\line(1,0){190.00}}
\put(0.00,40.00){\line(1,0){190.00}}
\thicklines
\put(140.00,20.00){\line(0,-1){20.00}}
\put(50.00,60.00){\line(0,-1){20.00}}
\qbezier(140.00,20.00)(160.00,40.00)(180.00,60.00)
\qbezier(100.00,60.00)(120.00,40.00)(140.00,20.00)
\qbezier(70.00,20.00)(60.00,10.00)(50.00,0.00)
\qbezier(50.00,40.00)(70.00,20.00)(90.00,0.00)
\qbezier(10.00,0.00)(30.00,20.00)(50.00,40.00)
\end{picture}}
}

\def\jj{{
\unitlength=.025em
\begin{picture}(120.00,60.00)(0.00,0.00)
\put(10.00,20.00){\line(1,0){100}}
\put(10.00,40.00){\line(1,0){100}}
\thicklines
\put(60.00,60.00){\line(0,-1){20.00}}
\qbezier(60.00,40.00)(80.00,20.00)(100.00,0.00)
\qbezier(60.00,40.00)(40.00,20.00)(20.00,0.00)
\qbezier(80.00,20.00)(75.00,10.00)(70.00,0.00)
\qbezier(40.00,20.00)(45.00,10.00)(50.00,0.00)
\end{picture}}
}

\def\jjinv{{
\unitlength=.025em
\begin{picture}(120.00,60.00)(0.00,-60.00)
\put(10.00,-20.00){\line(1,0){100}}
\put(10.00,-40.00){\line(1,0){100}}
\thicklines
\put(60.00,-60.00){\line(0,1){20.00}}
\qbezier(60.00,-40.00)(80.00,-20.00)(100.00,0.00)
\qbezier(60.00,-40.00)(40.00,-20.00)(20.00,0.00)
\qbezier(80.00,-20.00)(75.00,-10.00)(70.00,0.00)
\qbezier(40.00,-20.00)(45.00,-10.00)(50.00,0.00)
\end{picture}}
}

\def\kk{{
\unitlength=0.025em
\begin{picture}(120.00,60.00)(0.00,0.00)
\put(10.00,20.00){\line(1,0){100}}
\put(10.00,40.00){\line(1,0){100}}
\thicklines
\put(60.00,60.00){\line(0,-1){20.00}}
\qbezier(60.00,40.00)(80.00,20.00)(100.00,0.00)
\qbezier(60.00,40.00)(70.00,20.00)(80.00,0.00)
\qbezier(40.00,20.00)(50.00,10.00)(60.00,0.00)
\qbezier(60.00,40.00)(40.00,20.00)(20.00,0.00)
\end{picture}}
}

\def\kkinv{{
\unitlength=0.025em
\begin{picture}(120.00,60.00)(0.00,-60.00)
\put(10.00,-20.00){\line(1,0){100}}
\put(10.00,-40.00){\line(1,0){100}}
\thicklines
\put(60.00,-60.00){\line(0,1){20.00}}
\qbezier(60.00,-40.00)(80.00,-20.00)(100.00,0.00)
\qbezier(60.00,-40.00)(70.00,-20.00)(80.00,0.00)
\qbezier(40.00,-20.00)(50.00,-10.00)(60.00,0.00)
\qbezier(60.00,-40.00)(40.00,-20.00)(20.00,0.00)
\end{picture}}
}

\def\poiu{{
\unitlength=0.025em
\begin{picture}(190.00,60.00)(0.00,0.00)
\put(0.00,20.00){\line(1,0){190.00}}
\put(0.00,40.00){\line(1,0){190.00}}
\thicklines
\put(140.00,20.00){\line(0,-1){20.00}}
\put(50.00,60.00){\line(0,-1){60.00}}
\qbezier(140.00,20.00)(160.00,40.00)(180.00,60.00)
\qbezier(100.00,60.00)(120.00,40.00)(140.00,20.00)
\qbezier(50.00,40.00)(70.00,20.00)(90.00,0.00)
\qbezier(10.00,0.00)(30.00,20.00)(50.00,40.00)
\end{picture}}
}

\def\poiucut{{
\unitlength=0.025em
\begin{picture}(100.00,60.00)(0.00,0.00)
\put(0.00,20.00){\line(1,0){100.00}}
\thicklines
\put(50.00,60.00){\line(0,-1){60.00}}
\qbezier(50.00,40.00)(70.00,20.00)(90.00,0.00)
\qbezier(10.00,0.00)(30.00,20.00)(50.00,40.00)
\end{picture}}
}

\def\krtecek{{
\unitlength=0.025em
\begin{picture}(190.00,60.00)(0.00,0.00)
\put(0.00,20.00){\line(1,0){190.00}}
\put(0.00,40.00){\line(1,0){190.00}}
\thicklines
\put(50.00,60.00){\line(0,-1){20.00}}
\qbezier(120.00,40.00)(130.00,50.00)(140.00,60.00)
\put(140.00,20.00){\line(0,-1){20.00}}
\qbezier(140.00,20.00)(160.00,40.00)(180.00,60.00)
\qbezier(100.00,60.00)(120.00,40.00)(140.00,20.00)
\qbezier(50.00,40.00)(70.00,20.00)(90.00,0.00)
\qbezier(10.00,0.00)(30.00,20.00)(50.00,40.00)
\end{picture}}
}

\def\krtecekmodb{{
\unitlength=0.025em
\begin{picture}(190.00,60.00)(0.00,0.00)
\put(0.00,20.00){\line(1,0){190.00}}
\put(0.00,40.00){\line(1,0){190.00}}
\thicklines
\put(50.00,60.00){\line(0,-1){20.00}}
\qbezier(160.00,40.00)(150.00,50.00)(140.00,60.00)
\put(140.00,20.00){\line(0,-1){20.00}}
\qbezier(140.00,20.00)(160.00,40.00)(180.00,60.00)
\qbezier(100.00,60.00)(120.00,40.00)(140.00,20.00)
\qbezier(50.00,40.00)(70.00,20.00)(90.00,0.00)
\qbezier(10.00,0.00)(30.00,20.00)(50.00,40.00)
\end{picture}}
}

\def\krtecekmod{{
\unitlength=0.025em
\begin{picture}(190.00,60.00)(0.00,0.00)
\put(0.00,20.00){\line(1,0){190.00}}
\put(0.00,40.00){\line(1,0){190.00}}
\thicklines
\put(50.00,60.00){\line(0,-1){20.00}}
\put(140.00,60.00){\line(0,-1){60.00}}
\qbezier(140.00,20.00)(160.00,40.00)(180.00,60.00)
\qbezier(100.00,60.00)(120.00,40.00)(140.00,20.00)
\qbezier(50.00,40.00)(70.00,20.00)(90.00,0.00)
\qbezier(10.00,0.00)(30.00,20.00)(50.00,40.00)
\end{picture}}
}

\def\nninv{
{
\unitlength=0.025em
\begin{picture}(120.00,60.00)(0.00,-60.00)
\put(10.00,-20.00){\line(1,0){100}}
\put(10.00,-40.00){\line(1,0){100}}
\thicklines
\put(60.00,-60.00){\line(0,1){20.00}}
\qbezier(60.00,-40.00)(80.00,-20.00)(100.00,0.00)
\qbezier(60.00,-40.00)(40.00,-20.00)(20.00,0.00)
\qbezier(60.00,-20.00)(67.50,-10.00)(75.00,0.00)
\qbezier(60.00,-20.00)(52.5,-10.00)(45.00,0.00)
\put(60.00,-40.00){\line(0,1){20}}
\end{picture}}
}

\def\nn{
{
\unitlength=0.025em
\begin{picture}(120.00,60.00)(0.00,0.00)
\put(10.00,20.00){\line(1,0){100}}
\put(10.00,40.00){\line(1,0){100}}
\thicklines
\put(60.00,60.00){\line(0,-1){20.00}}
\qbezier(60.00,40.00)(80.00,20.00)(100.00,0.00)
\qbezier(60.00,40.00)(40.00,20.00)(20.00,0.00)
\qbezier(60.00,20.00)(67.50,10.00)(75.00,0.00)
\qbezier(60.00,20.00)(52.5,10.00)(45.00,0.00)
\put(60.00,40.00){\line(0,-1){20}}
\end{picture}}
}

\def\ll{{
\unitlength=0.025em
\begin{picture}(120.00,60.00)(0.00,0.00)
\put(10.00,20.00){\line(1,0){100}}
\put(10.00,40.00){\line(1,0){100}}
\thicklines
\put(60.00,60.00){\line(0,-1){20.00}}
\qbezier(60.00,40.00)(80.00,20.00)(100.00,0.00)
\qbezier(60.00,40.00)(40.00,20.00)(20.00,0.00)
\qbezier(60.00,40.00)(50.00,20.00)(40.00,0.00)
\qbezier(80.00,20.00)(70.00,10.00)(60.00,0.00)
\end{picture}}
}

\def\llinv{{
\unitlength=0.025em
\begin{picture}(120.00,60.00)(0.00,-60.00)
\put(10.00,-20.00){\line(1,0){100}}
\put(10.00,-40.00){\line(1,0){100}}
\thicklines
\put(60.00,-60.00){\line(0,1){20.00}}
\qbezier(60.00,-40.00)(80.00,-20.00)(100.00,0.00)
\qbezier(60.00,-40.00)(40.00,-20.00)(20.00,0.00)
\qbezier(60.00,-40.00)(50.00,-20.00)(40.00,0.00)
\qbezier(80.00,-20.00)(70.00,-10.00)(60.00,0.00)
\end{picture}}
}

\def\oo{{
\unitlength=0.035em
\begin{picture}(100.00,40.00)(0.00,0.00)
\put(10.00,20.00){\line(1,0){80}}
\thicklines
\qbezier(50.00,20.00)(55.00,10.00)(60.00,0.00)
\qbezier(50.00,20.00)(65.00,10.00)(80.00,0.00)
\qbezier(50.00,20.00)(45.00,10.00)(40.00,0.00)
\qbezier(50.00,20.00)(35.00,10.00)(20.00,0.00)
\put(50.00,40.00){\line(0,-1){20.00}}
\end{picture}}
}
\def\D{\Delta}
\def\ooinv{{
\unitlength=0.035em
\begin{picture}(100.00,40.00)(0.00,-40.00)
\put(10.00,-20.00){\line(1,0){80}}
\thicklines
\qbezier(50.00,-20.00)(55.00,-10.00)(60.00,0.00)
\qbezier(50.00,-20.00)(65.00,-10.00)(80.00,0.00)
\qbezier(50.00,-20.00)(45.00,-10.00)(40.00,0.00)
\qbezier(50.00,-20.00)(35.00,-10.00)(20.00,0.00)
\put(50.00,-40.00){\line(0,1){20.00}}
\end{picture}}
}

\def\ooooa{{
\unitlength=0.035em
\begin{picture}(110.00,40.00)(0.00,0.00)
\put(0.00,20.00){\line(1,0){110.00}}
\thicklines
\put(30.00,20.00){\line(0,-1){20.00}}
\put(80.00,20.00){\line(0,-1){20.00}}
\put(30.00,40.00){\line(0,-1){20.00}}
\qbezier(80.00,20.00)(90.00,30.00)(100.00,40.00)
\qbezier(60.00,40.00)(70.00,30.00)(80.00,20.00)
\qbezier(30.00,20.00)(40.00,10.00)(50.00,0.00)
\qbezier(10.00,0.00)(20.00,10.00)(30.00,20.00)
\end{picture}}
}

\def\oooob{{
\unitlength=0.035em
\begin{picture}(110.00,40.00)(0.00,0.00)
\put(0.00,20.00){\line(1,0){110.00}}
\thicklines
\put(80.00,40.00){\line(0,-1){20.00}}
\put(80.00,20.00){\line(0,-1){20.00}}
\put(30.00,40.00){\line(0,-1){20.00}}
\qbezier(80.00,20.00)(90.00,30.00)(100.00,40.00)
\qbezier(60.00,40.00)(70.00,30.00)(80.00,20.00)
\qbezier(30.00,20.00)(40.00,10.00)(50.00,0.00)
\qbezier(10.00,0.00)(20.00,10.00)(30.00,20.00)
\end{picture}}
}

\title{Bipermutahedron and  biassociahedron}

\author{Martin Markl}

\begin{document}
\bibliographystyle{plain}

\begin{abstract}
We give a simple description of the face poset of a version of the biassociahedra
that generalizes, in a straightforward manner, the
description of the faces of the Stasheff's associahedra via planar
trees.  We believe that our description will substantially simplify
the notation of \cite{sanenlidze-umble:HHA11} making it, as well as
the related papers, more accessible.
\end{abstract}

\thanks{The author was supported by The Eduard \v Cech Institute 
                 P201/12/G028 and by RVO: 67985840.}
\address{Institute of Mathematics, Czech Academy, \v
  Zitn\'a 11, 115 67 Prague, The Czech Republic}

\address{MFF UK, Sokolovsk\'a 83, 186 75 Prague, The Czech Republic}

\maketitle

\setcounter{tocdepth}{1}
\tableofcontents

\section*{History and pitfalls}

In this introductory section we recall the history and indicate the
pitfalls of the `quest for the biassociahedron,' hoping to elucidate
the r\^ole of the present paper in this struggle. 

\noindent
{\bf History.}
Let us start by reviewing the precursor of the biassociahedron.
 J.~Stasheff in his seminal paper~\cite{stasheff:TAMS63} introduced
$A_\infty$-spaces (resp.~$A_\infty$-algebras, called also
strongly homotopy  or sh associative algebras) as
spaces (resp.~algebras) with a multiplication associative up to a
coherent system of homotopies. The central object of his approach was a cellular
operad $K = \{K_\me\}_{\me \geq 2}$ whose $\me$th piece $K_\me$ was a
convex $(\me-2)$-dimensional polytope called the Stasheff
associahedron. $A_\infty$-space was then defined as a topological space on which
the operad $K$ acted, while $A_\infty$-algebras were algebras over the
operad $\CC_*(K)$ of cellular chains on~$K$.     
Let us briefly recall the basic features of the construction
of~\cite{stasheff:TAMS63}, emphasizing the algebraic side. 
More details can be found for instance 
in~\cite[II.1.6]{markl-shnider-stasheff:book} or in the
original source~\cite{stasheff:TAMS63}.

Consider a dg-vector space $V$ with a homotopy associative 
multiplication $\mu:\otexp V2 \to
V$. This means that there is a chain homotopy $\mu_3 :
\otexp V3 \to V$ between $\mu(\mu \ot \id)$ and $\mu(\id \ot \mu)$,
where $\id$ denotes the identity endomorphism $\id : V \to V$. 
The homotopy $\mu_3$ will be symbolized by the~interval
\[
K_3:= \hskip 35pt
\unitlength=1.5pt
{
\unitlength=1.5pt
\begin{picture}(60.00,0.00)(0.00,-2.00)
\thicklines
\put(30.00,3){\makebox(0.00,0.00)[b]{$\mu_3(a,b,c)$}}
\put(65.00,00){\makebox(0.00,0.00)[l]{$a(bc)$}}
\put(-5,0.00){\makebox(0.00,0.00)[r]{$(ab)c$}}
\put(60.00,0.00){\makebox(0.00,0.00){$\bullet$}}
\put(0.00,0.00){\makebox(0.00,0.00){$\bullet$}}
\put(0.00,0.00){\line(1,0){60.00}}
\end{picture}}
\]
connecting the two possible products,  $(ab)c$ and $a(bc)$, 
of three elements $a,b,c \in V$.
We abbreviate, as usual, $(ab)c:= \mu\big(\mu(a,b),c\big) = \mu(\mu \ot \id)(a,b,c)$, \&c. 
As the next step, consider all possible products of four 
elements and organize them into the
vertices of the pentagon:
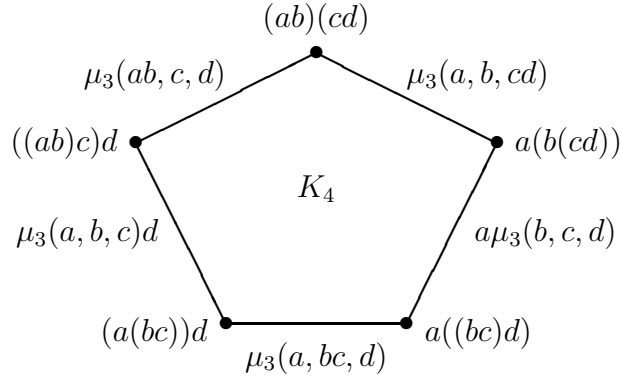
\begin{center}
\setlength{\unitlength}{1.2cm}
\thicklines
\begin{picture}(4,4)(-2,-0.5)
%
\put(-1,0){\line(1,0){2}}
\put(1,0){\line(1,2){1}}
\put(0,3){\line(2,-1){2}}
\put(-2,2){\line(2,1){2}}
\put(-2,2){\line(1,-2){1}}
%
\put(1,0){\makebox(0,0){$\bullet$}}
\put(2,2){\makebox(0,0){$\bullet$}}
\put(0,3){\makebox(0,0){$\bullet$}}
\put(-2,2){\makebox(0,0){$\bullet$}}
\put(-1,0){\makebox(0,0){$\bullet$}}
%
\put(0,3.2){\makebox(0,0)[b]{$(ab)(cd)$}}
\put(1.2,-.1){\makebox(0,0)[l]{$a\relax((bc)d\relax)$}}
\put(2.2,2){\makebox(0,0)[l]{$a\relax(b(cd)\relax)$}}
\put(-2.2,2){\makebox(0,0)[r]{$\relax((ab)c\relax)d$}}
\put(-1.2,-.1){\makebox(0,0)[r]{$\relax(a(bc)\relax)d$}}
%
\put(0,-.2){\makebox(0,0)[t]{{$\mu_3(a,bc,d)$}}}
\put(-1.75,1){\makebox(0,0)[r]{{$\mu_3(a,b,c)d$}}}
\put(1.75,1){\makebox(0,0)[l]{{$a\mu_3(b,c,d)$}}}
\put(-1,2.75){\makebox(0,0)[r]{{$\mu_3(ab,c,d)$}}}
\put(1,2.75){\makebox(0,0)[l]{{$\mu_3(a,b,cd)$}}}
\put(0,1.5){\makebox(0,0){{$K_4$}}}
\end{picture}
\end{center}
The products labeling adjacent vertices are homotopic and we labelled
the edges by the corresponding homotopies. Observe that all these 
homotopies are constructed using $\mu_3$ and the multiplication~$\mu_2$.

Next, we require the homotopy for the associativity to be {\em coherent\/},
by which we mean that the pentagon $K_4$ can be `filled' with a higher homotopy
$\mu_4 : \otexp V4 \to V$ whose differential equals the sum (with
appropriate signs) of the homotopies labelling the edges. 
This process can be continued, giving
rise to a sequence $K = \{K_\me\}_{\me \geq 2}$ of the {\em Stasheff
  associahedra\/}. It turns out that $K$ is a polyhedral
operad. An $A_\infty$-algebra is then an algebra over the operad $\CC_*(K)$ of
cellular chains on $K$.

Much later there appeared another, purely algebraic, way to introduce
$A_\infty$-algebras. As proved in~\cite{markl:zebrulka}, the operad
$\Associative$ for associative algebras admits a
unique, up to isomorphism, minimal cofibrant model
$\Ainftyop$
which turns out to be isomorphic to the operad $\CC_*(K)$. We may thus
as well say that $A_\infty$-algebras are algebras over the minimal
model of $\Associative$. 
Finally, one can describe $A_\infty$-algebras explicitly, 
as a structure with operations $\mu_\me :
\otexp V{\me} \to V$, $\me \geq 2$, satisfying a very explicit infinite set
of axioms, see~\cite[page~294]{stasheff:TAMS63}.
In the case of $A_\infty$-algebras thus {\em topology\/}, represented by the
associahedron, preceded {\em algebra\/}.

There were similar attempts to find a suitable notion of
$A_\infty$-bialgebras,\footnote{Other possible names are
  $B_\infty$-algebras or strongly homotopy bialgebras.} that is,
structures whose multiplication  and comultiplication are compatible
and (co)associative up to a~system of coherent homotopies. The
motivation for such a quest was, besides restless nature of human
mind, homotopy invariance and the related transfer properties which these
structures should posses. For instance, given a (strict) bialgebra $H$,
each dg-vector space quasi-isomorphic to the underlying dg-vector
space of $H$ ought to have an induced $A_\infty$-bialgebra structure.

Here {\em algebra\/} by far preceded {\em topology\/}.
The existence of a minimal 
model $\Binftyop$  for the PROP $B$ governing
bialgebras\footnote{PROPs generalize operads. We briefly recall them in
  the Appendix.}  was proved 
in~\cite{markl:ba}. According to
general philosophy~\cite{markl:ha}, $A_\infty$-bialgebras defined as algebras over
$\Binftyop$ are homotopy invariant concepts. Moreover, it
follows from the 
description of $\Binftyop$ given in~\cite{markl:ba} that an
$A_\infty$-bialgebra defined in this way
has operations $\mu^\en_\me : \otexp V\me \to \otexp V\en$, $m,n \in \N$,
$(m,n) \not= (1,1)$, but axioms as explicit as the ones for $A_\infty$-algebras
were given only for $m+n \leq 6$.

\noindent
{\bf Pitfalls.}  It is clearly desirable  to have some
polyhedral PROP $\KK =
\{\KK^\en_\me\}$  playing the same r\^ole for
$A_\infty$-bialgebras as the Stasheff's operad plays for
$A_\infty$-algebras. By this we mean
that $\Binftyop$ should be isomorphic to the PROP of cellular chains of $\KK$, so the 
differential in $\Binftyop$ and therefore also the axioms of
$A_\infty$-bialgebras would be encoded in the combinatorics of $\KK$. To
see where the
pitfalls are hidden, we try to mimic the inductive construction of
the associahedra in the context of bialgebras.
 
The first step is obvious. Assume we have a dg-vector space $V$ with a
multiplication $\mu : \otexp V2 \to V$ and a comultiplication $\Delta
: V \to \otexp V2$ such that $\mu$ is associative up to a~homotopy
$\mu_3^1 : \otexp V3 \to V$ symbolized by the interval
\[
K^1_3:= \hskip 35pt
\unitlength=1.5pt
{
\unitlength=1.5pt
\begin{picture}(60.00,10.00)(0.00,-2.00)
\thicklines
\put(30.00,3){\makebox(0.00,0.00)[b]{$\mu^2_3(a,b,c)$}}
\put(65.00,00){\makebox(0.00,0.00)[l]{$a(bc),$}}
\put(-5,0.00){\makebox(0.00,0.00)[r]{$(ab)c$}}
\put(60.00,0.00){\makebox(0.00,0.00){$\bullet$}}
\put(0.00,0.00){\makebox(0.00,0.00){$\bullet$}}
\put(0.00,0.00){\line(1,0){60.00}}
\end{picture}}
\]
$\mu$ and $\D$ are compatible up to a
homotopy $\mu^2_2 : \otexp V2 \to \otexp V2$ symbolized by
\[ 
K^2_2:=
\hskip 45pt
\unitlength=1.5pt
{
\unitlength=1.5pt
\begin{picture}(60.00,10.00)(0.00,-2.00)
\thicklines
\put(30.00,3){\makebox(0.00,0.00)[b]{$\mu^2_2(a,b)$}}
\put(65.00,00){\makebox(0.00,0.00)[l]{$\D(a)\D(b),$}}
\put(-5,0.00){\makebox(0.00,0.00)[r]{$\D(ab)$}}
\put(60.00,0.00){\makebox(0.00,0.00){$\bullet$}}
\put(0.00,0.00){\makebox(0.00,0.00){$\bullet$}}
\put(0.00,0.00){\line(1,0){60.00}}
\end{picture}}
\]
and $\D$ is coassociative up to a homotopy $\mu_1^3 : V \to \otexp V3$
depicted as
\[
K^3_1:=\hskip 40pt
\hskip 35pt
\unitlength=1.5pt
{
\unitlength=1.5pt
\begin{picture}(60.00,10)(0.00,-2.00)
\thicklines
\put(30.00,3){\makebox(0.00,0.00)[b]{$\mu^3_1(a)$}}
\put(65.00,00){\makebox(0.00,0.00)[l]{$(\id \ot \D)\D(a)$.}}
\put(-5,0.00){\makebox(0.00,0.00)[r]{$(\D \ot \id)\D(a)$}}
\put(60.00,0.00){\makebox(0.00,0.00){$\bullet$}}
\put(0.00,0.00){\makebox(0.00,0.00){$\bullet$}}
\put(0.00,0.00){\line(1,0){60.00}}
\end{picture}}
\]

Let us take all elements of $\otexp V2$ constructed out of three
elements of $V$ using $\D$ 
and the multiplication  
on the tensor powers of $V $ induced in the standard manner by $\mu$. 
Let us call such elements {\em algebraic\/}.
There are six of them, labelling the vertices
of a hexagon:
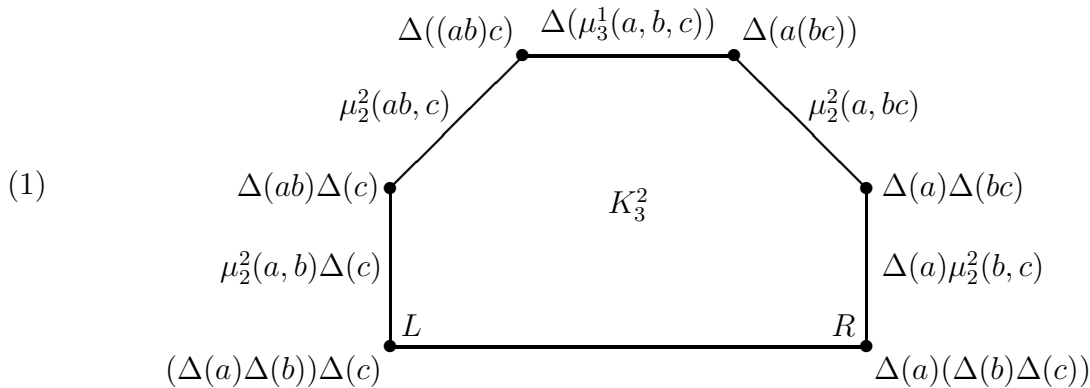
\begin{equation}
\label{Kdy_zas_s_Jarunkou?}
\raisebox{-2.7cm}{}
{
\unitlength=1.000000pt
\begin{picture}(180.00,72)(0.00,58)
\thicklines
\put(90.00,116){\makebox(0.00,0.00)[b]{$\D(\mu^1_3(a,b,c))$}}
\put(158,85){\makebox(0.00,0.00)[bl]{$\mu^2_2(a,bc)$}}
\put(186,30){\makebox(0.00,0.00)[l]{$\D(a)\mu^2_2(b,c)$}}
\put(-3,30){\makebox(0.00,0.00)[r]{$\mu^2_2(a,b)\D(c)$}}
\put(23,85){\makebox(0.00,0.00)[br]{$\mu^2_2(ab,c)$}}
%
\put(90.00,60.00){\makebox(0.00,0.00)[t]{$K^2_3$}}
\put(-3,-3){\makebox(0.00,0.00)[rt]{$(\D(a)\D(b))\D(c)$}}
\put(-4,60.00){\makebox(0.00,0.00)[r]{$\D(ab)\D(c)$}}
\put(47,113){\makebox(0.00,0.00)[br]{$\D((ab)c)$}}
\put(133.00,113.00){\makebox(0.00,0.00)[bl]{$\D(a(bc))$}}
\put(186,60.00){\makebox(0.00,0.00)[l]{$\D(a)\D(bc)$}}
\put(183,-3){\makebox(0.00,0.00)[tl]{$\D(a)(\D(b)\D(c))$}}
\put(4.00,4.00){\makebox(0.00,0.00)[bl]{$L$}}
\put(176,4.00){\makebox(0.00,0.00)[br]{$R$}}
\put(0.00,60.00){\makebox(0.00,0.00){$\bullet$}}
\put(0.00,0.00){\makebox(0.00,0.00){$\bullet$}}
\put(180.00,0.00){\makebox(0.00,0.00){$\bullet$}}
\put(180.00,60.00){\makebox(0.00,0.00){$\bullet$}}
\put(130.00,110.00){\makebox(0.00,0.00){$\bullet$}}
\put(50.00,110.00){\makebox(0.00,0.00){$\bullet$}}
\put(180.00,0.00){\line(-1,0){180.00}}
\put(180.00,60.00){\line(0,-1){60.00}}
\put(130.00,110.00){\line(1,-1){50.00}}
\put(50.00,110.00){\line(1,0){80.00}}
\put(0.00,60.00){\line(1,1){50.00}}
\put(0.00,0.00){\line(0,1){60.00}}
\end{picture}}
\end{equation}
All products labelling adjacent vertices except the two bottom ones
are homotopic via
an `algebraic' homotopy, i.e.~a homotopy constructed using
$\D$, $\mu^1_3$, $\mu^2_2$, and the multiplication induced by $\mu$ on
the powers of $V$. 

Let us inspect the vertices $L$ and $R$.  The `obvious' candidate
$\mu^1_3\big(\D(a),\D(b),\D(c)\big)$ for
the~connecting homotopy 
does not have any meaning. The labels of these vertices are, however,  still 
homotopic but in an unexpected manner.
For $a,b,c \in V$ define $X(a,b,c) \in \otexp V2$~by
\[
X(a,b,c):=
\big(\mu(\id \ot \mu) \ot \mu(\mu \ot \id)\big) \sigma\Fuk32 
\big(\D(a) \ot \D(b) \ot \D(c)\big)
\]
where $\sigma\Fuk32 : \otexp V6\to \otexp V6$ is the permutation
acting on $v_1,\ldots,v_6 \in V$ as
\[
\sigma\Fuk32 (v_1 \ot v_2 \ot v_3 \ot v_4 \ot v_5 \ot v_6):= 
 (v_1 \ot v_3 \ot v_5 \ot v_2 \ot v_4 \ot v_6).
\]
Similarly, put 
\[
Y(a,b,c):=
\big(\mu(\mu \ot \id) \ot \mu(\id \ot \mu)\big) \sigma\Fuk32 
\big(\D(a) \ot \D(b) \ot \D(c)\big).
\]
Define furthermore the homotopies $H_l,H_r,G_l,G_r : \otexp V3 \to
\otexp V2$ by the formulas
\begin{align*}
H_l(a,b,c) :&=    \big(\mu^1_3 \ot \mu(\mu \ot \id)\big) \sigma\Fuk32 
\big(\D(a) \ot \D(b) \ot \D(c)\big),
\\
H_r(a,b,c) :&= \big(\mu(\id \ot \mu) \ot \mu^1_3\big) \sigma\Fuk32 
\big(\D(a) \ot \D(b) \ot \D(c)\big),
\\
G_l(a,b,c): &= \big(\mu(\mu \ot \id) \ot \mu^1_3\big) \sigma\Fuk32 
\big(\D(a) \ot \D(b) \ot \D(c)\big), \ \mbox { and }
\\
G_r(a,b,c): &= \big(\mu^1_3\ot \mu(\id \ot \mu)\big) \sigma\Fuk32 
\big(\D(a) \ot \D(b) \ot \D(c)\big).
\end{align*}
Observing that
\begin{align*}
\big(\D(a)\D(b)\big)\D(c) &= \big(\mu(\mu \ot \id) \ot\mu(\mu \ot \id)\big) \sigma\Fuk32 
\big(\D(a) \ot \D(b) \ot \D(c)\big), \mbox { and }
\\
\D(a)\big(\D(b)\D(c)\big) &= \big(\mu(\id \ot \mu)\ot\mu(\id \ot\mu)\big) \sigma\Fuk32 
\big(\D(a) \ot \D(b) \ot \D(c)\big),
\end{align*}
we see the following composite chain of homotopies
\[
\raisebox{-.3cm}{}
{
\unitlength=2pt
\begin{picture}(130.00,9)(0.00,0.00)
\thicklines
\put(95.00,2.00){\makebox(0.00,0.00)[b]{$H_r(a,b,c)$}}
\put(25.00,2.00){\makebox(0.00,0.00)[b]{$H_l(a,b,c)$}}
\put(123.00,0){\makebox(0.00,0.00)[l]{$\D(a)(\D(b)\D(c))$}}
\put(60.00,2.00){\makebox(0.00,0.00)[b]{$X(a,b,c)$}}
\put(-3,0){\makebox(0.00,0.00)[r]{$(\D(a)\D(b))\D(c)$}}
\put(0.00,0.00){\line(1,0){120}}
\put(120.00,0.00){\makebox(0.00,0.00){$\bullet$}}
\put(60.00,0.00){\makebox(0.00,0.00){$\bullet$}}
\put(0.00,0.00){\makebox(0.00,0.00){$\bullet$}}
\end{picture}}
\]
and also 
\[
\raisebox{-.3cm}{}
{
\unitlength=2pt
\begin{picture}(130.00,9)(0.00,0.00)
\thicklines
\put(95.00,2.00){\makebox(0.00,0.00)[b]{$G_r(a,b,c)$}}
\put(25.00,2.00){\makebox(0.00,0.00)[b]{$G_l(a,b,c)$}}
\put(123.00,0){\makebox(0.00,0.00)[l]{\hphantom.$\D(a)(\D(b)\D(c))$.}}
\put(60.00,2.00){\makebox(0.00,0.00)[b]{$Y(a,b,c)$}}
\put(-3,0){\makebox(0.00,0.00)[r]{$(\D(a)\D(b))\D(c)$}}
\put(0.00,0.00){\line(1,0){120}}
\put(120.00,0.00){\makebox(0.00,0.00){$\bullet$}}
\put(60.00,0.00){\makebox(0.00,0.00){$\bullet$}}
\put(0.00,0.00){\makebox(0.00,0.00){$\bullet$}}
\end{picture}}.
\]

To proceed as in the case of the associahedron, we need to subdivide
the bottom edge of the hexagon $K^2_3$ in~(\ref{Kdy_zas_s_Jarunkou?}) 
and consider the heptagon $\KK^2_3$
\begin{equation}
\raisebox{-2.7cm}{}
\label{Jarka_zas_vazne_nemocna?}
{
\unitlength=1.000000pt
\begin{picture}(180.00,72)(0.00,58)
\thicklines
\put(4.00,4.00){\makebox(0.00,0.00)[bl]{$L$}}
\put(176,4.00){\makebox(0.00,0.00)[br]{$R$}}
\put(145,-6){\makebox(0.00,0.00)[t]{$H_r(a,b,c)$}}
\put(35,-6){\makebox(0.00,0.00)[t]{$H_l(a,b,c)$}}
\put(90.00,116){\makebox(0.00,0.00)[b]{$\D(\mu^1_3(a,b,c))$}}
\put(158,85){\makebox(0.00,0.00)[bl]{$\mu^2_2(a,bc)$}}
\put(186,30){\makebox(0.00,0.00)[l]{$\D(a)\mu^2_2(b,c)$}}
\put(-3,30){\makebox(0.00,0.00)[r]{$\mu^2_2(a,b)\D(c)$}}
\put(23,85){\makebox(0.00,0.00)[br]{$\mu^2_2(ab,c)$}}
\put(90.00,-3.00){\makebox(0.00,0.00)[t]{$X(a,b,c)$}}
\put(90.00,60.00){\makebox(0.00,0.00)[t]{$\KK^2_3$}}
\put(-3,-3){\makebox(0.00,0.00)[rt]{$(\D(a)\D(b))\D(c)$}}
\put(-4,60.00){\makebox(0.00,0.00)[r]{$\D(ab)\D(c)$}}
\put(47,113){\makebox(0.00,0.00)[br]{$\D((ab)c)$}}
\put(133.00,113.00){\makebox(0.00,0.00)[bl]{$\D(a(bc))$}}
\put(186,60.00){\makebox(0.00,0.00)[l]{$\D(a)\D(bc)$}}
\put(183,-3){\makebox(0.00,0.00)[tl]{$\D(a)(\D(b)\D(c))$}}
\put(90.00,0.00){\makebox(0.00,0.00){$\bullet$}}
\put(0.00,60.00){\makebox(0.00,0.00){$\bullet$}}
\put(0.00,0.00){\makebox(0.00,0.00){$\bullet$}}
\put(180.00,0.00){\makebox(0.00,0.00){$\bullet$}}
\put(180.00,60.00){\makebox(0.00,0.00){$\bullet$}}
\put(130.00,110.00){\makebox(0.00,0.00){$\bullet$}}
\put(50.00,110.00){\makebox(0.00,0.00){$\bullet$}}
\put(180.00,0.00){\line(-1,0){180.00}}
\put(180.00,60.00){\line(0,-1){60.00}}
\put(130.00,110.00){\line(1,-1){50.00}}
\put(50.00,110.00){\line(1,0){80.00}}
\put(0.00,60.00){\line(1,1){50.00}}
\put(0.00,0.00){\line(0,1){60.00}}
\end{picture}}
\end{equation}
Observe that the subdivision and therefore also $\KK^2_3$ 
is {\em not unique\/}, we could as well take
$Y,G_l,G_r$ instead of $X,H_l,H_r$. Notice also that {\em neither\/} the
expressions $X$, $Y$ {\em nor\/} the homotopies $H_l,H_r,G_l,G_r$ are
{\em algebraic\/}.

\noindent 
{\bf Two types of biassociahedra.}
We can already glimpse the following pattern. There naturally appear 
polytopes $K^\en_\me$, $m,n \in \N$,
such that $K^1_\me$ and $K^\me_1$ are isomorphic to Stasheff's
associahedron $K_\me$.
We will also see that $K^2_\me$ is
isomorphic to the multiplihedron $J_\me$. 
 We call these polytopes the {\em
step-one biassociahedra\/}. 
In this paper we give a simple and clean description of their face
posets.

To continue as in the case of $A_\infty$-algebras, 
one however needs to subdivide some faces of $K^\en_\me$; and example of
such a subdivision is the heptagon $\KK^2_3$
in~(\ref{Jarka_zas_vazne_nemocna?}) subdividing the hexagon $K^2_3$. 
The subdivisions must be compatible so that the result will be a
cellular PROP $\KK = \{\KK^\en_\me\}$. 
Its associated cellular chain complex is moreover required to be isomorphic
to the minimal model $\Binftyop$ of the bialgebra PROP. 
We call these polyhedra {\em step-two biassociahedra\/}.

The polytopes $\KK^\en_\me$ were, for $m+n \leq 6$, constructed
in~\cite{markl:ba}. In higher dimensions, the issue of the
compatibility of the subdivisions arises.
In~\cite{sanenlidze-umble:HHA11}, a construction of the step-two
biassociahedra was proposed, but we admit that we were not able to
verify it. By our opinion, a reasonably simple construction of the
polyhedral PROP $\KK$ or at least a convincing proof of its existence
still remains a challenge.

We think that a necessary starting point to address the above problems is
a suitable notation. In this paper we give  a simple
description of the face poset of the step-one biassociahedron 
$K^\en_\me$ that generalize the classical
description of the Stasheff associahedron in terms of planar directed
trees. We will also give a
`coordinate-free' characterization of $K^\en_\me$
which shows that it is not a human
invention but has existed since the beginning of time.  
As a~by-product of our approach, it will be obvious that $K^2_\me$ is
isomorphic to the multiplihedron $J_\me$, for each~\hbox{$\me \geq 2$}.

\noindent 
{\bf Notation and terminology.} 
Some low-dimensional examples  of the step-two biassociahedra appeared
for the first time, without explicit name, in \cite{markl:ba}; they
were denoted $B^\en_\me$ there. The word {\em biassociahedron\/} was used
by S.~Saneblidze and R.~Umble,  see
e.g.~\cite{sanenlidze-umble:HHA11}, 
referring to what we called above the step-two
biassociahedron; they denoted it $\KK_{m,n}$. 
Step-one biassociahedra can also be, without explicit name,
found in \cite{sanenlidze-umble:HHA11}; they were denoted $K_{m,n}$ there. 
Whenever we mention the
biassociahedron in this paper, we always mean the 
{\em step-one\/} biassociahedron which we denote by $K^\en_\me$.

\noindent
{\bf Acknowledgment.} I would like to express my gratitude to Samson
Saneblidze, Jim Stasheff, Ron Umble and the anonymous referee 
for reading the manuscript and offering
helpful remarks and suggestions. I also enjoyed the wonderful
atmosphere in the Max-Planck Institut
f\"ur Matematik in Bonn 
during the period when the first draft of this paper was
completed.

\section*{Main results}

Let us recall some standard
facts~\cite{markl-shnider-stasheff:book,tonks97}.  The {\em
permutahedron\/}\footnote{Sometimes also spelled
permut\underline{o}hedron.} $P_{\me-1}$ is, for $\me \geq 2$, a convex
polytope whose poset of faces $\EP_{\me-1}$ is isomorphic to the set
$\lT_\me$ of planar directed trees with levels and $\me$ leaves, with the
partial order generated by identifying adjacent levels. The
permutahedron $P_{\me-1}$ can be realized
as the convex hull of the vectors obtained by permuting the
coordinates of $(1,\ldots,\me-1) \in {\mathbb R}^{\me-1}$, its vertices
correspond to elements of the symmetric group $\Sigma_{\me-1}$.  The
face poset $\EK_\me$ of the Stasheff's {\em associahedron\/} $K_\me$ is
the set of directed planar trees (no levels) $T_\me$ with $\me$ leaves;
the partial order is given by contracting the internal edges.  The
obvious epimorphism $\varpi_\me : lT_\me \epi T_\me$ erasing the levels
induces the {\em Tonks projection} $\Tonks : \EP_{\me-1} \epi \EK_\me$ of
the face posets, see~\cite{tonks97} for details.

There is a conceptual explanation of Tonks' projection that uses
a natural map 
\begin{equation}
\label{eq:1}
\varpi_\me : \lT_\me \to \free(\xi_2,\xi_3,\ldots)(\me)
\end{equation} 
to the arity $\me$ piece of the free non-$\Sigma$ operad 
\cite[Section~4]{markl:handbook} 
generated by the operations $\xi_2,\xi_3,\ldots$ of arities
$2,3,\ldots$, respectively. The map $\varpi_\me$, roughly speaking,
replaces the vertices of a tree $T \in \lT_\me$ with the generators of
$\free(\xi_2,\xi_3,\ldots)$ whose arities equal the number of
inputs of the corresponding vertex, and then composes these generators 
using $T$ as the composition scheme, see~\S\ref{ted_umiram_na_plice}.

It is almost evident that the set
$T_\me$ is isomorphic to the image of $\varpi_\me$. In other words, the
face poset $\EK_\me$ of the associahedron $K_\me$ can be {\em defined\/} as the
quotient of $lT_{\me}$ modulo the equivalence that identifies elements
having the same image under $\varpi_\me$, with the induced \po. 
Tonks' projection then appears as the
epimorphism in the factorization
\[
{
\unitlength=0.09em
\begin{picture}(90.00,30.00)(0.00,0.00)
\put(-7.00,13.00){\makebox(0.00,0.00)[lb]{$\lT_\me\stackrel{\Tonks}{\epi} T_\me \hookrightarrow
\free(\xi_2,\xi_3,\ldots)$.}} 
\put(45.00,3.00){\makebox(0.00,0.00)[b]{\scriptsize $\varpi_\me$}} 
\put(90.00,0.00){\vector(0,1){10.00}}
\put(0.00,0.00){\line(1,0){90.00}}
\put(0.00,10.00){\line(0,-1){10.00}}
\end{picture}}
\] 

The aim of this note is to define in the same manner the poset
$\EK^\en_\me$ of faces of the step-one {\em biassociahedron\/} $K^\en_\me$ constructed
in \cite[\S 9.5]{sanenlidze-umble:HHA11}.\footnote{In
\cite{sanenlidze-umble:HHA11} it was denoted $K_{\me,\en}$.}  To
this end, we introduce in~\ref{vezmou_mi_kus_plice?}, for each $m,n
\geq 1$, the set $\lT^\en_\me$ of {\em complementary pairs\/} of directed
planar trees with levels. The set $\lT^\en_\me$ has a \po\ $<$ similar to
that of $\lT_\me$.  The poset $\EP^\en_\me = (\lT^\en_\me,<)$ provides a natural
indexing of the face poset of the {\em bipermutahedron\/} $P^\en_\me$ of
\cite{sanenlidze-umble:HHA11}.\footnote{Denoted $P_{m-1,n-1}$
in \cite{sanenlidze-umble:HHA11}.}  It turns out that the
posets $\EP^\en_\me$ and $\EP_{m+n-1}$ are isomorphic; we give a simple
proof of this fact in
Section~\ref{radeji_k_tem_doktorum_nechodit}. Comparing it with the
proof of the analogous~\cite[Proposition~6]{sanenlidze-umble:HHA11}
convincingly demonstrates the naturality of our language of
complementary pairs.

As the next step, we describe, for each $m,n \geq 1$, a natural map
\[
\varpi^\en_\me : \lT^\en_\me \to \free\big(\xi^b_a \, | \, a,b \geq 1,\ 
(a,b) \not = (1,1)\big).
\]   
The object in the right hand side is the free PROP
\cite[Section~8]{markl:handbook} generated by operations $\xi^b_a$ of
biarity $(b,a)$, i.e.~with $b$ outputs and $a$ inputs.  We then {\em define\/}
the face poset $\EK^\en_\me$ of the biassociahedron as $\lT^\en_\me$ modulo
the relation that identifies the complementary pairs of trees having
the same image under $\varpi^\en_\me$, with the induced \po.  We prove
that $\EK^\en_\me$ is isomorphic to the poset of complementary pairs of
planar directed trees $\zT^\en_\me$ {\em with zones\/}, see Definition~A
on page \pageref{sec:compl-pairs-with-2}.  It will be obvious that
this is the simplest possible description of the poset of faces of the
Saneblidze-Umble biassociahedron that generalizes the standard
description of the face poset of the Stasheff's associahedron.

In the last section, we analyze in detail the special case of
$\EK^2_\me$ when complementary pairs of trees with zones can
equivalently be described as trees with a {\em diaphragm\/}. Using
this description we
prove that $\EK^2_\me$ is isomorphic to the face poset of the
multiplihedron $J_\me$. 
Necessary facts about PROPs and calculus of fractions are recalled in
the Appendix.

The main definitions are {\bf Definition~A} on page
\pageref{sec:compl-pairs-with-2} and {\bf Definition~B\/} on
page~\pageref{sec:compl-pairs-with-1}.  The main result is {\bf
Theorem~C} on page~\pageref{zitra_do_Lnar?} and the main application
is {\bf Proposition~D} on page~\pageref{dnes_s_TK_Holter}.

\section{Trees with levels and (bi)permutahedra}
\label{radeji_k_tem_doktorum_nechodit}

\subsection{Up- and down-rooted trees}
Let us start by recalling some classical material from 
\cite[II.1.5]{markl-shnider-stasheff:book}. A~planar directed
(also called rooted) tree is a planar tree with a specified leg called
the {\em root\/}. The remaining legs are the {\em leaves\/}.
We will tacitly assume that all vertices have at least three adjacent
edges.

We will distinguish between {\em up-rooted\/} trees
whose all edges different from the root are oriented {\em
towards\/} the root while, in {\em down-rooted\/} trees, we orient the
edges to point {\em away\/} from the root.
The set of vertices $\vert(T)$ of an up- or down-rooted
tree $T$ is partially ordered by requiring that $u < v$ if and only if
there exist an oriented edge path starting at $u$ and ending at $v$.   

An up-rooted planar tree with $h$ {\em levels\/}, $h \geq 1$, is an
up-rooted planar tree $U$ with vertices placed at $h$ horizontal
lines numbered $1,\ldots,h$ from the top down.  More formally, an
up-rooted tree with $h$ levels is an up-rooted planar tree $U$
together with a strictly order-preserving {\em level function\/} $\ell
: \vert(U) \to \{1,\ldots,h\}$.\footnote{Strictly order-preserving
means that $v' < v''$ implies $\ell(v') < \ell(v'')$.} We tacitly
assume that the level function is an epimorphism (no `dummy' levels with
no vertices); if this is not the case, we say that $\ell$ is {\em
degenerate\/}. We believe that Figure~\ref{fig:1} clarifies these
notions. Since we numbered the level lines from the top down, saying
that vertex $v'$ lies {\em above\/} $v''$ means $\ell(v') < \ell(v'')$.

\begin{figure}
{
\unitlength=1.000000pt
\begin{picture}(200.00,120.00)(0.00,0.00)
\put(0,2.5){
\put(0.00,30.00){\makebox(0.00,0.00)[t]{\scriptsize $4$}}
\put(0.00,50.00){\makebox(0.00,0.00)[t]{\scriptsize $3$}}
\put(0.00,70.00){\makebox(0.00,0.00)[t]{\scriptsize $2$}}
\put(0.00,90.00){\makebox(0.00,0.00)[t]{\scriptsize $1$}}}
\put(10.00,30.00){\line(1,0){190.00}}
\put(10.00,50.00){\line(1,0){190.00}}
\put(10.00,70.00){\line(1,0){190.00}}
\put(10.00,90.00){\line(1,0){190.00}}
\put(100.00,0.00){\makebox(0.00,0.00)[t]{\scriptsize the leaves}}
\put(100.00,110.00){\makebox(0.00,0.00)[l]{\scriptsize root}}
\put(80.00,30.00){\makebox(0.00,0.00){$\bullet$}}
\put(130.00,30.00){\makebox(0.00,0.00){$\bullet$}}
\put(103.00,70.00){\makebox(0.00,0.00){$\bullet$}}
\put(60.00,50.00){\makebox(0.00,0.00){$\bullet$}}
\put(90.00,90.00){\makebox(0.00,0.00){$\bullet$}}
\thicklines
\put(103.00,70.00){\line(1,-6){10.00}}
\put(103.00,70.00){\vector(0,1){0.00}}
\put(183.00,10){\vector(-4,3){80.00}}
\put(90.00,90.00){\vector(0,1){30.00}}
\put(140.00,10.00){\vector(-1,2){10.00}}
\put(130.00,10.00){\vector(0,1){20.00}}
\put(70.00,10.00){\vector(1,2){10.00}}
\put(160.00,10.00){\vector(-3,2){30.00}}
\put(120.00,10.00){\vector(1,2){10.00}}
\put(130.00,30.00){\vector(-2,3){40.00}}
\put(130.00,30.00){\vector(-2,3){27.00}}
\put(60.00,10.00){\vector(0,1){40.00}}
\put(100.00,10.00){\vector(-1,1){40.00}}
\put(100.00,10.00){\vector(-1,1){20.00}}
\put(30.00,10.00){\vector(3,4){60.00}}
\put(30.00,10.00){\vector(3,4){30.00}}
\end{picture}}
\caption{\label{fig:1}An up-rooted 
tree with $10$ leaves and $5$ vertices aligned at $4$ levels, none of
  them `dummy.' The edges are oriented towards the root.}
\end{figure}
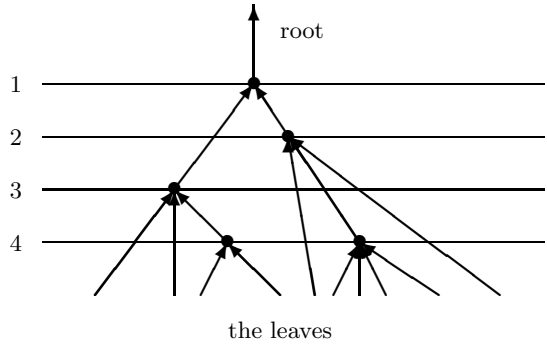

Let us denote by $\lT_\me$ the set of up-rooted trees with levels and
$\me$ leaves. It forms a~category whose morphisms $(U',\ell') \to
(U'',\ell'')$ are couples $(\phi,\hat \phi)$ consisting of a
map\footnote{By a map of trees we understand a sequence of
  contractions of internal edges. In particular, the root and
  leaves are fixed.} of
up-rooted planar trees $\phi : U' \to U''$ and of an order-preserving map
$\hat \phi : \{1,\ldots,h'\} \to \{1,\ldots,h''\}$ forming the commutative
diagram
\[
{
\unitlength=1.000000pt
\begin{picture}(110.00,60.00)(0,-5.00)
\thicklines
\put(100.00,30.00){\vector(0,-1){20.00}}
\put(10.00,30.00){\vector(0,-1){20.00}}
\put(38.00,0.00){\vector(1,0){30}}
\put(35.00,40.00){\vector(1,0){40.00}}
\put(104.00,20.00){\makebox(0.00,0.00)[l]{\scriptsize $\ell''$}}
\put(6.00,20.00){\makebox(0.00,0.00)[r]{\scriptsize $\ell'$}}
\put(50.00,5.00){\makebox(0.00,0.00)[b]{\scriptsize $\hat \phi$}}
\put(52.00,45.00){\makebox(0.00,0.00)[b]{\scriptsize $\phi$}}
\put(100.00,0.00){\makebox(0.00,0.00){$\{1,\ldots,h''\}$}}
\put(10.00,0.00){\makebox(0.00,0.00){$\{1,\ldots,h'\}$}}
\put(100.00,40.00){\makebox(0.00,0.00){$\vert(U'')$}}
\put(10.00,40.00){\makebox(0.00,0.00){$\vert(U')$}}
\end{picture}}
\]
in which we denote $\phi: U' \to U''$ and the induced map $\vert(U)
\to \vert(U'')$ by the same symbol.

We say that $(U',\ell') < (U'',\ell'')$ if there exists a morphism
$(U',\ell') \to (U'',\ell'')$.  Since all endomorphisms in $\lT_\me$
are the identities, the relation $<$ is a
partial order. The set $\lT_\me$ with this partial order is isomorphic to
the face poset $\EP_{\me-1}$ of the permutahedron $P_{\me-1}$. This result
is so classical that we will not give full details here,
see~\cite{tonks97}.\footnote{We however recall the correspondence between
  trees with levels and ordered partitions in~\S\ref{zas_na_mne_neco_leze}.}
 For $\me=4$, this
isomorphism is illustrated in Figure~\ref{fig:ab}.  

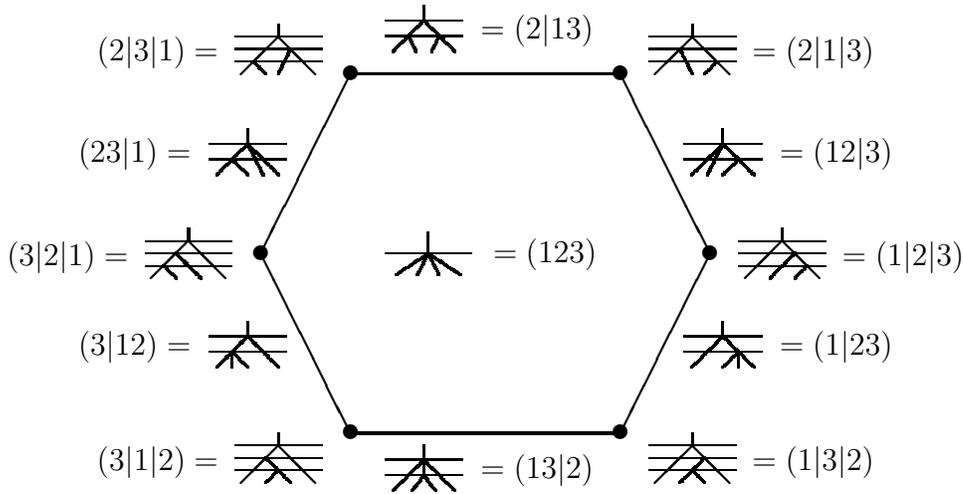
\begin{figure}
{
\unitlength=.85pt
\begin{picture}(200.00,210.00)(0.00,-10.00)
\put(185.00,60){\makebox(0.00,0.00)[tl]{$\ii\raisebox{5pt}{$\ =(1|23)$}$}}
\put(100.00,5.00){\makebox(0.00,0.00)[t]{$\nn\raisebox{5pt}{$\ =(13|2)$}$}}
\put(15.00,60.00){\makebox(0.00,0.00)[rt]{$\raisebox{5pt}{$(3|12)= \ $}\hh$}}
\put(15.00,125.00){\makebox(0.00,0.00)[br]{$\raisebox{5pt}{$(23|1)= \ $}\kk$}}
\put(100.00,180.00){\makebox(0.00,0.00)[b]{$\jj\raisebox{5pt}{$\ =(2|13)$}$}}
\put(185.00,125.00){\makebox(0.00,0.00)[lb]{$\ll\raisebox{5pt}{$\ =(12|3)$}$}}
\put(170.00,10.00){\makebox(0.00,0.00)[lt]{$\cc\raisebox{5pt}{$\ =(1|3|2)$}$}}
\put(30.00,10.00){\makebox(0.00,0.00)[tr]{$\raisebox{5pt}{$(3|1|2)=\ $}\ee$}}
\put(-10.00,90.00){\makebox(0.00,0.00)[r]{$\raisebox{5pt}{$(3|2|1)=\ $}\aa$}}
\put(30.00,170.00){\makebox(0.00,0.00)[rb]{$\raisebox{5pt}{$(2|3|1)=\ $} \ff$}}
\put(170.00,170.00){\makebox(0.00,0.00)[lb]{$\gg\raisebox{5pt}{$\ =(2|1|3)$}$}}
\put(210.00,90.00){\makebox(0.00,0.00)[l]{$\bb\raisebox{5pt}{$\ =(1|2|3)$}$}}
\put(100.00,90.00){\makebox(0.00,0.00){$\oo\raisebox{5pt}{$\ =(123)$}$}}
\put(160.00,10.00){\makebox(0.00,0.00){\large$\bullet$}}
\put(40.00,10.00){\makebox(0.00,0.00){\large$\bullet$}}
\put(0.00,90.00){\makebox(0.00,0.00){\large$\bullet$}}
\put(40.00,170.00){\makebox(0.00,0.00){\large$\bullet$}}
\put(160.00,170.00){\makebox(0.00,0.00){\large$\bullet$}}
\put(200.00,90.00){\makebox(0.00,0.00){\large$\bullet$}}
\thicklines
\put(160.00,10.00){\line(-1,0){10.00}}
\put(200.00,90.00){\line(-1,-2){40.00}}
\put(160.00,170.00){\line(1,-2){40.00}}
\put(150.00,170.00){\line(1,0){10.00}}
\put(40.00,170.00){\line(1,0){110.00}}
\put(40.00,10.00){\line(1,0){110.00}}
\put(0.00,90.00){\line(1,-2){40.00}}
\put(40.00,170.00){\line(-1,-2){40.00}}
\end{picture}}
\caption{
\label{fig:ab}
The faces of the permutahedron $P_3$ indexed by the set $\lT_4$. The
ordered partitions of $\{1,2,3\}$ 
corresponding to the faces under the
correspondence described in \S\ref{zas_na_mne_neco_leze} are also
shown.}
\end{figure}

The definition of the poset $(\lT^\en,<)$ of down-rooted trees with
levels and $\en$ leaves is similar.  We will also need the {\em
exceptional tree\/} $\except$ with one edge and no vertices. We define
$\lT_1 = \lT^1 := \{\except\}$.

\subsection{Complementary pairs}
\label{vezmou_mi_kus_plice?}
For $m,n \geq 1$ we denote by $\lT^\en_\me$ the set of all triples
$(U,D,\ell)$ consisting of an up-rooted planar tree $U$ with
$\me$ leaves and a down-rooted planar tree $D$ with $\en$ leaves, equipped
with a strictly order-preserving level function
\[
\ell :
\vert(U) \cup \vert(D)   \epi \{1,\ldots,h\}.
\]
Observe that if we denote $\ell_u:= \ell|_{\vert(U)}$ (resp.~$\ell_d:= 
\ell|_{\vert(D)}$), then $(U,\ell_u)$ (resp.~$(D,\ell_d)$) is an up-rooted
(resp.~down-rooted) rooted tree with possibly degenerate level function.

We call the objects $(U,D,\ell)$ 
the {\em complementary pairs\/} of trees
with levels.  Figure~\ref{figure35} explains the terminology, concrete
examples can be found in Figure~\ref{celeste}.
\begin{figure}
{
\unitlength=1.000000pt
\begin{picture}(180.00,100.00)(0.00,0.00)
\put(130.00,100.00){\makebox(0.00,0.00){\scriptsize $\en$ leaves}}
\put(60.00,0.00){\makebox(0.00,0.00){\scriptsize $\me$ leaves}}
\put(130.00,90.00){\makebox(0.00,0.00)[t]{$\cdots$}}
\put(60.00,10.00){\makebox(0.00,0.00)[b]{$\cdots$}}
\put(130.00,50.00){\makebox(0.00,0.00){$D$}}
\put(60.00,50.00){\makebox(0.00,0.00){$U$}}
\put(0.00,20.00){\makebox(0.00,0.00){\scriptsize $4$}}
\put(0.00,40.00){\makebox(0.00,0.00){\scriptsize $3$}}
\put(0.00,60.00){\makebox(0.00,0.00){\scriptsize $2$}}
\put(0.00,80.00){\makebox(0.00,0.00){\scriptsize $1$}}
\put(10.00,20.00){\line(1,0){170.00}}
\put(10.00,40.00){\line(1,0){170.00}}
\put(10.00,60.00){\line(1,0){170.00}}
\put(10.00,80.00){\line(1,0){170.00}}
\thicklines
\put(160.00,90.00){\line(0,-1){10.00}}
\put(110.00,90.00){\line(0,-1){10.00}}
\put(100.00,90.00){\line(0,-1){10.00}}
\put(90.00,10.00){\line(0,1){10.00}}
\put(40.00,10.00){\line(0,1){10.00}}
\put(30.00,10.00){\line(0,1){10.00}}
\put(130.00,0.00){\vector(0,1){20.00}}
\put(60.00,80.00){\vector(0,1){20.00}}
\put(160.00,80.00){\line(-1,0){60.00}}
\put(130.00,20.00){\line(1,2){30.00}}
\put(100.00,80.00){\line(1,-2){30.00}}
\put(90.00,20.00){\line(-1,2){30.00}}
\put(30.00,20.00){\line(1,0){60.00}}
\put(60.00,80.00){\line(-1,-2){30.00}}
\end{picture}}
\caption{\label{figure35}A schematic picture of a couple $(U,D) \in
  \lT^\en_\me$ of complementary trees with $4$ levels.}
\end{figure}
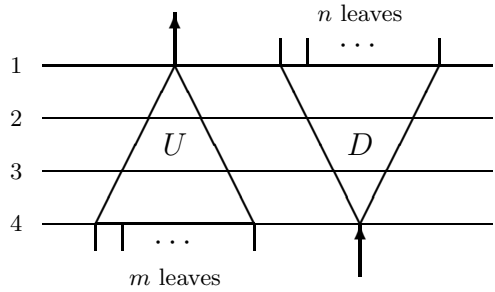
When the level function is clear from the context, we drop it from
the notation.  The set $\lT^\en_\me$ forms a category in the same way as $\lT_\me$. A
morphism
\[
\phi: (U',D',\ell') \to (U'',D'',\ell'')
\]
is a triple $(\phi_u,\phi_d,\hat \phi)$ consisting of a morphism
$\phi_u : U' \to U''$ (resp.~$\phi_d : D' \to D''$) of up-rooted
(resp.~down-rooted) planar trees and of an order-preserving map $\hat
\phi : \{1,\ldots,h'\} \to \{1,\ldots,h''\}$ such that the diagram
\begin{equation}
\label{eq:2}
\raisebox{-2em}{}
{
\unitlength=1.000000pt
\begin{picture}(110.00,35.00)(0,15)
\thicklines
\put(125.00,30.00){\vector(0,-1){20.00}}
\put(-20,30.00){\vector(0,-1){20.00}}
\put(10,0.00){\vector(1,0){81}}
\put(35.00,40.00){\vector(1,0){35.00}}
\put(132,20.00){\makebox(0.00,0.00)[l]{\scriptsize $\ell''$}}
\put(-26,20.00){\makebox(0.00,0.00)[r]{\scriptsize $\ell'$}}
\put(50.00,5.00){\makebox(0.00,0.00)[b]{\scriptsize $\hat \phi$}}
\put(52.00,45.00){\makebox(0.00,0.00)[b]{\scriptsize $\phi_u \cup \phi_d$}}
\put(125,0.00){\makebox(0.00,0.00){$\{1,\ldots,h''\}$}}
\put(-20,0.00){\makebox(0.00,0.00){$\{1,\ldots,h'\}$}}
\put(125,40.00){\makebox(0.00,0.00){$\vert(U'')\cup \vert(D'')$}}
\put(-20.00,40.00){\makebox(0.00,0.00){$\vert(U') \cup \vert(D')$}}
\end{picture}}
\end{equation} 
commutes. The \po\ of $\lT^\en_\me$, analogous to that of $\lT_\me$, is given by the
existence of a morphism in the above category.

\begin{theorem}
\label{Opet}
The posets $\EP_{m+n-2} = (\lT_{m+n-1},<)$ and $\EP^\en_\me = (\lT^\en_\me,<)$
are  naturally isomorphic for each $m,n \geq 1$.
\end{theorem}

\begin{proof}
Let us describe an isomorphisms of the underlying sets.
Since clearly $\lT^1_u \cong \lT^u_1 \cong \lT_u$ for each $u\geq 1$, it is
enough to construct, for each $\en \geq 2$, $\me \geq 1$, an isomorphism
\begin{equation}
\label{eq:3}
\lT^\en_\me \cong \lT^{\en-1}_{\me+1}.
\end{equation}
Let $X = (U,D') \in \lT^\en_\me$.  Denote by $v$ the initial vertex of
the leftmost leaf of $D'$ and $L$ its level line.  Amputation of this
leftmost leaf at $v$ gives a down-rooted tree $D$ with $\en-1$ leaves.
Now extend the up-rooted tree $U$ by grafting an up-going leaf at the rightmost
point $w$ in which the level line $L$ intersects $U$. If $w$ is a vertex,
we graft the leaf at this vertex, if $w$ is a point of 
an edge, we introduce a new vertex with two input edges.

Let $U'$ denotes this extended tree. The isomorphism~(\ref{eq:3}) is
given by the correspondence $(U,D') \mapsto (U',D)$, with the pair
$(U',D)$ equipped with the level function induced, in the obvious
manner, by the level function of $(U,D')$.  We believe that
Figure~\ref{zase_umiram} makes isomorphism~(\ref{eq:3}) obvious.  It
would, of course, be possible to define it using the formal language
of trees with level functions, but we consider the above informal,
intuitive definition more satisfactory.
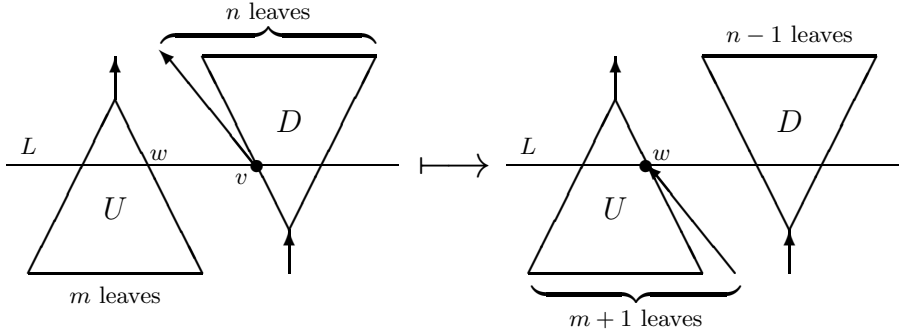
\begin{figure}
\begin{center}
{
\unitlength=.07em
\begin{picture}(410.00,120.00)(0.00,0.00)
\put(206.00,58.00){\makebox(0.00,0.00){\Large$\longmapsto$}}
\put(360.00,120.00){\makebox(0.00,0.00){\scriptsize $\en-1$ leaves}}
\put(118,115.00){\makebox(0.00,0.00)[b]{
$\overbrace{\rule{6.9em}{0mm}}^{\mbox{\scriptsize $\en$ leaves}}$}}
\put(287.00,5.00){\makebox(0.00,0.00)[t]{
$\underbrace{\rule{6.7em}{0mm}}_{\mbox{\scriptsize $\me+1$ leaves}}$}}
\put(50.00,0.00){\makebox(0.00,0.00){\scriptsize $\me$ leaves}}
\put(360.00,80.00){\makebox(0.00,0.00){$D$}}
\put(130.00,80.00){\makebox(0.00,0.00){$D$}}
\put(280.00,40.00){\makebox(0.00,0.00){$U$}}
\put(50.00,40.00){\makebox(0.00,0.00){$U$}}
\put(294,59){\makebox(0.00,0.00){$\bullet$}}
\put(297,62){\makebox(0.00,0.00)[lb]{\scriptsize $w$}}
\put(66,62){\makebox(0.00,0.00)[lb]{\scriptsize $w$}}
\put(111,56){\makebox(0.00,0.00)[tr]{\scriptsize $v$}}
\put(115,59){\makebox(0.00,0.00){$\bullet$}}
\thinlines
\put(230.00,60.00){\line(1,0){180.00}}
\put(10.00,65.00){\makebox(0.00,0.00)[b]{\scriptsize $L$}}
\put(240.00,65.00){\makebox(0.00,0.00)[b]{\scriptsize $L$}}
\put(0.00,60.00){\line(1,0){180.00}}
\thicklines
\put(335.00,10.00){\line(-4,5){40.00}}
\put(294,60){\vector(-1,1){0}}
\put(70,113.00){\vector(-1,1){0}}
\put(74.00,110.00){\line(4,-5){40.00}}
\put(390.00,110.00){\line(1,0){10.00}}
\put(400.00,110.00){\line(-1,-2){10.00}}
\put(320.00,110.00){\line(1,0){70.00}}
\put(360.00,30.00){\line(-1,2){40.00}}
\put(280.00,90){\vector(0,1){20.00}}
\put(360.00,30.00){\line(1,2){30.00}}
\put(360.00,10.00){\vector(0,1){20.00}}
\put(320.00,10.00){\line(-1,0){80.00}}
\put(280.00,90.00){\line(1,-2){40.00}}
\put(240.00,10.00){\line(1,2){40.00}}
\put(50.00,90){\vector(0,1){20.00}}
\put(170.00,110.00){\line(-1,0){80.00}}
\put(130.00,30.00){\line(1,2){40.00}}
\put(130.00,30.00){\line(-1,2){40.00}}
\put(130.00,10.00){\vector(0,1){20.00}}
\put(90.00,10.00){\line(-1,0){80.00}}
\put(50.00,90.00){\line(1,-2){40.00}}
\put(10.00,10.00){\line(1,2){40.00}}
\end{picture}}  
\end{center}
\caption{\label{zase_umiram}Isomorphism~(\ref{eq:3}).}
\end{figure}
It is simple to verify that~(\ref{eq:3}) preserves the partial orders,
giving rise to a poset isomorphism $\EP^\en_\me \cong \EP^{\en-1}_{\me+1}$,
for each $\en \geq 2$, $\me \geq 1$.
\end{proof}

\begin{Example}
The isomorphism  of Theorem~\ref{Opet} is, for $m+n = 4$,
presented in Figure~\ref{figt}.
\end{Example}

\begin{figure}
\begin{center}
{
\unitlength=0.03em
\begin{picture}(340.00,200.00)(0.00,30.00)
\put(220.00,30.00){\makebox(0.00,0.00){$>$}}
\put(225.00,110.00){\makebox(0.00,0.00){$>$}}

\put(100.00,30.00){\makebox(0.00,0.00){$<$}}
\put(95.00,110.00){\makebox(0.00,0.00){$<$}}

\put(130.00,30.00){\line(1,0){60.00}}

\thicklines
\put(160.00,190.00){\line(0,-1){30.00}}
\put(160.00,60.00){\line(0,-1){60.00}}

\put(30,0){
\thicklines
\qbezier(250.00,40.00)(260.00,50.00)(270.00,60.00)
\qbezier(270.00,20.00)(290.00,40.00)(310.00,60.00)
\qbezier(230.00,60.00)(250.00,40.00)(270.00,20.00)
\put(270.00,20.00){\line(0,-1){20.00}}
\thinlines
\put(230.00,20.00){\line(1,0){80.00}}
\put(230.00,40.00){\line(1,0){80.00}}
}
\qbezier(160.00,30.00)(170.00,40.00)(190.00,60.00)
\qbezier(130.00,60.00)(150.00,40.00)(160.00,30.00)
\thinlines
\put(130.00,190.00){\line(1,0){60.00}}
\put(130.00,110.00){\line(1,0){60.00}}
\put(-30,0){
\thicklines
\qbezier(50.00,60.00)(60.00,50.00)(70.00,40.00)
\qbezier(50.00,20.00)(70.00,40.00)(90.00,60.00)
\qbezier(10.00,60.00)(30.00,40.00)(50.00,20.00)
\put(50.00,20.00){\line(0,-1){20.00}}
\thinlines
\put(10.00,20.00){\line(1,0){80.00}}
\put(10.00,40.00){\line(1,0){80.00}}
}
\put(50,0){
\thicklines
\qbezier(300.00,100.00)(320.00,120.00)(340.00,140.00)
\qbezier(260.00,140.00)(280.00,120.00)(300.00,100.00)
\qbezier(240.00,120.00)(260.00,100.00)(280.00,80.00)
\qbezier(200.00,80.00)(220.00,100.00)(240.00,120.00)
\put(300.00,100.00){\line(0,-1){20.00}}
\put(240.00,140.00){\line(0,-1){20.00}}
\thinlines
\put(200.00,100.00){\line(1,0){140.00}}
\put(200.00,120.00){\line(1,0){140.00}}
}
\thicklines
\put(190.00,140.00){\line(-1,-1){60.00}}
\put(130.00,140.00){\line(1,-1){60.00}}

\put(-50,0){
\thicklines
\qbezier(80.00,120.00)(90.00,130.00)(100.00,140.00)
\qbezier(60.00,140.00)(70.00,130.00)(80.00,120.00)
\qbezier(20.00,100.00)(30.00,90.00)(40.00,80.00)
\qbezier(0.00,80.00)(10.00,90.00)(20.00,100.00)
\put(80.00,120.00){\line(0,-1){40.00}}
\put(20.00,140.00){\line(0,-1){40.00}}
\thinlines
\put(0.00,120.00){\line(1,0){100.00}}
\put(0.00,100.00){\line(1,0){100.00}}

}
\put(30,0){
\thicklines
\qbezier(290.00,180.00)(280.00,170.00)(270.00,160.00)
\put(270.00,220.00){\line(0,-1){20.00}}
\qbezier(270.00,200.00)(290.00,180.00)(310.00,160.00)
\qbezier(230.00,160.00)(260.00,190.00)(270.00,200.00)
\thinlines
\put(230.00,180.00){\line(1,0){80.00}}
\put(230.00,200.00){\line(1,0){80.00}}
}
\put(160.00,220.00){\line(0,-1){30.00}}
\thicklines
\qbezier(160.00,190.00)(180.00,170.00)(190.00,160.00)
\qbezier(130.00,160.00)(150.00,180.00)(160.00,190.00)
\put(100.00,190.00){\makebox(0.00,0.00){$<$}}
\put(220.00,190.00){\makebox(0.00,0.00){$>$}}
\put(-30,0){
\thinlines
\put(10.00,180.00){\line(1,0){90.00}}
\put(10.00,200.00){\line(1,0){90.00}}
\thicklines
\put(50.00,220.00){\line(0,-1){20.00}}
\qbezier(30.00,180.00)(40.00,170.00)(50.00,160.00)
\qbezier(50.00,200.00)(70.00,180.00)(90.00,160.00)
\qbezier(10.00,160.00)(30.00,180.00)(50.00,200.00)
}
\put(-70,190.00){\makebox(0.00,0.00)[r]{$\EP^1_3:$}}
\put(-70,110.00){\makebox(0.00,0.00)[r]{$\EP^2_2:$}}
\put(-70,30.00){\makebox(0.00,0.00)[r]{$\EP^3_1:$}}
\end{picture}}
\end{center}
\caption{\label{figt}The isomorphic posets $\EP_2 = \EP^1_3$, 
$\EP^2_2$ and $\EP^3_1$.}
\end{figure}

\begin{Example}
Isomorphism~(\ref{eq:3}) of complementary pairs is, for $m+n=5$,
shown in the table of Figure~\ref{celeste}.
\begin{figure}
\begin{center}
\fbox{Le banquet c\'eleste}
\\
\[
\def\arraystretch{2}
\begin{array}{|c|c|c|c||c|c|c|}
\hline
\mbox {$\lT_4 = \lT^1_4$}&\hbox{$\lT^2_3$}&\hbox{$\lT^3_2$}&
\mbox {$\lT^4 = \lT_1^4$}&\mbox
      {$\varpi(\lT^2_3)$}&\mbox{$\pi(\lT^2_3)$}&\mbox{$\pT_3$}

\cr
\hline
\rule{0pt}{2em}
\aa & \uuuu & \uuuinvtot & \bbup &
\raiseB{
\def\arraystretch{0}
\begin{array}{cc}
\gen21\\  \ZbbZb
\end{array}}
&\uuuucut&\painteduuuucut
\cr
\ff&\uuuua&\uuuinvtotmod&\ccup &\raiseB{\frac{\gen12\gen12}{\dvojiteypsilon\gen21}}&\uuuuacut&\painteduuuuacut
\cr
\gg&\ffff&\iuert&\eeup&\raiseB{\frac{\ZbbZb \ZbbZb}{\raisebox{.8em}{}\gen21 \gen21 \gen21}}&\ffffcut&\paintedffffcut
\cr
\bb&\ffffmod&\uytr&\aaup&\raiseB{\frac{\bZbbZ \bZbbZ}{\raisebox{.8em}{}\gen21 \gen21 \gen21}}&\ffffcutmod&\paintedffffcutmod
\cr
\cc&\uuuuamod&\uytrmod&\ffinv&\raiseB{\frac{\gen12 \gen12}{\gen21
    \dvojiteypsilon}}&\uuuuacutmod &\painteduuuuacutmod 
\cr
\ee&\uuuumod&\iuertmod&\gginv&\raiseB{
\def\arraystretch{0}
\begin{array}{cc}
\gen21\\  \bZbbZ
\end{array}}&\uuuucutmod&\painteduuuucutmod
\cr
\hline
\rule{0pt}{2em}
\kk&\asc&\xyt&\iiinv&\raiseB{\dvaZbbZb}&\uuuucutcut&\painteduuuucutcut
\cr
\jj&\ssss&\tttt&\nninv&\raiseB{\frac{\gen12 \gen12}{\raisebox{.8em}{}\dvadva \gen21}}&\uuuucutcutmod&\painteduuuucutcutmod
\cr
\ll&\poiu&\krtecek&\hhinv&\raiseB{\frac{\gen13 \hskip
    .5em\gen13}{\raisebox{.8em}{}\gen21\gen21\gen21}}&\poiucut&\paintedpoiucut
\cr
\ii&\ssssmod&\krtecekmod&\kkinv&\raiseB{\frac{\gen12 \gen12}{\raisebox{.8em}{}\gen21 \dvadva}}&\uuuucutcutmodmod&\painteduuuucutcutmodmod
\cr
\nn&\ascmodq&\krtecekmodb&\jjinv&\raiseB{\dvabZbbZ}
&\uuuucutcutmodmodmod&\painteduuuucutcutmodmodmod 
\cr
\hh&\asci&\ttttmod&\llinv&\raiseB{
\def\arraystretch{0}
\begin{array}{cc}
\gen21\\  \gen13
\end{array}}&\ascicut&\paintedascicut
\cr
\hline
\rule{0pt}{2em}
\oo&\ooooa&\oooob&\ooinv&\raiseB\dvacarkatri&\ascicutmod&\paindedascicutmod
\cr
\hline
\end{array}
\]
\end{center}
\caption{\label{celeste}The isomorphic sets $\lT^\en_\me$ with $m+n = 5$,
$\varpi(\lT^2_3)$, $\zT^2_3$ and $\pT_3$. The upper section of the
  left part corresponds to the
  vertices, the middle to the edges, and the bottom row to the
  $2$-cell of the bipermutahedron $P_3 = P^1_4 
\cong P^2_3 \cong P^3_2 \cong P^4_1$.}
\end{figure}
Comparing the entries in the leftmost column with the corresponding
entries of the 4th one, we see a nontrivial isomorphism between
the posets of up- and down-rooted trees.  We suggest, as an exercise,
to decorate the faces of the permutahedron $P_4$ in
Figure~\ref{fig:ab} by the corresponding entries of the table in
Figure~\ref{celeste} on page~\pageref{celeste} to verify in this particular case that the
partial orders are indeed preserved.
\end{Example}

\subsection{Relation to the standard permutahedron.} 
\label{zas_na_mne_neco_leze}
In this subsection we recall the well-known isomorphism between the
poset $\EP_\me = (\lT_\me,<)$ of rooted planar trees with levels and the
poset of ordered decompositions of the set $\{1,\ldots,\me-1\}$ which
we denote by $\SP_\me = (\Dec_\me,<)$ (the ${\EuScript S}$ in front of
$\EP$ abbreviating the {\em standard\/} permutahedron). This
isomorphism, which forms a necessary link
to~\cite{sanenlidze-umble:HHA11} and related papers, extends to an
isomorphism between $(\EP^\en_\me,<)$ and the poset of ordered {\em
bi\/}partitions $\SP^\en_\me = (\Dec^\en_\me,<)$. As the gadgets described
here will not be used later in our note, this subsections can be
safely skipped.

We start by drawing a tree $U \in \lT_\me$, as always in this note, with
the root up, and labelling the intervals between its leaves, from the
left to the right, by \hbox{$1,\ldots,\me-1$}. We then replace the
labels by party balloons, release them and let them lift to the highest possible
level.\footnote{Alternatively, replace the labels by balls and
change the direction of gravity.}  The first set of the corresponding
partition is formed by the balloons that lifted to the root
(level~$1$), the second by the balloons that lifted to level~$2$,
\&c. For instance, to the tree $U \in \lT_8$ in Figure~\ref{fig:1bis} one
associates the ordered decomposition $(4|57|12|36) \in
\Dec_8$. Another instance of the above correspondence is shown in
Figure~\ref{fig:ab}.
\begin{figure}
{
\unitlength=1.000000pt
\begin{picture}(200.00,120.00)(0.00,0.00)
\put(0.00,30.00){\makebox(0.00,0.00)[t]{\scriptsize $4$}}
\put(0.00,50.00){\makebox(0.00,0.00)[t]{\scriptsize $3$}}
\put(0.00,70.00){\makebox(0.00,0.00)[t]{\scriptsize $2$}}
\put(0.00,90.00){\makebox(0.00,0.00)[t]{\scriptsize $1$}}
\put(10.00,30.00){\line(1,0){190.00}}
\put(10.00,50.00){\line(1,0){190.00}}
\put(10.00,70.00){\line(1,0){190.00}}
\put(10.00,90.00){\line(1,0){190.00}}
\put(45.00,10.00){\makebox(0.00,0.00)[t]{\baloon1}}
\put(65.00,10.00){\makebox(0.00,0.00)[t]{\baloon2}}
\put(85.00,10.00){\makebox(0.00,0.00)[t]{\baloon3}}
\put(105.00,10.00){\makebox(0.00,0.00)[t]{\baloon4}}
\put(122.00,10.00){\makebox(0.00,0.00)[t]{\baloon5}}
\put(142.00,10.00){\makebox(0.00,0.00)[t]{\baloon6}}
\put(172.00,10.00){\makebox(0.00,0.00)[t]{\baloon7}}
\put(80.00,30.00){\makebox(0.00,0.00){$\bullet$}}
\put(130.00,30.00){\makebox(0.00,0.00){$\bullet$}}
\put(103.00,70.00){\makebox(0.00,0.00){$\bullet$}}
\put(60.00,50.00){\makebox(0.00,0.00){$\bullet$}}
\put(90.00,90.00){\makebox(0.00,0.00){$\bullet$}}
\thicklines
\put(103.00,70.00){\line(1,-6){10.00}}
\put(103.00,70.00){\line(0,1){0.00}}
\put(183.00,10){\line(-4,3){80.00}}
\put(90.00,90.00){\line(0,1){30.00}}
\put(130.00,10.00){\line(0,1){20.00}}
\put(70.00,10.00){\line(1,2){10.00}}
\put(160.00,10.00){\line(-3,2){30.00}}
\put(130.00,30.00){\line(-2,3){40.00}}
\put(130.00,30.00){\line(-2,3){27.00}}
\put(60.00,10.00){\line(0,1){40.00}}
\put(100.00,10.00){\line(-1,1){40.00}}
\put(100.00,10.00){\line(-1,1){20.00}}
\put(30.00,10.00){\line(3,4){60.00}}
\put(30.00,10.00){\line(3,4){30.00}}
\end{picture}}
\caption{\label{fig:1bis}An up-rooted 
tree $U \in \lT_8$ with party balloons ready to take off.}
\end{figure}
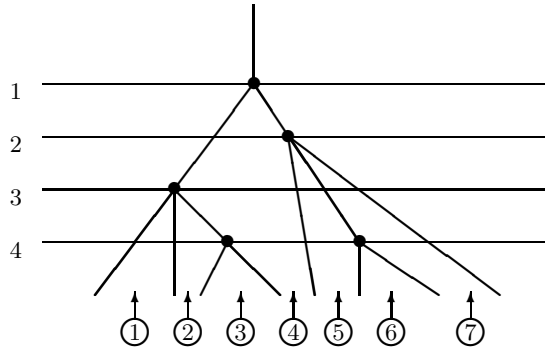

Let us denote the resulting isomorphism by $\gamma : \EP_\me \cong
\SP_\me$. Combined with the isomorphisms of Theorem~\ref{Opet},
it leads to an isomorphism (denoted by the same symbol) $\gamma : \EP^\en_\me
\cong \SP_{m+m-2}$ for each $m,n \geq 1$.

\begin{Example}
A particular instance of the above isomorphism is $\gamma: \EP^\me \cong
\SP_\me$. On the other hand the posets $\EP_\me = (\lT_\me,<)$ and $\EP^\me =
(\lT^\me,<)$ are in fact the {\em same\/}, it is only that we draw the
trees in $\EP_\me$ with the root up, and those in $\EP^\me$ 
with the root down.  One can prove
that the composite
\begin{equation}
\label{eq:18}
\tau: \SP_\me \stackrel{\gamma^{-1}}\cong \EP_\me
= \EP^\me \stackrel\gamma\cong \SP_\me
\end{equation} 
is given by reversing the order of the members of the decomposition. For
instance, 
\[
\tau(4|57|12|36) = (36|12|57|4).
\]
\end{Example}

Let us extend the above isomorphism to complementary pairs of trees. 
For $m,n \geq 1$, denote by $\SP^\en_\me = (\Dec^\en_\me,<)$ the poset of
ordered {\em bipartitions\/} of the sets
\hbox{$\{1,\ldots,\me-1\}$}, \hbox{$\{\me,\ldots,m+n-2\}$}, by which we mean~arrays
\[
\left(\slp{D_\ell}{U_1}\slp{D_{\ell-1}}{U_2} \cdots 
\slp{D_2}{U_{\ell-1}}\slp {D_1}{U_\ell} \right),
\]
where $U_1,U_2,\ldots,U_\ell$ (resp.~$D_1,D_2,\ldots,D_\ell)$ is an
ordered partition of $1,\ldots,\me-1$ (resp.~an ordered partition of
$\me,\ldots,\hbox{$m+n-2$}$). Here we allow some of the sets $U_j$
(resp.~$D_j$) to be empty, but we require $U_j \cup D_j \not= \emptyset$
for each $1\leq j \leq \ell$. 

We associate to a complementary pair $X = (U,D) \in \lT^\en_\me$ a
bipartition in $\Dec^\en_\me$ as follows. We attach to the intervals
between the leaves of the up-rooted tree $U$ balloons labeled
$1,\ldots,\me-1$, and to the intervals between the leaves of the
down-rooted tree $D$ balls labelled $\me,\ldots,m+n-2$ as
indicated in Figure~\ref{figure35bis}.
\begin{figure}
{
\unitlength=1.000000pt
\begin{picture}(180.00,100.00)(0.00,0.00)
\put(105,100.00){\makebox(0.00,0.00){$\ball \me$}}
\put(117,100.00){\makebox(0.00,0.00){$\ball {}$}}
\put(155,100.00){\makebox(0.00,0.00){$\ball {}$}}
\put(35,0.00){\makebox(0.00,0.00){$\baloon1$}}
\put(47,0.00){\makebox(0.00,0.00){$\baloon2$}}
\put(85,0.00){\makebox(0.00,0.00){$\baloon{}$}}
\put(130.00,90.00){\makebox(0.00,0.00)[t]{$\cdots$}}
\put(65.00,10.00){\makebox(0.00,0.00)[b]{$\cdots$}}
\put(130.00,50.00){\makebox(0.00,0.00){$D$}}
\put(60.00,50.00){\makebox(0.00,0.00){$U$}}
\put(0.00,20.00){\makebox(0.00,0.00){\scriptsize $4$}}
\put(0.00,40.00){\makebox(0.00,0.00){\scriptsize $3$}}
\put(0.00,60.00){\makebox(0.00,0.00){\scriptsize $2$}}
\put(0.00,80.00){\makebox(0.00,0.00){\scriptsize $1$}}
\put(10.00,20.00){\line(1,0){170.00}}
\put(10.00,40.00){\line(1,0){170.00}}
\put(10.00,60.00){\line(1,0){170.00}}
\put(10.00,80.00){\line(1,0){170.00}}
\thicklines
\put(160.00,90.00){\line(0,-1){10.00}}
\put(110.00,90.00){\line(0,-1){10.00}}
\put(100.00,90.00){\line(0,-1){10.00}}
\put(90.00,10.00){\line(0,1){10.00}}
\put(40.00,10.00){\line(0,1){10.00}}
\put(30.00,10.00){\line(0,1){10.00}}
\put(130.00,0.00){\vector(0,1){20.00}}
\put(60.00,80.00){\vector(0,1){20.00}}
\put(160.00,80.00){\line(-1,0){60.00}}
\put(130.00,20.00){\line(1,2){30.00}}
\put(100.00,80.00){\line(1,-2){30.00}}
\put(90.00,20.00){\line(-1,2){30.00}}
\put(30.00,20.00){\line(1,0){60.00}}
\put(60.00,80.00){\line(-1,-2){30.00}}
\end{picture}}
\caption{\label{figure35bis}A complementary pair decorated by $m+n-2$ balloons
  and balls. The labels $\me-1$, $\me+1$ and $n+m-2$ are
  not shown since the balloons resp.~balls are too small to hold them.}
\end{figure}
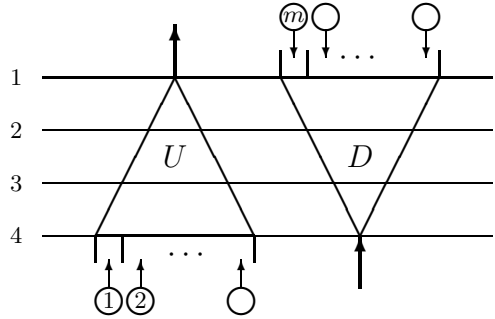
Then $U_j$ (resp.~$D_j$) is the set of balloons (resp.~balls) that lift
(resp.~fall) to level $j$, $1 \leq j \leq \ell$. Observe that the
reversing map~(\ref{eq:18}) is already built in the above assignment.

\begin{Example}
The following bipartitions correspond 
to the entries of the second line of the first part of the table in
Figure~\ref{celeste}:
\[
\left(\slp{}{2} \slp{}{3} \slp{}{1}\right),
\left(\slp{}{2} \slp{3}{} \slp{}{1}\right),
\left(\slp{\hphantom{2}}{} \slp{3}{} \slp{2}{1}\right)
\ \mbox { and } \
\left(\slp{1}{} \slp{3}{} \slp{2}{}\right),
\]
while the bipartitions corresponding to the second line of the second
part are
\[
\left(\slp{}{2} \slp{}{13}\right),
\left(\slp{3}{2} \slp{}{1}\right),
\left(\slp{3}{} \slp{2}{1}\right),
\ \mbox { and } \
\left(\slp{13}{} \slp{2}{}\right).
\]
\end{Example}

As an exercise, we recommend describing the isomorphism of
Theorem~\ref{Opet} in terms of bipartitions. 
The above example serves as a clue how to do~so.

\section{Trees with zones and the biassociahedron $\EK^\en_\me$}

In this section we present our definition of the face poset $\EK^\en_\me$ of the
biassociahedron. Let us recall the classical associahedron first.

\subsection{The associahedron $K_\me$}
\label{ted_umiram_na_plice}
As in~(\ref{eq:1}), denote by $\free(\xi_2,\xi_3,\ldots)$ the free
non-$\Sigma$ operad in the monoidal category of sets, generated by the
operations of $\xi_2,\xi_3,\ldots$ of arities $2,3,\ldots$,
respectively. Its component of arity $n$
consists of (up-rooted) planar rooted trees with vertices having at
least $2$ inputs \cite[Section~4]{markl:handbook}. 
We can therefore define the map~(\ref{eq:1}) simply
by forgetting the level functions. We however give a more formal, inductive
definition which exhibits some features of other constructions used
later in this note.

Let $(U,\ell) \in \lT_\me$. Since our description of the
map~(\ref{eq:1}) will not depend on
the level function, we drop it from the notation. 
If $U$ is the up-rooted $n$-corolla $c_\me$,
$\me \geq 2$, i.e.~the tree with one vertex and $\me$ leaves, we put
\[
\varpi(c_\me) := \xi_\me \in  \free(\xi_2,\xi_3,\ldots)(\me).
\] 
Agreeing that $c_1$ denotes the exceptional tree $\except$, we extend
the above formula for $\me=1$ by
\[
\varpi_1(c_1) := e \in \free(\xi_2,\xi_3,\ldots)(1),
\]
where $e$ is the operad unit.
Let us proceed by induction on the number of vertices. 
An arbitrary $U \in  \lT_\me$, $\me \geq 2$, is of the 
form 
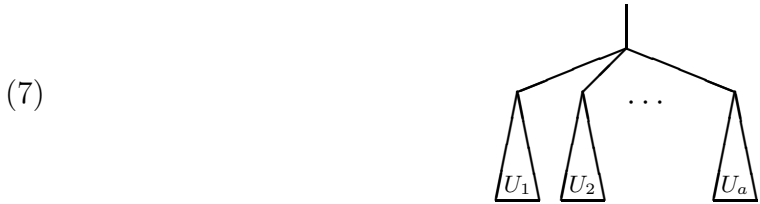
\begin{equation}
\raisebox{-3.2em}{}
\label{umru_na_rakovinu_plic?}
{
\unitlength=.07em
\thicklines
\begin{picture}(120.00,45.00)(0.00,45.00)
\put(110.00,7.00){\makebox(0.00,0.00){\scriptsize $U_a$}}
\put(40.00,7.00){\makebox(0.00,0.00){\scriptsize $U_2$}}
\put(10.00,7.00){\makebox(0.00,0.00){\scriptsize $U_1$}}
\put(70.00,45.00){\makebox(0.00,0.00){$\cdots$}}
\put(120.00,0.00){\line(-1,5){10.00}}
\put(100.00,0.00){\line(1,0){20.00}}
\put(110.00,50.00){\line(-1,-5){10.00}}
\put(50.00,0.00){\line(-1,5){10.00}}
\put(30.00,0.00){\line(1,0){20.00}}
\put(40.00,50.00){\line(-1,-5){10.00}}
\put(60.00,70.00){\line(-1,-1){20.00}}
\put(20.00,0.00){\line(-1,5){10.00}}
\put(0.00,0.00){\line(1,0){20.00}}
\put(10.00,50.00){\line(-1,-5){10.00}}
\put(60.00,70.00){\line(5,-2){50.00}}
\put(60.00,70.00){\line(-5,-2){50.00}}
\put(60.00,90.00){\line(0,-1){20.00}}
\end{picture}}  
\end{equation}
with some up-rooted, possibly exceptional, trees $U_1,\ldots,U_a$, $a
\geq 2$, each having strictly fewer vertices than $U$.  We
then put
\[
\varpi(U) := \xi_a\big(\varpi(U_1),\ldots,\varpi(U_a)\big) \in
\free(\xi_2,\xi_3,\ldots)(\me), 
\]
where $-(-,\ldots,-)$ in the right hand side denotes the operad composition.
Notice that we simplified the notation by dropping the subscripts of
$\varpi$.

For $(U',\ell), (U'',\ell'') 
\in \lT_\me$ let $(U',\ell') \sim (U'',\ell'')$ if $\varpi(U',\ell') = 
\varpi(U'',\ell'')$. Since obviously the latter equality holds 
if and only if $U' = U''$, the levels disappear and the quotient
$\lT_\me/\sim$ is isomorphic to the set $T_\me$ of up-rooted trees with $\me$
leaves. The \po\ of $\lT_\me$ induces the standard
\po\footnote{The one such that $T' < T''$ if and only if there exists
a morphism of planar up-rooted trees $T' \to T''$.} on $T_\me$, so we
have the isomorphism
\[
\EK_\me \cong\EP_{\me-1}/\sim.
\] 
We can take the above equation as a {\em definition\/} of the face poset of
the associahedron $K_\me$. The discrepancy between the indices ($\me$
versus $\me-1$) is of historical origin.

\subsection{Complementary pairs with zones}
\label{Vezmou_mi_kus_plice?}
For $m,n \geq 1$, consider a triple $(U,D,z)$ consisting of an
up-rooted planar tree $U$ with $\me$ leaves, a down-rooted planar tree
$D$ with $\en$ leaves, and an order preserving epimorphism
\begin{equation}
\label{eq:7}
z:
\vert(U) \cup \vert(D) \epi \{1,\ldots,l\}.
\end{equation}

We call $i\in \{1,\ldots,l\}$ such that $z^{-1}(i)$ contains both a
vertex of $U$ and a vertex of $D$ a~{\em barrier\/}. The remaining
$i$'s are the {\em zones\/} of $z$. If $z^{-1}(i) \subset \vert(U)$ 
(resp.~$z^{-1}(i) \subset \vert(D)$), we call $i$ an {\em up-zone\/}
(resp.~{\em down-zone\/}).

\begin{definitionA}
\label{sec:compl-pairs-with-2}
We call~(\ref{eq:7}) a {\em zone function} if 
\begin{itemize}
\item[(i)]
$z$ is strictly order-preserving on barriers and
\item[(ii)]
there are no adjacent zones of the same type.
\end{itemize}
We denote by $\zT^\en_\me$ the set of all triples $(U,D,z)$ consisting of a
planar up-rooted tree $U$ with $\me$ leaves, a planar down-rooted tree
with $\en$ leaves, and a zone function~(\ref{eq:7}). 
\end{definitionA}

Condition~(i) means the following. 
If $u',u'' \in \vert(U)$ and $v \in \vert(D)$ are such that $z(u')
= z(u'') = z(v)$, then $u'$ and $u''$ are unrelated.\footnote{By this
  we mean that neither $u' < u''$ nor  $u' > u''$.} Dually,  
if $v',v'' \in \vert(D)$ and $u \in \vert(U)$ are such that $z(u) = z(v')
= z(v'')$, then $v'$ and $v''$ are unrelated.
Condition~(ii) can be rephrased as follows. For $i \in
(1,\ldots,l)$ let 
\[
\textstyle
t_z(i) := \left\{
\begin{array}{ll}
{\rm U}&\mbox { if $i$ is an up-zone},
\\
{\rm D}&\mbox { if $i$ is an down-zone, and}
\\
{\rm B}&\mbox { if $i$ is a barrier.}
\end{array}
\right.
\] 
Condition~(ii) then says that the sequence
$\big(t_z(1)\cdots t_z(l)\big)$ does not contains subsequences UU or
DD. We call $\big(t_z(1)\cdots t_z(l)\big)$ the {\em type\/} of $z$.

The notion of complementary pairs with zones is illustrated in
Figure~\ref{fig:2}. In the picture,  
the values 1,3,4,7 are zones, the values 2,5,6 are
barriers. The type of the zone function is $({\rm DBDUBBUB})$.
\begin{figure}
\begin{center}
{
\unitlength=1pt
\begin{picture}(310.00,125)(0.00,0.00)
\put(310.00,15.00){\makebox(0.00,0.00){$8$ B}}
\put(310.00,29.0){\makebox(0.00,0.00){$7$ U}}
\put(310.00,42.50){\makebox(0.00,0.00){$6$ B}}
\put(310.00,57.00){\makebox(0.00,0.00){$5$ B}}
\put(310.00,70.00){\makebox(0.00,0.00){$4$ U}}
\put(310.00,88.00){\makebox(0.00,0.00){$3$ D}}
\put(310.00,100.00){\makebox(0.00,0.00){$2$ B}}
\put(310.00,112.50){\makebox(0.00,0.00){$1$ D}}
\put(210.00,15.00){\makebox(0.00,0.00){$\bullet$}}
\put(120.00,15.00){\makebox(0.00,0.00){$\bullet$}}
\put(70.00,15.00){\makebox(0.00,0.00){$\bullet$}}
\put(100.00,27.00){\makebox(0.00,0.00){$\bullet$}}
\put(70.00,33.00){\makebox(0.00,0.00){$\bullet$}}
\put(56.00,28.60){\makebox(0.00,0.00){$\bullet$}}
\put(90.00,42.50){\makebox(0.00,0.00){$\bullet$}}
\put(60.00,42.50){\makebox(0.00,0.00){$\bullet$}}
\put(210.00,42.50){\makebox(0.00,0.00){$\bullet$}}
\put(192.00,57.00){\makebox(0.00,0.00){$\bullet$}}
\put(58.00,57.00){\makebox(0.00,0.00){$\bullet$}}
\put(85.00,57.00){\makebox(0.00,0.00){$\bullet$}}
\put(90.00,67.00){\makebox(0.00,0.00){$\bullet$}}
\put(70.00,74.00){\makebox(0.00,0.00){$\bullet$}}
\put(235.00,92.00){\makebox(0.00,0.00){$\bullet$}}
\put(205.00,88.00){\makebox(0.00,0.00){$\bullet$}}
\put(210.00,100.00){\makebox(0.00,0.00){$\bullet$}}
\put(80.00,100.00){\makebox(0.00,0.00){$\bullet$}}
\put(250.00,113.00){\makebox(0.00,0.00){$\bullet$}}
\put(220.00,111.00){\makebox(0.00,0.00){$\bullet$}}
\put(190.00,109.00){\makebox(0.00,0.00){$\bullet$}}
\put(177.00,112.00){\makebox(0.00,0.00){$\bullet$}}
\thicklines
\put(270.00,120.00){\line(-1,-2){60.00}}
\put(150.00,120.00){\line(1,0){120.00}}
\put(210.00,0.00){\line(-1,2){60.00}}
\put(140.00,0.00){\line(-1,2){60.00}}
\put(20.00,0.00){\line(1,0){120.00}}
\put(80.00,120.00){\line(-1,-2){60.00}}
\put(30,-10.00){\vector(0,1){10.00}}
\put(210,-10.00){\vector(0,1){10.00}}
\put(80,120.00){\vector(0,1){10.00}}
\put(40,-10.00){\vector(0,1){10.00}}
\put(130,-10.00){\vector(0,1){10.00}}
\put(160,120.00){\vector(0,1){10.00}}
\put(170,120.00){\vector(0,1){10.00}}
\put(260,120.00){\vector(0,1){10.00}}
\put(210.00,125){\makebox(0.00,0.00)[b]{$\cdots$}}
\put(80.00,-5){\makebox(0.00,0.00)[t]{$\cdots$}}
\thinlines
\put(0.00,15){\line(1,0){290.00}}
\put(0.00,42.50){\line(1,0){290.00}}
\put(0.00,57.00){\line(1,0){290.00}}
\multiput(0.00,80.00)(9.8,0){30}{\line(1,0){4.9}}
\put(0.00,100.00){\line(1,0){290.00}}
\end{picture}}
\end{center}
\caption{\label{fig:2} An element of $\zT^\en_\me$.}
\end{figure}

\begin{Example}
Let us look at case $\en=2$ of Definition~A. 
For $X = (U,\fgen21,z) \in \zT^2_\me$, only the following three
cases may happen.

\noindent 
{\em Case $l = 1$.}  $1$ is a barrier if $\me \geq 2$ and $1$ is a
down-zone if $\me=1$.

\noindent 
{\em Case $l = 2$.} One has four possibilities for the type of $z$, namely
$({\rm D}{\rm U}), ({\rm U}{\rm D}), ({\rm B}{\rm U})$ 
or $({\rm U}{\rm B}).$
In the $({\rm D}{\rm U})$ and  $({\rm U}{\rm D})$ cases $\me \geq 2$,
in the remaining two cases $\me \geq 3$.

\noindent 
{\em Case $l = 3$.} $\me \geq 3$ and the only possibility for the
type is Alfred Jarry's $({\rm U}{\rm B}{\rm U})$.
\end{Example}

For $i \in \{\rada 1h\}$ define its {\em closure\/}
$\overline i 
\subset \{\rada 1h\}$ by
$\overline i := \{i\}$ if $i$ is a barrier, and $\overline i$ to be the
set consisting of $i$ and its adjacent barriers in the opposite case.
For the complementary pair in Figure~\ref{fig:2},
\[
\begin{array}{llll}
\ol1 = \{1,2\}, & \ol2 = \{2\}, &
\ol3 = \{2,3\}, &\ol4 = \{4,5\},
\\
\ol5 = \{5\}, &\ol6 = \{6\},
&
\ol7 = \{6,7,8\}, &\ol8 = \{8\}.
\end{array}
\]

A {\em morphism\/} $(U',D',z') \to (U'',D'',z'')$ is a triple
$(\phi_u,\phi_d,\hat \phi)$ consisting of morphisms $\phi_u : U' \to
U''$, $\phi_d : D'' \to D''$ of trees and of an order-preserving map
$\hat \phi : \{\rada 1{l'}\} \to \{\rada 1{l''}\}$ such that the
closures are preserved, that is
\[
z''\big(\phi_u(u)\big) \in \overline{\hat \phi\big(z'(u)\big)} \ \mbox { and }
z''\big(\phi_d(v)\big) \in \overline{\hat \phi\big(z'(v)\big)}, 
\]
for $u \in \vert(U)$ and $v \in \vert(D)$. We may also say that the
obvious analog of~(\ref{eq:2}), i.e.
\begin{equation}
\label{eq:8}
\raisebox{-2em}{}
{
\unitlength=1.000000pt
\begin{picture}(110.00,35.00)(0,15)
\thicklines
\put(125.00,30.00){\vector(0,-1){20.00}}
\put(-20,30.00){\vector(0,-1){20.00}}
\put(10,0.00){\vector(1,0){81}}
\put(35.00,40.00){\vector(1,0){33.00}}
\put(129,20.00){\makebox(0.00,0.00)[l]{\scriptsize $z''_u\cup z''_d$}}
\put(-26,20.00){\makebox(0.00,0.00)[r]{\scriptsize $z_u' \cup z'_d$}}
\put(50.00,5.00){\makebox(0.00,0.00)[b]{\scriptsize $\hat \phi$}}
\put(52.00,45.00){\makebox(0.00,0.00)[b]{\scriptsize $\phi_u
    \cup \phi_d$}}
\put(125,0.00){\makebox(0.00,0.00){\hphantom{.}$\{1,\ldots,h''\}$.}}
\put(-20,0.00){\makebox(0.00,0.00){$\{1,\ldots,h'\}$}}
\put(125,40.00){\makebox(0.00,0.00){$\vert(U'')\cup \vert(D'')$}}
\put(-20.00,40.00){\makebox(0.00,0.00){$\vert(U') \cup \vert(D')$}}
\end{picture}}
\end{equation} 
commutes up to the closures. The notion of a morphism induces a partial
order $<$.

\begin{definitionB}
\label{sec:compl-pairs-with-1}
The {\em face poset of the (step-one) biassociahedron\/} is the poset $\EK^\en_\me :=
(\zT^\en_\me,<)$ of complementary pairs of 
trees with zones, with the above partial order.
\end{definitionB}

Let $(U,D,\ell) \in \lT^\en_\me$ be a pair with the level function $\ell :
\vert(U) \cup \vert(D) \to \{1,\ldots,h\}$. We call $i \in
\{1,\ldots,h\}$ an {\em up-level\/} (resp.~{\em down-level\/}) if
$\ell^{-1}(i) \subset \vert(U)$ (resp.~$\ell^{-1}(i) \subset
\vert(D)$).

\begin{definition}
\label{sec:compl-pairs-with}
Let   $(U,D,\ell) \in \lT^\en_\me$ be as above and $(1,\ldots,l)$ the quotient
cardinal obtained from $(1,\ldots,h)$ by identifying the adjacent
up-levels and the adjacent down-levels. Denote by $p: (1,\ldots,h) \epi
(1,\ldots,l)$ the projection and define 
\[
\pi(U,D,\ell) := (U,D,z) \in \zT^\en_\me, \ \mbox { with } z := p \circ \ell.
\]
We call the map $\pi:  \lT^\en_\me \to \zT^\en_\me$ defined in this way
the {\em canonical projection\/} and $ z = p \circ \ell$ the 
{\em induced zone function\/}.  
\end{definition}

It is easy to show that the map $\pi$ preserves the partial orders,
giving rise to the projection $\EP^\en_\me \epi \EK^\en_\me$ of posets.
We finish this subsection by two statements needed in the proof of Theorem~C.

\begin{subequations}
\begin{proposition}
\label{sec:compl-pairs-with-3}
Let $(U,D,z'), (U,D,z'') \in \zT^\en_\me$. If, 
for each vertices $u \in \vert(U)$ and \hbox{$v \in \vert(D)$},
\begin{equation}
\label{eq:15}
z'(u) < z'(v)\ (\mbox {resp.~}z'(u) = z'(v),\
\mbox {resp.~} z'(u) > z'(v)) 
\end{equation}
is equivalent to
\begin{equation}
\label{eq:16}
z''(u) < z''(v)\ (\mbox {resp.~}z''(u) = z''(v),\
\mbox {resp.~} z''(u) > z''(v)),
\end{equation}
then $z' = z''$.
\end{proposition}
\end{subequations}

\begin{proof}
Let
\[
z' : \vert(U) \cup \vert (D) \to \{1,\ldots,h'\} \ \mbox { and } \
z'' : \vert(U) \cup \vert (D) \to \{1,\ldots,h''\}
\]
be zone functions as in the proposition. Let us show that, for
$x,y \in \vert(U) \cup \vert (D)$, 
\begin{equation}
\label{eq:14}
\mbox {if $z'(x) = z'(y)$ then $z''(x) = z''(y)$.} 
\end{equation}
The above implication immediately follows from the assumptions if $x
\in \vert(U)$ and $y \in \vert(D)$, or if $x
\in \vert(D)$ and $y \in \vert(U)$. So assume $x,y \in \vert(U)$,
$z'(x) = z'(y)$
and, say, $z''(x) > z''(y)$. Since, by definition, $z''$ does not have
two adjacent zones of type ${\rm U}$, there must exist $v \in
\vert(D)$ such that $z''(x) \geq z''(v) \geq  z''(y)$, where at least
one relation is sharp. Assuming the
equivalence between~(\ref{eq:15}) and~(\ref{eq:16}), we get
$z'(x) \geq z'(v) \geq  z'(y)$ where again at least one relation is
sharp, so $z'(x)
\not = z'(y)$, which is a contradiction. The case  $x,y \in \vert(D)$
can be discussed in the same way, thus~(\ref{eq:14}) is established.

For $i \in \{1,\ldots,h'\}$ (resp.~$j \in \{1,\ldots,h''\}$) denote
$S'_i := z'^{-1}(i)$ (resp.~$S''_j := z''^{-1}(j)$). By~(\ref{eq:14}),
for each $i$ there exists a unique $j$ such that $S'_i \subset
S''_j$. Exchanging the r\^oles of $z'$ and $z''$, we see that, vice
versa, for each $j$ there exists a unique $i$ such that $S''_j \subset
S'_i$. This obviously means that there exists an automorphism $\varphi
: \{1,\ldots,h'\} \to \{1,\ldots,h''\}$ such that $z' = \varphi \circ
z''$. As both $z'$ and $z''$ are order-preserving epimorphisms, 
$\varphi$ must be the identity. 
\end{proof}

The following proposition in conjunction with
Proposition~\ref{sec:compl-pairs-with-3} shows that the induced zone
function remembers the relative heights of vertices of $U$ and $D$ but
nothing more. 

\begin{proposition}
\label{Jarca_s_M1_na_chalupe}
Let $(U,D,\ell) \in \lT^\en_\me$ and $z$ be the zone function induced by
$\ell$. Then, for each $u \in \vert(U)$ and $v \in \vert(D)$,
\[
\ell(u) < \ell(v)\ (\mbox {resp.~}\ell(u) = \ell(v),\
\mbox {resp.~} \ell(u) > \ell(v))
\]
if and only if 
\[
z(u) < z(v)\ (\mbox {resp.~}z(u) = z(v),\
\mbox {resp.~} z(u) > z(v)).
\]
\end{proposition}

\begin{proof}
The proof is a simple application of the definition of the induced zone function.
\end{proof}

\subsection{The map $\varpi: \lT^\en_\me \to \free(\Xi)$.} 
\label{sec:map-pi:-ltm_n}
This subsection relies on the notation and terminology recalled in the
Appendix.  For $(U,D,\ell) \in \lT^\en_\me$ and subtrees $\overline U
\subset U$, $\overline D \subset D$ with, say, $\overline \me$ and
resp.~$\overline \en$ leaves, one has a natural {\em restriction\/}
\begin{equation}
\label{eq:9}
r_{\overline U, \overline D}(U,D,\ell) 
= (\overline U, \overline D, \overline \ell) \in
\lT^{\overline \en}_{\overline \me}
\end{equation}
with the level function $\overline \ell : \vert(\overline U) \cup
\vert( \overline D) \to (1,\ldots,\overline h)$ defined as follows.
The image of the restriction of $\ell$ to $\vert(\overline U) \cup
\vert( \overline D)$ is a sub-cardinal of $(1,\ldots, h)$, canonically
isomorphic to $(1,\ldots,\overline h)$ for some $\overline h \leq
h$. The level function $\overline \ell$ is then the composition of the
restriction of $\ell$ with this canonical isomorphism.  In other
words, $\overline \ell$ is the epimorphism in the factorization
\[
\vert(\overline U) \cup \vert(
\overline D) \stackrel {\overline \ell}\epi (1,\ldots,\overline h)
\hookrightarrow (1,\ldots,h)
\]
of the restriction of $\ell$.

Let $\free(\Xi)$
be the free PROP in the category of sets generated by
\[
\Xi := \big\{ \xi^\en_\me \ | \ m,n \geq 1, \ (m,n) \not= (1,1)\big\},
\]
where $\xi^\en_\me$ is the generator of biarity $(\en,\me)$ ($m$ inputs and
$n$ outputs). Observe that one has, for each $m,n \geq 1$, the inclusions
\begin{equation}
\label{eq:4}
\iota_U: \free(\xi_2,\xi_3,\ldots)(\me) \hookrightarrow \Fr \Fuk 1\me
    \ \mbox
     { and } 
\iota_D : \free(\xi_2,\xi_3,\ldots)(\en) \hookrightarrow \Fr\FuK(\en,1),
\end{equation}
given by
\[
\iota_U(\xi_a) := \xi^1_a \  \mbox { resp. }\ 
\iota_D(\xi_a) := \xi_1^a,\ a \geq 2.
\]
We will use $\iota_U$ to identify $\free(\xi_2,\xi_3,\ldots)(\me)$ with a
subset of $\Fr\FuK(1,\me)$.

Let us start the actual construction of the map $\varpi: \lT^\en_\me \to
\free(\Xi)\Fuk \en\me$.  First of all, for a~down-rooted tree $D$, i.e.~an
element of $\lT^\en_1$, define
\[
\varpi(D) := \iota_D\big(\varpi(D')\big)\in \free(\Xi) \Fuk \en1,
\] 
where $D'$ is $D$ turned upside down (i.e.~an up-rooted tree), $\varpi$ is as
in~\ref{ted_umiram_na_plice} and $\iota_D$ the second inclusion
of~(\ref{eq:4}). 

Each element $X \in \lT^\en_\me$ is a triple $X = (U,D,\ell)$, where $U$
is an up-rooted tree with $\me$ leaves and $D$ a down-rooted
tree with $\en$ leaves. We construct $\varpi(X)$ by induction on the
number of vertices of $D$. We distinguish three cases.

\noindent {\it Case 1.} 
The root vertex of $D$ is above the
root vertex of $U$,\footnote{That is the vertex adjacent to the root.}
schematically 
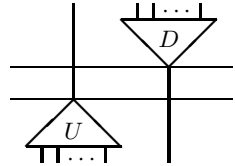
\[
{
\unitlength=.6pt
\begin{picture}(140.00,90)(0.00,12.00)
\put(0.00,40.00){\line(1,0){140.00}}
\put(0.00,60.00){\line(1,0){140.00}}
\thicklines
\put(106.00,92.00){\makebox(0.00,0.00)[b]{\scriptsize $\cdots$}}
\put(46.00,5){\makebox(0.00,0.00)[t]{\scriptsize $\cdots$}}
\put(100.00,78.00){\makebox(0.00,0.00){\scriptsize $D$}}
\put(40.00,20.00){\makebox(0.00,0.00){\scriptsize $U$}}
\put(60.00,10.00){\line(0,-1){10.00}}
\put(30.00,10.00){\line(0,-1){10.00}}
\put(20.00,10.00){\line(0,-1){10.00}}
\put(120.00,100.00){\line(0,-1){10.00}}
\put(90.00,100.00){\line(0,-1){10.00}}
\put(80.00,100.00){\line(0,-1){10.00}}
\put(100.00,60.00){\line(0,-1){60.00}}
\put(130.00,90.00){\line(-1,0){60.00}}
\put(100.00,60.00){\line(1,1){30.00}}
\put(70.00,90.00){\line(1,-1){30.00}}
\put(40.00,100.00){\line(0,-1){60.00}}
\put(80.00,60.00){\line(0,1){0.00}}
\put(70.00,10.00){\line(-1,1){30.00}}
\put(60.00,10.00){\line(1,0){10.00}}
\put(10.00,10.00){\line(1,0){50.00}}
\put(40.00,40.00){\line(-1,-1){30.00}}
\end{picture}}
\]
Then we put 
\begin{equation}
\label{eq:12}
\varpi(X) := \frac{\varpi(D)}{\varpi(U)} = \varpi(D) \circ \varpi(U)
\in \free(\Xi) \Fuk \en\me.
\end{equation}

\noindent {\it Case 2.} 
The vertex of $D$ is at the same level as  the
root vertex of $U$. We decompose $U$ as
in~(\ref{umru_na_rakovinu_plic?}) and $D$ in the obvious dual manner,
the result is portrayed in Figure~\ref{pojedu?}.
\begin{figure}
\[
\raisebox{-4.7em}{}
{
\unitlength=.07em
\begin{picture}(120.00,70.00)(50.00,66.00)
\put(-20.00,70.00){\makebox(0.00,0.00)[r]{$X =$}}
\put(-10.00,70.00){\line(1,0){280}}
\thicklines
\put(110.00,7.00){\makebox(0.00,0.00){\scriptsize $U_a$}}
\put(40.00,7.00){\makebox(0.00,0.00){\scriptsize $U_2$}}
\put(10.00,7.00){\makebox(0.00,0.00){\scriptsize $U_1$}}
\put(70.00,45.00){\makebox(0.00,0.00){$\cdots$}}
\put(120.00,0.00){\line(-1,5){10.00}}
\put(100.00,0.00){\line(1,0){20.00}}
\put(110.00,50.00){\line(-1,-5){10.00}}
\put(50.00,0.00){\line(-1,5){10.00}}
\put(30.00,0.00){\line(1,0){20.00}}
\put(40.00,50.00){\line(-1,-5){10.00}}
\put(60.00,70.00){\line(-1,-1){20.00}}
\put(20.00,0.00){\line(-1,5){10.00}}
\put(0.00,0.00){\line(1,0){20.00}}
\put(10.00,50.00){\line(-1,-5){10.00}}
\put(60.00,70.00){\line(5,-2){50.00}}
\put(60.00,70.00){\line(-5,-2){50.00}}
\put(60.00,90.00){\line(0,-1){20.00}}
\put(140,140){
\thicklines
\put(110.00,-7.00){\makebox(0.00,0.00){\scriptsize $D_b$}}
\put(40.00,-7.00){\makebox(0.00,0.00){\scriptsize $D_2$}}
\put(10.00,-7.00){\makebox(0.00,0.00){\scriptsize $D_1$}}
\put(70.00,-45.00){\makebox(0.00,0.00){$\cdots$}}
\put(120.00,0.00){\line(-1,-5){10.00}}
\put(100.00,0.00){\line(1,0){20.00}}
\put(110.00,-50.00){\line(-1,5){10.00}}
\put(50.00,0.00){\line(-1,-5){10.00}}
\put(30.00,0.00){\line(1,0){20.00}}
\put(40.00,-50.00){\line(-1,5){10.00}}
\put(60.00,-70.00){\line(-1,1){20.00}}
\put(20.00,0.00){\line(-1,-5){10.00}}
\put(0.00,0.00){\line(1,0){20.00}}
\put(10.00,-50.00){\line(-1,5){10.00}}
\put(60.00,-70.00){\line(5,2){50.00}}
\put(60.00,-70.00){\line(-5,2){50.00}}
\put(60.00,-90.00){\line(0,1){20.00}}
}
\end{picture}}
\hskip 4em  
\]
\caption{\label{pojedu?}The decomposition of $U$ and $D$ in the 2nd case.}
\end{figure}
We then define
\begin{equation}
\label{eq:13}
\varpi(X) :=
\begin{array}{c}
\varpi(D_1)\cdots \varpi(D_b)
\\
\hline
\rule{0pt}{1em}
\xi^b_a
\\
\hline
\varpi(U_1)\cdots \varpi(U_a)
\end{array}
=
 \big(\varpi(D_1) \boxtimes \cdots \boxtimes \varpi(D_b)\big)\circ \xi^b_a 
\circ\big(\varpi(U_1) \boxtimes \cdots \boxtimes \varpi(U_a)\big).
\end{equation}

\noindent {\it Case 3.} 
The root vertex of $D$ is below the
root vertex of $U$. In this case we decompose $U$ and $D$ as 
in Figure~\ref{diagnoza}
\begin{figure}
\begin{center}
{
\unitlength=.98pt
\begin{picture}(390.00,80.00)(-100.00,0.00)
\put(-40.00,40.00){\makebox(0.00,0.00)[r]{$X =$}}
\put(-300,0){
\put(270.00,20.00){\line(1,0){240.00}}
\thicklines
\put(340.00,10.00){\makebox(0.00,0.00){$\cdots$}}
\put(371.00,1.50){\makebox(0.00,0.00)[b]{\scriptsize $U_a$}}
\put(321.00,1.50){\makebox(0.00,0.00)[b]{\scriptsize $U_2$}}
\put(291.00,1.50){\makebox(0.00,0.00)[b]{\scriptsize $U_1$}}
\put(330.00,40.00){\makebox(0.00,0.00){\scriptsize $T$}}
\put(380.00,0.00){\line(-1,2){10.00}}
\put(360.00,0.00){\line(1,0){20.00}}
\put(370.00,20.00){\line(-1,-2){10.00}}
\put(330.00,0.00){\line(-1,0){20.00}}
\put(320.00,20.00){\line(1,-2){10.00}}
\put(320.00,20.00){\line(-1,-2){10.00}}
\put(370.00,20.00){\line(-1,1){40.00}}
\put(290.00,20.00){\line(1,0){80.00}}
\put(300.00,0.00){\line(-1,2){10.00}}
\put(280.00,0.00){\line(1,0){20.00}}
\put(290.00,20.00){\line(-1,-2){10.00}}
\put(330.00,60.00){\line(-1,-1){40.00}}
\put(330.00,80.00){\line(0,-1){20.00}}
}
\put(100,78.5){
\unitlength=.07em
\thicklines
\put(110.00,-7.00){\makebox(0.00,0.00){\scriptsize $D_b$}}
\put(60.00,-62.00){\makebox(0.00,0.00)[b]{\scriptsize $c^b_1$}}
\put(40.00,-7.00){\makebox(0.00,0.00){\scriptsize $D_2$}}
\put(10.00,-7.00){\makebox(0.00,0.00){\scriptsize $D_1$}}
\put(70.00,-25.00){\makebox(0.00,0.00){$\cdots$}}
\put(120.00,0.00){\line(-1,-5){10.00}}
\put(100.00,0.00){\line(1,0){20.00}}
\put(110.00,-50.00){\line(-1,5){10.00}}
\put(50.00,0.00){\line(-1,-5){10.00}}
\put(30.00,0.00){\line(1,0){20.00}}
\put(40.00,-50.00){\line(-1,5){10.00}}
\put(60.00,-70.00){\line(-1,1){20.00}}
\put(20.00,0.00){\line(-1,-5){10.00}}
\put(0.00,0.00){\line(1,0){20.00}}
\put(10.00,-50.00){\line(-1,5){10.00}}
\put(60.00,-70.00){\line(5,2){50.00}}
\put(60.00,-70.00){\line(-5,2){50.00}}
\put(60.00,-90.00){\line(0,1){20.00}}
}
\end{picture}}
\end{center}
\caption{\label{diagnoza}The decomposition of $U$ and $D$ in the $3$rd
case.}
\end{figure}
in which $T$ is the maximal up-rooted tree containing all vertices
above the level of the root vertex of $D$ and the up-rooted trees
$U_1,\ldots,U_a$ contain all the remaining vertices of $U$.  The
decomposition of $D$ is the same as in Case~2.  Using the
restriction~(\ref{eq:9}), we denote
\begin{equation}
\label{eq:10}
X_i :=  r_{U_i,c^b_1}(X),\ 1\leq i \leq a,\
\mbox { and }
Y_j := r_{T,D_j}(X),\ 1\leq j \leq b.
\end{equation}
Clearly $\varpi(X_i)$'s fall into the previous two cases.
Since $D_j$ has strictly less vertices than $D$, $\varpi(Y_j)$'s have
been defined by induction. We put
\begin{equation}
\label{rakovina??}
\varpi(X) = \frac{\varpi(Y_1) \cdots \varpi(Y_b)}{\varpi(X_1)\cdots\varpi(X_a)}  \in \FrFuK(\en,\me).
\end{equation}

\begin{Remark}
In the above construction of the map $\varpi$, the root vertex of $D$
plays a different r\^ole than the root vertex of $U$. One can exchange
the r\^oles of $U$ and $D$, arriving at a formally different formula
for $\varpi(X)$.  Due to the associativity of fractions
\cite[Section~6]{markl:ba}, both formulas give the same element of
$\FrFuK(\en,\me)$.  It is also possible to write a non-inductive formula
for $\varpi(X)$ based on the technique of {\em block transversal
matrices\/} developed in \cite{sanenlidze-umble:HHA11}.
\end{Remark}

\begin{Example}
If 
\[
X = \ \raisebox{-.5em}{\uuuua},
\]
we are in Case~3, with $T = U_1$ the up-rooted $2$-corollas $\fgen 12$,
and $U_2$ the exceptional tree $\except$. We have
\[
\raisebox{-1em}{}
X_1 =
\unitlength=0.03em
\begin{picture}(120.00,40.00)(0,0)
\put(10,-100){
\thicklines
\qbezier(80.00,120.00)(90.00,130.00)(100.00,140.00)
\qbezier(60.00,140.00)(70.00,130.00)(80.00,120.00)
\qbezier(20.00,100.00)(30.00,90.00)(40.00,80.00)
\qbezier(0.00,80.00)(10.00,90.00)(20.00,100.00)
\put(80.00,120.00){\line(0,-1){40.00}}
\put(20.00,140.00){\line(0,-1){40.00}}
\thinlines
\put(0.00,120.00){\line(1,0){100.00}}
\put(0.00,100.00){\line(1,0){100.00}}
}
\end{picture}
\ \mbox { and } \
X_2 =
\unitlength=0.03em
\begin{picture}(120.00,40.00)(0,0)
\put(10,-110){
\thicklines
\qbezier(80.00,120.00)(90.00,130.00)(100.00,140.00)
\qbezier(60.00,140.00)(70.00,130.00)(80.00,120.00)
\put(80.00,120.00){\line(0,-1){20.00}}
\put(40.00,140.00){\line(0,-1){40.00}}
\thinlines
\put(20.00,120.00){\line(1,0){80}}
}
\end{picture}.
\]
Both $X_1$ and $X_2$ fall into Case~1, and $\varpi(X_1) = \xi_1^2 \circ
\xi^1_2$ while $\varpi(X_2) = \xi^2_1$. Formula~(\ref{rakovina??})
gives
\[
\varpi(X) = \frac{\xi^1_2 \hskip .5em
  \xi^1_2}{\frac{\xi^2_1}{\xi^1_2}\hskip .5em \xi^2_1}=
  \frac{\gen12\gen12}{\dvojiteypsilon\gen21}.
\]
The rightmost term is obtained by depicting $\xi^\en_\me$ as 
an oriented corolla with $\me$ inputs and $\en$ outputs. 
The same notation is used in the 5th column of
the table of Figure~\ref{celeste} which lists elements in
the image $\varpi(\lT^2_3)$.
\end{Example}

Let us formulate the main result of this note which relates the
canonical projection of Definition~\ref{sec:compl-pairs-with} with the
map $\varpi$.

\begin{theoremC}
\label{zitra_do_Lnar?}
Let $X',X'' \in \lT^\en_\me$ be two complementary pairs of trees. Then
\[
\hskip .9em\varpi(X') = \varpi(X'') \ \mbox {(equality in $\FrFuK(\en,\me)$)}
\]
if and only if 
\[
\pi(X') = \pi(X'') \ \mbox{(equality in $\zT^\en_\me$),}\hskip 1.5em
\] 
so the image of $\varpi : \lT^\en_\me \to \Fr$ is isomorphic to $\zT^\en_\me$.
\end{theoremC}

Our {\em proof\/} of this theorem occupies the rest of this section.

\subsection*{Proof that $\pi(X') = \pi(X'')$ implies $\varpi(X') = \varpi(X'')$}
Parallel to $r_{\overline U, \overline D}(U,D,\ell) \in
\lT^{\overline \en}_{\overline \me}$ of~(\ref{eq:9}) there is a
similar restriction 
$s_{\overline U, \overline D}(U,D,z) \in
\zT^{\overline \en}_{\overline \me}$ defined for each $(U,D,z) \in
\zT^\en_\me$ and subtrees $\overline U \subset U$ and $\overline D \subset D$.
They commute with the canonical
projection in the sense that, for $X = (U,D,\ell) \in \lT^\en_\me$,
\begin{equation}
\label{eq:11}
\pi\big(r_{\overline U, \overline D}(X)\big) = 
s_{\overline U, \overline D}\big(\pi(X)\big).  
\end{equation}

The restriction $s_{\overline U,\overline D}$ can be defined along
similar lines as $r_{\overline U,\overline D}$. The only subtlety is
that the zone function restricted to $\vert (\overline U) \cup \vert
(\overline D)$ need not satisfy (ii) of
Definition~A, so we need to identify, in
its image, adjacent zones of the same type. We leave the details to
the reader.

The construction of $\varpi(X)$ given in \S\ref{sec:map-pi:-ltm_n} was
divided into three cases, determined by the relative positions of the
root vertices of $U$ and $D$. This information is, by
Proposition~\ref{Jarca_s_M1_na_chalupe}, retained by
the induced zone function of $\pi(X)$. Therefore the case into which $X$
falls depends only on the projection $\pi(X)$. 

Let $X',X'' \in \lT^\en_\me$ be such $\pi(X') = \pi(X'')$. Then $X'$ and
$X''$ may differ only by the level functions, i.e.~$X' =
(U,D,\ell')$ and $X''= (U,D,\ell'')$.  Let us proceed by induction on
the number of vertices of $D$.

If one (hence both) of $X'$, $X''$ falls into Case~1 or Case~2 of our
construction of $\varpi$, then clearly $\varpi(X'') = \varpi(X'')$ since in
these cases $\varpi$ manifestly depends only on the trees $U$ and $D$ not
on the level function. The induced zone function of
$\pi(X') = \pi(X'')$ must be of type $(\rm D \rm U)$ in Case~1 and $(\rm
D\rm B\rm U)$ in Case~2.

In Case~3 we observe first that the exact form of the decomposition in
Figure~\ref{diagnoza} depends only on the relative positions of the
root vertex of $D$ and the vertices of $U$. It is therefore, by
Proposition~\ref{Jarca_s_M1_na_chalupe}, determined by the induced
zone function, so it is the same for both $X'$ and $X''$. Now we
invoke the commutativity~(\ref{eq:11}) to check that
\[
\pi(X'_i) =  \pi\big(r_{U_i,c^b_1}(X')\big) = s_{U_i,c^b_1}\big(\pi(X')\big) =
s_{U_i,c^b_1}\big(\pi(X'')\big) =  \pi\big(r_{U_i,c^b_1}(X'')\big) = \pi(X''_i)
\]
for each $1 \leq i \leq a$. Similarly we verify that $\pi(Y'_j) =
\pi(Y''_j)$ for each $1 \leq j \leq b$. By the induction assumption, 
\[
\varpi(X'_i) = \varpi(X''_i) \ \mbox { and } \ 
\varpi(Y'_j) = \varpi(Y''_j)\ \mbox { for }\
1 \leq i \leq a,\ 1 \leq j \leq b,  
\]
therefore
\[
\varpi(X') = 
\frac{\varpi(Y'_1) \cdots \varpi(Y'_b)}{\varpi(X'_1)\cdots\varpi(X'_a)} =
\frac{\varpi(Y''_1) \cdots \varpi(Y''_b)}{\varpi(X''_1)\cdots\varpi(X''_a)} = \varpi(X'').
\]
This finishes our proof of the implication $\pi(X') = \pi(X'')
\Longrightarrow \varpi(X') = \varpi(X'')$.

\subsection*{Proof that $\varpi(X') = \varpi(X'')$ implies $\pi(X') = \pi(X'')$.}
Let us show that $\pi(X)$ is uniquely determined by $\varpi(X) \in
\FrFuK(\en,\me)$. As we already remarked, elements of $\Fr$ are represented
by directed graphs $G$ whose vertices are corollas $c^b_a$ with $a$
inputs and $b$ outputs, where $a,b \geq 1$, $(a,b) \not= (1,1)$. Let
$e$ be an internal edge of $G$, connecting an output of $c^s_r$ with an
input of $c^v_u$. We say that $e$ is {\em special\/} if either $s=1$
or $u=1$. The graph $G$ is {\em special\/} if all its internal edges
are special.  Finally, and element of $\Fr$ is {\em special\/} if it
is represented by a special graph.  We have the following simple lemma
whose proof immediately follows from the definition of the fraction.

\begin{lemma}
\label{Jarka_se_zas_opila}
Let $A_1, \ldots A_l, B_1, \ldots B_k \in \Fr$ be as in
Definition~\ref{3}. The fraction
\[
\frac{A_1 \cdots A_l}{B_1 \cdots B_k}
\]  
is special if and only if all 
$A_1, \ldots A_l, B_1, \ldots B_k$ are special and if
$k=1$ or $l=1$.
\end{lemma}

Lemma~\ref{Jarka_se_zas_opila} implies that $\varpi(X)$ is special if
and only if $X$ falls into Case~1 or Case~2 of
\S\ref{sec:map-pi:-ltm_n}. It is also clear that $X$ falls into Case~2
if and only if $\varpi(X)$ is special and the graph representing
$\varpi(X)$ has a (unique) vertex $c^b_a$ with $a,b \geq 2$. Therefore
$\varpi(X)$ bears the information to which case of its construction $X
= (U,D,\ell) \in \lT^\en_\me$ falls.

Suppose that $X$ falls to Case~1 of our definition of
$\varpi(X)$. Clearly, formula (\ref{eq:12}) uniquely determines the planar
up-rooted tree $U$ and a down-rooted tree $D$ such that $X =
(U,D,\ell)$. The only possible zone function $z$ for $\pi(X) = (U,D,z)$
is of type $(\rm D\rm U)$ with
\[
z\big(\vert(D)\big) = \{1\},\ z\big(\vert(U)\big)  =  \{2\}.
\]

If  $X = (U,D,\ell)$ falls
into Case~2 of our construction of $\varpi(X)$, we argue as in the
previous paragraph.
Formula~(\ref{eq:13}) uniquely determines
$\varpi(U_1),\ldots,\varpi(U_a)$ and $\varpi(Y_1),\ldots,\varpi(Y_b)$ and
therefore also the trees $U_1,\ldots,U_a,Y_1,\ldots,Y_b$ in the
decomposition in Figure~\ref{pojedu?}. Therefore also $U$ and $D$ are
uniquely determined, and clearly the only possible zone function $z$ for
$\pi(X) = (U,D,z)$ is of type $(\rm D,\rm B,\rm U)$ with
\begin{align*}
z\big(\vert(D_j)\big) = \{1\},&\ z\big(\vert(U_i)\big)  =  \{3\}, 
\\
z(\mbox {root vertex~of $U$})=&\ z(\mbox {root vertex~of $D$}) =  \{2\},
\end{align*}
$1 \leq i \leq a$, $1 \leq j \leq b$.

Assume that $X$ falls into Case~3 of our construction. Let us call,
only for the purposes of this proof, a directed graph a {\em
generalized tree\/}, if it is obtained by grafting directed up-rooted trees
$S_1,\rada,S_u$ into the inputs of the directed corolla $c^v_u$, with
some $u,v \geq 1$, $(u,v) \not= (1,1)$. So a generalized tree is a
directed graph of the form
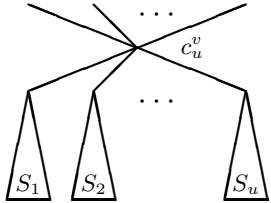
\[
\raisebox{-4.7em}{}
{
\unitlength=.07em
\begin{picture}(120.00,25.00)(0.00,66.00)
\thicklines
\put(110.00,7.00){\makebox(0.00,0.00){\scriptsize $S_u$}}
\put(40.00,7.00){\makebox(0.00,0.00){\scriptsize $S_2$}}
\put(10.00,7.00){\makebox(0.00,0.00){\scriptsize $S_1$}}
\put(70.00,45.00){\makebox(0.00,0.00){$\cdots$}}
\put(70.00,85){\makebox(0.00,0.00){$\cdots$}}
\put(120.00,0.00){\line(-1,5){10.00}}
\put(100.00,0.00){\line(1,0){20.00}}
\put(110.00,50.00){\line(-1,-5){10.00}}
\put(50.00,0.00){\line(-1,5){10.00}}
\put(30.00,0.00){\line(1,0){20.00}}
\put(40.00,50.00){\line(-1,-5){10.00}}
\put(60.00,70.00){\line(-1,-1){20.00}}
\put(60.00,70.00){\line(-1,1){20.00}}
\put(20.00,0.00){\line(-1,5){10.00}}
\put(0.00,0.00){\line(1,0){20.00}}
\put(10.00,50.00){\line(-1,-5){10.00}}
\put(60.00,70.00){\line(5,-2){50.00}}
\put(80,70){\makebox(0.00,0.00)[l]{\scriptsize $c^v_u$}}
\put(60.00,70.00){\line(5,2){50.00}}
\put(60.00,70.00){\line(-5,-2){50.00}}
\put(60.00,70.00){\line(-5,2){50.00}}
\end{picture}}
\]
where we keep our convention that all edges are oriented to point
upwards.  Since $X$ falls into Case~3, we know that $\varpi(X)$ is as
in~(\ref{rakovina??}), for some $X_i$, $Y_j$, $1 \leq i \leq a$, $1
\leq j \leq b$. We moreover know that the complementary pairs $X_i$
fall into Case~1 or Case~2 of the construction of $\varpi(X_i)$.

Now let $G_1,\ldots,G_r$ be the maximal generalized trees containing
the inputs of the graph representing $\varpi(X)$, numbered from left to
right.  It is clear from the definition of the fraction that $r=a$ and
that $G_i$ is, for each $1 \leq i \leq a$, the graph representing
$\varpi(X_i)$. So all $\varpi(X_1),\ldots,\varpi(X_a)$ are determined by
$\varpi(X)$.  A simple argument shows that if
\[
\frac{A'_1 \cdots A'_l}{B_1 \cdots B_k} 
 = 
\frac{A''_1 \cdots A''_l}{B_1 \cdots B_k} 
\]
for some $A'_1,\ldots, A'_l,A''_1,\ldots, A''_l \in \FrFuK(*,k)$ and 
$B_1,\ldots, B_k \in \FrFuK(l,*)$, then
$A'_i = A''_i$ for each $1 \leq i \leq l$. We therefore see that also
$\varpi(Y_1),\ldots,\varpi(Y_b)$ are determined by $\varpi(X)$.

Let us summarize what we have. We know each $\varpi(X_i)$. Since the
construction of $\varpi(X_i)$ falls into Case~1 or Case~2, we know, as we
have already proved, the trees $U_1,\ldots,U_a$ in
Figure~\ref{diagnoza}, and also the relative positions of
vertices of $U_1,\ldots,U_a$ and the root vertex of $D$.

Now we preform a similar analysis of $\varpi(Y_1),\ldots,\varpi(Y_b)$ and
repeat this process until we get trivial trees. It is clear that,
during this process, we fully reconstruct the trees $U,D$ in $X =
(U,D,\ell)$ and the relative positions of their vertices. By
Proposition~\ref{Jarca_s_M1_na_chalupe}, this uniquely determines the
zone function in $\pi(X) = (U,D,z)$. This finishes our proof of the
second implication.

\section{The particular case  $\EK^2_\me$}

In this section we analyze in detail the poset $\EK^2_\me$ for which the
notion of complementary pairs with zones takes a particularly simple form.

\subsection{Trees with a diaphragm.} 
\label{sec:trees-with-zones}
Let us consider the ordinal 
$\big\{(-\infty,1) < 1 < (1,+\infty)\big\}$. A~{\em diaphragm\/} of an
up-rooted tree $U$ is an order-preserving map 
\begin{equation}
\label{eq:6}
\zeta : \vert(U) \to \textstyle \big\{ (-\infty,1), 1, (1,+\infty)\big\}
\end{equation}
which is strictly order-preserving at $1$. 
By this we mean that, if
$\zeta(v') = \zeta(v'')= 1$ then neither $v' < v''$ nor  $v' >
v''$. We will denote $\bT^2_\me$ the set of all pairs $(U,\zeta)$, where
$U$ is an up-rooted tree with $\me$ leaves and $\zeta$ a diaphragm.

One may imagine a tree with a diaphragm as a planar up-rooted tree crossed by a
horizontal line, i.e.~a diaphragm, see the rightmost column of the table in
Figure~\ref{celeste} for examples.
It is convenient to introduce the following subsets of $\vert(U)$:
\[
\textstyle
\vert_{<1}(U) := \zeta^{-1}(-\infty,1),\ \vert_{1}(U) := \zeta^{-1}(1),\
\vert_{>1}(U) := \zeta^{-1}(1,+\infty)
\]
and the `closures'
\[
\vert_{\leq1}(U) := \vert_{<1}(U) \cup \vert_{1}(U)\ \mbox { and }
\vert_{\geq1}(U) := \vert_{>1}(U) \cup \vert_{1}(U).
\]

A {\em morphism\/} $\phi : (U',\zeta') \to (U'',\zeta'')$ of trees
with a diaphragm
is a morphism $\phi : U' \to U''$ of planar up-rooted trees which
preserves the closures, i.e.\ 
\[
\phi\big(\vert_1(U')\big) \subset \vert_1(U''),\
\phi\big(\vert_{< 1}(U')\big) \subset \vert_{\leq 1}(U''),
\ 
\phi\big(\vert_{> 1}(U')\big) \subset \vert_{\geq 1}(U'').
\]
We say that $(U',\zeta') < (U'',\zeta'')$ if there
exists a morphism $(U',\zeta') \to (U'',\zeta'')$.

\begin{proposition}
\label{Zitra_budu_chairmanem}
The posets $(\zT^2_\me,<)$ and $(\bT^2_\me,<)$ are, for each $\me \geq 1$,
naturally isomorphic.
\end{proposition}

\begin{proof}
For $X = (U,\fgen21,z)  \in \zT^2_\me$, denote $L$ the value of $z$ on the
vertex of $\fgen21$. Define $\zeta$ by
\[
\textstyle
\zeta(v) := \left\{
\begin{array}{ll}
(-\infty,0)&\mbox { if } z(v) < L,
\\
1&\mbox { if } z(v) = L,\ \mbox {and}
\\
(1,+\infty)&\mbox { if } z(v) > L.
\end{array}
\right.
\] 
It is easy to see that $(U,\zeta)$ is a tree with a diaphragm, that
the correspondence $(U,z) \mapsto (U,\zeta)$ is one-to-one and that it
preserves the partial orders.
\end{proof}

The natural projection $\pi: \lT^2_\me \to \zT^2_\me$ can be, in terms of
trees with a diaphragm, described as follows. Let $X =
(U,\fgen21,\ell) \in \lT^2_\me$ and assume that the vertex of $\fgen21$
is placed at level $L$. Then $\pi(X) := (U,\zeta)$, with the diaphragm
\[
\textstyle
\zeta(v) := \left\{
\begin{array}{ll}
(-\infty,0)&\mbox { if } \ell(v) < L,
\\
1&\mbox { if } \ell(v) = L,\ \mbox {and}
\\
(1,+\infty)&\mbox { if } \ell(v) > L.
\end{array}
\right.
\] 

\begin{Example}
The $\pi$-images of complementary pairs in $\lT^2_3$ are listed in
the rightmost column of the table in Figure~\ref{celeste}.
\end{Example}

\begin{Example}
Figure~\ref{Jak_dlouho_budu_zit?} illustrates the projection $P^2_4
\to K^2_4$. It shows the face 
poset of a square face of the $3$-dimensional $P^2_4$
together with the corresponding complementary pairs in $\lT^2_4$ and
its projection, which is in this case the poset of the interval
indexed by the corresponding trees with a diaphragm. 
\begin{figure}
\[
{
\unitlength=.09em
\begin{picture}(250.00,140.00)(0.00,-10.00)
\put(60.00,60.00){\makebox(0.00,0.00){$\DEVET$}}
\put(120.00,60.00){\makebox(0.00,0.00)[l]{$\OSM$}}
\put(0.00,60.00){\makebox(0.00,0.00)[r]{$\SEDM$}}
\put(60.00,0.00){\makebox(0.00,0.00)[t]{$\SEST$}}
\put(60.00,120.00){\makebox(0.00,0.00)[b]{$\PET$}}
\put(120.00,0.00){\makebox(0.00,0.00)[tl]{$\CTYRI$}}
\put(0.00,0.00){\makebox(0.00,0.00)[tr]{$\TRI$}}
\put(120.00,120.00){\makebox(0.00,0.00)[bl]{$\DVA$}}
\put(0.00,120.00){\makebox(0.00,0.00)[br]{$\JEDNA$}}
\put(250.00,60.00){\makebox(0.00,0.00)[l]{$\DVANACT$}}
\put(240.00,0.00){\makebox(0.00,0.00)[t]{$\JEDENACT$}}
\put(240.00,120.00){\makebox(0.00,0.00)[b]{$\DESET$}}
\thicklines
\put(190.00,63.00){\makebox(0.00,0.00)[b]{\scriptsize $\pi$}}
\put(165.00,60.00){\vector(1,0){60.00}}
\put(240.00,10.00){\makebox(0.00,0.00){$\bullet$}}
\put(240.00,110.00){\makebox(0.00,0.00){$\bullet$}}
\put(110.00,10.00){\makebox(0.00,0.00){$\bullet$}}
\put(10.00,10.00){\makebox(0.00,0.00){$\bullet$}}
\put(110.00,110.00){\makebox(0.00,0.00){$\bullet$}}
\put(10.00,110.00){\makebox(0.00,0.00){$\bullet$}}
\put(240.00,10.00){\line(0,1){100.00}}
\put(110.00,10.00){\line(-1,0){100.00}}
\put(110.00,110.00){\line(0,-1){100.00}}
\put(10.00,110.00){\line(1,0){100.00}}
\put(10.00,10.00){\line(0,1){100.00}}
\end{picture}}
\]
\caption{\label{Jak_dlouho_budu_zit?}
The projection of a face of $P^2_4$ to $K^2_4$. The faces of the
interval in $K^2_4$ (right) are indexed by trees with a diaphragm.}
\end{figure}
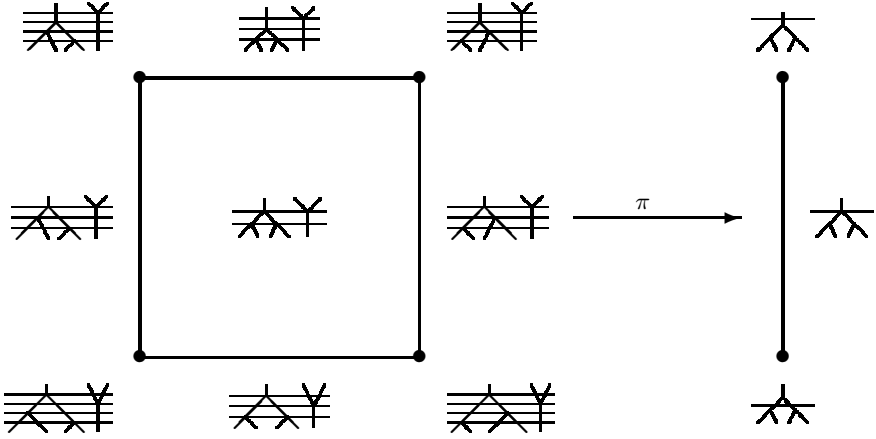
\end{Example}

As an exercise, we recommend describing the map $\varpi: \zT^2_\me \to
\Fr\Fuk 2\me$ in terms of trees with a diaphragm.
One may generalize the above description of the poset 
$\zT^\en_\me$ also to $\en > 2$. In this case, the tree $U$
corresponding to $(U,D,z) \in \zT^\en_\me$ may have {\em several\/} diaphragms,
depending on the relative positions of the vertices of $D$. The
combinatorics of this kind of description becomes, however, unmanageably
complicated with growing $\en$.

\subsection{Relation to the multiplihedron}
Multiplihedra appeared in the study of homotopy multiplicative maps
between $A_\infty$-spaces~\cite{jim:book}. The $\me$-th multiplihedron
$J_\me$ is a convex polytope of dimension $\me-1$ whose vertices
correspond to ways of bracketings $\me$ variables and applying an
operation. As explained in~\cite{forcey:hull}, the faces of $J_\me$ are
indexed by {\em painted $\me$-trees\/} which are, by definition,
directed (rooted) planar trees with $\me$ leaves, two types of edges --
black and white -- and vertices of the following two types:
\begin{itemize}
\item[(i)] 
vertices with at least two inputs 
whose all adjacent edges are of the same color, or
\item[(ii)] 
vertices whose  all inputs are white and whose output is black.
\end{itemize}

The set $\pT_\me$  of all painted $\me$-trees has a
partial order $<$ induced by contracting the edges. The
poset $\J_\me := (\pT_\me,<)$ is then the poset of faces of the $\me$-th
multiplihedron $J_\me$. We believe that Figure~\ref{fig:psano_v_Srni}
makes the above definitions clear. 

\begin{figure}
{
\unitlength=.85pt
\begin{picture}(200.00,210.00)(0.00,-10.00)
\put(185.00,60){\makebox(0.00,0.00)[tl]{$\painteduuuucutcutmodmod$}}
\put(100.00,5.00){\makebox(0.00,0.00)[t]{$\painteduuuucutcutmodmodmod$}}
\put(15.00,60.00){\makebox(0.00,0.00)[rt]{$\paintedascicut$}}
\put(15.00,125.00){\makebox(0.00,0.00)[br]{$\painteduuuucutcut$}}
\put(100.00,180.00){\makebox(0.00,0.00)[b]{$\painteduuuucutcutmod$}}
\put(185.00,125.00){\makebox(0.00,0.00)[lb]{$\paintedpoiucutmmod$}}
\put(170.00,10.00){\makebox(0.00,0.00)[lt]
{$\painteduuuuacutmodred\raisebox{5pt}{\scriptsize $=f(\bullet)f(\bullet\bullet)$}$}}
\put(30.00,10.00){\makebox(0.00,0.00)[tr]{$\raisebox{5pt}{\scriptsize $f\big(\bullet(\bullet\bullet)\big)=\ $}\painteduuuucutmod$}}
\put(-10.00,90.00){\makebox(0.00,0.00)[r]
{$\raisebox{5pt}{\scriptsize \scriptsize $f\big((\bullet\bullet)\bullet\big)=\ $}\painteduuuucut$}}
\put(30.00,170.00){\makebox(0.00,0.00)[rb]{$\raisebox{5pt}{\scriptsize $f(\bullet\bullet)f(\bullet)=\ $}\painteduuuuacutmmo$}}
\put(170.00,170.00){\makebox(0.00,0.00)[lb]
{$\paintedffffcutred\raisebox{5pt}{\scriptsize $\ =\big(f(\bullet)f(\bullet)\big)f(\bullet)$}$}}
\put(210.00,90.00){\makebox(0.00,0.00)[l]
{$\paintedffffcutmodred\raisebox{5pt}{\scriptsize $\ =f(\bullet)\big(f(\bullet)f(\bullet)\big)$}$}}
\put(100.00,90.00){\makebox(0.00,0.00){$\paindedascicutmodvetsi$}}
\put(160.00,10.00){\makebox(0.00,0.00){\large$\bullet$}}
\put(40.00,10.00){\makebox(0.00,0.00){\large$\bullet$}}
\put(0.00,90.00){\makebox(0.00,0.00){\large$\bullet$}}
\put(40.00,170.00){\makebox(0.00,0.00){\large$\bullet$}}
\put(160.00,170.00){\makebox(0.00,0.00){\large$\bullet$}}
\put(200.00,90.00){\makebox(0.00,0.00){\large$\bullet$}}
\thicklines
\put(160.00,10.00){\line(-1,0){10.00}}
\put(200.00,90.00){\line(-1,-2){40.00}}
\put(160.00,170.00){\line(1,-2){40.00}}
\put(150.00,170.00){\line(1,0){10.00}}
\put(40.00,170.00){\line(1,0){110.00}}
\put(40.00,10.00){\line(1,0){110.00}}
\put(0.00,90.00){\line(1,-2){40.00}}
\put(40.00,170.00){\line(-1,-2){40.00}}
\end{picture}}
\caption{
\label{fig:psano_v_Srni}
The faces of the multiplihedron $J_3$ indexed by the set $\pT_3$ of
painted $3$-trees. The
labels of vertices in terms of bracketings of $3$ variables and an
operation $f$ are also shown.}
\end{figure}
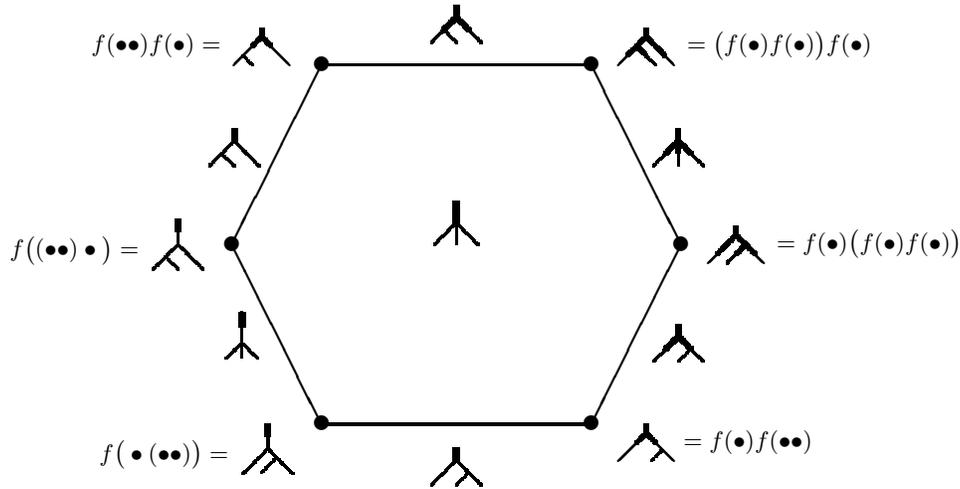 

\begin{propositionD}
\label{dnes_s_TK_Holter}
The face poset $\EK^2_\me$ of the biassociahedron $K^2_\me$ is isomorphic to the
face poset $\J_\me$ of the multiplihedron $J_\me$, for each $\me \geq 2$.
\end{propositionD}

\begin{proof}
By Proposition~\ref{Zitra_budu_chairmanem}, it suffices to prove that
the posets $(\bT_\me,<)$ and $(\pT_\me,<)$ are isomorphic. It is very
simple. Having a tree $U$ with a diaphragm, we paint everything that
lies above\footnote{We keep our convention that all
edges are oriented to point upwards.} the diaphragm black, and everything below
white. If the diaphragm intersects an edge of $U$, we introduce at the
intersection a new vertex of type (ii) with one input edge. The result
will obviously be a painted tree belonging to $\pT_\me$. The isomorphism we have
thus described clearly preserves the partial orders.
\end{proof}

The correspondence of Proposition~D is, for $\me=3$, illustrated by
the two rightmost columns of the table in Figure~\ref{celeste}.

\begin{Remark}
S.~Forcey in~\cite{forcey:hull} constructed an explicit
realization of the poset $\J_\me$ by the face poset of a convex
polyhedron. This, combined with Proposition~D
proves, independently of~\cite{sanenlidze-umble:HHA11}, that $\EK^2_\me$
is the face poset of a convex polyhedron, too.
\end{Remark}

\appendix
\section{Calculus of fractions}

A PROP in the monoidal category of sets
is a sequence of
sets $\sfP = \big\{\sfP \Fuk \en\me\big\}_{m,n \geq 1}$ with compatible left
$\Sigma_\me$- right $\Sigma_\en$-actions and two types of equivariant
compositions, vertical:
\[
\circ: 
\sfP \Fuk \en u \times_{\Sigma_u} \sfP \Fuk u\me \to \sfP \Fuk \en\me, \ m,n,u \geq 1,
\] 
and horizontal:
\[
\boxtimes : \sfP \Fuk {\en_1}{\me_1} \times  \sfP \Fuk{\en_2}{\me_2} \to 
\sfP\Fuk {\en_1+\en_2}{\me_1+\me_2},\
m_1,m_2,n_1,n_2 \geq 1,
\]
together with an identity $e \in \sfP \Fuk11$, satisfying appropriate
axioms~\cite{maclane:RiceUniv.Studies63,markl:handbook}.  One can
imagine elements of $\sfP\Fuk \en\me$ as `abstract' operations with $\me$ inputs
and $\en$ outputs.  We say that $X$ has {\em biarity\/} $\Fuk \en\me$ if $X
\in \sfP \Fuk \en\me$.

Calculus of fractions was devised in \cite[Section~4]{markl:ba} 
to handle particular
types of compositions in PROPs.
For $k ,l \geq 1$ and $1 \leq i \leq kl$, let $\sigma\Fuk lk \in
\Sigma_{kl}$ be the permutation given~by
\[
\sigma\Fuk lk(i) := l(i-1 - (s-1)k) + s,
\]
where $s$ is such that $(s-1)k < i \leq sk$.

\begin{Example}
We have
\[
\sigma \Fuk22 = 
\left(
\begin{array}{cccc}
1 & 2 & 3 & 4
\\
1 & 3 & 2 & 4
\end{array}
\right)
= \hskip 2mm
{
\unitlength=1.000000pt
\begin{picture}(30.00,30.00)(0.00,12.00)
\put(30.00,0.00){\makebox(0.00,0.00){\scriptsize$\bullet$}}
\put(30.00,30.00){\makebox(0.00,0.00){\scriptsize$\bullet$}}
\put(20.00,0.00){\makebox(0.00,0.00){\scriptsize$\bullet$}}
\put(20.00,30.00){\makebox(0.00,0.00){\scriptsize$\bullet$}}
\put(10.00,0.00){\makebox(0.00,0.00){\scriptsize$\bullet$}}
\put(10.00,30.00){\makebox(0.00,0.00){\scriptsize$\bullet$}}
\put(0.00,0.00){\makebox(0.00,0.00){\scriptsize$\bullet$}}
\put(0.00,30.00){\makebox(0.00,0.00){\scriptsize$\bullet$}}
\put(30.00,20.00){\line(0,-1){10.00}}
\put(20.00,20.00){\line(-1,-1){10.00}}
\put(10.00,20.00){\line(1,-1){10.00}}
\put(0.00,20.00){\line(0,-1){10.00}}
\end{picture}}\hskip .5em \in \Sigma_4 .
\]
Similarly
\[
\sigma\Fuk32 := 
\left(
\begin{array}{cccccc}
1 & 2 & 3 & 4 & 5 & 6
\\
1 & 4 & 2 & 5 & 3 & 6
\end{array}
\right)
= \hskip 3mm
{
\unitlength=1.000000pt
\begin{picture}(50.00,30.00)(0.00,12.00)
\put(50.00,0.00){\makebox(0.00,0.00){\scriptsize$\bullet$}}
\put(50.00,30.00){\makebox(0.00,0.00){\scriptsize$\bullet$}}
\put(40.00,0.00){\makebox(0.00,0.00){\scriptsize$\bullet$}}
\put(40.00,30.00){\makebox(0.00,0.00){\scriptsize$\bullet$}}
\put(30.00,0.00){\makebox(0.00,0.00){\scriptsize$\bullet$}}
\put(30.00,30.00){\makebox(0.00,0.00){\scriptsize$\bullet$}}
\put(20.00,0.00){\makebox(0.00,0.00){\scriptsize$\bullet$}}
\put(20.00,30.00){\makebox(0.00,0.00){\scriptsize$\bullet$}}
\put(10.00,0.00){\makebox(0.00,0.00){\scriptsize$\bullet$}}
\put(10.00,30.00){\makebox(0.00,0.00){\scriptsize$\bullet$}}
\put(0.00,0.00){\makebox(0.00,0.00){\scriptsize$\bullet$}}
\put(0.00,30.00){\makebox(0.00,0.00){\scriptsize$\bullet$}}
\put(0.00,30.00){\line(0,1){0.00}}
\put(50.00,20.00){\line(0,-1){10.00}}
\put(40.00,20.00){\line(-1,-1){10.00}}
\put(30.00,20.00){\line(-2,-1){20.00}}
\put(20.00,20.00){\line(2,-1){20.00}}
\put(10.00,20.00){\line(1,-1){10.00}}
\put(0.00,20.00){\line(0,-1){10.00}}
\end{picture}} \hskip 2mm \in \Sigma_6.
\]
\end{Example}

\begin{definition}
\label{3}
Let $\sfP$ be an arbitrary PROP. 
Let $k,l \geq 1$, $\Rada a1l,\Rada b1k \geq 1$,  $\Rada A1l \in
\sfP\Fuk k{a_j}$ and 
$\Rada B1k \in \sfP \Fuk {b_i}l$.
Then define the {\em fraction\/}
\[
\frac{B_1 \cdots B_k}{A_1 \cdots A_l} :=
(B_1 \boxtimes \cdots \boxtimes B_k) 
\circ \sigma \Fuk lk \circ 
(A_1 \boxtimes \cdots \boxtimes A_l)
\in \sfP\Fuk{b_1+\cdots + b_k}{a_1+\cdots + a_l}.
\]
\end{definition}

\begin{Example}
If $k=1$ or $l=1$, the fractions give the `operadic' and `cooperadic'
compositions:
\[
\frac{B_1}{A_1  \cdots  A_l} = B_1 \circ (A_1
\boxtimes \cdots \boxtimes A_l)
\ \mbox { and }\
\frac{B_1 \cdots B_k}{A_1} = (B_1 \boxtimes \cdots \boxtimes B_k) \circ A_1.
\]
\end{Example}

\begin{Example}
For $\zeroatwo a, \zeroatwo b \in \sfP\FuK(*,2)$ and $\twoazero c,
\twoazero d \in \sfP\FuK(2,*)$,
\[
\frac{\zeroatwo a \hskip .2em  \zeroatwo b}%
     {\twoazero c \hskip .2em \twoazero d}
=
(\raisebox{-2pt}{\zeroatwo a} \hskip -.3em 
\lboxtimes \raisebox{-2pt}{\zeroatwo b} 
\hskip -4pt) 
\circ \sigma\FuK(2,2) \circ 
(\raisebox{-2pt}{\twoazero c} 
\hskip -.3em\lboxtimes \raisebox{-2pt}{\twoazero d} \hskip -4pt)
=
{
\unitlength=.5pt
\begin{picture}(70.00,70.00)(0.00,30.00)
\put(50.00,10.00){\makebox(0.00,0.00)[l]{$d$}}
\put(10.00,10.00){\makebox(0.00,0.00)[l]{$c$}}
\put(50.00,60.00){\makebox(0.00,0.00)[l]{$b$}}
\put(10.00,60.00){\makebox(0.00,0.00)[l]{$a$}}
\put(60.00,50.00){\line(0,-1){30.00}}
\put(40.00,0.00){\line(0,1){20.00}}
\put(70.00,0.00){\line(-1,0){30.00}}
\put(70.00,20.00){\line(0,-1){20.00}}
\put(40.00,20.00){\line(1,0){30.00}}
\put(40.00,70.00){\line(0,-1){20.00}}
\put(70.00,70.00){\line(-1,0){30.00}}
\put(70.00,50.00){\line(0,1){20.00}}
\put(40.00,50.00){\line(1,0){30.00}}
\put(50.00,50.00){\line(-1,-1){30.00}}
\put(20.00,50.00){\line(1,-1){30.00}}
\put(0.00,0.00){\line(0,1){20.00}}
\put(30.00,0.00){\line(-1,0){30.00}}
\put(30.00,20.00){\line(0,-1){20.00}}
\put(0.00,20.00){\line(1,0){30.00}}
\put(10.00,40.00){\line(0,-1){20.00}}
\put(10.00,50.00){\line(0,-1){10.00}}
\put(30.00,70.00){\line(-1,0){30.00}}
\put(30.00,50.00){\line(0,1){20.00}}
\put(0.00,50.00){\line(1,0){30.00}}
\put(0.00,70.00){\line(0,-1){20.00}} 
\end{picture}}\hskip 2mm.
\]
Similarly, for $\zeroathree x, \zeroathree y \in \sfP\FuK(*,3)$ and
$\twoazero z, \twoazero u, \twoazero v \in \sfP\FuK(2,*)$,

\vglue -1.3em
\[
\frac{\zeroathree x \hskip 4mm \zeroathree y}%
     {\twoazero z \hskip 1mm \twoazero u \hskip 1mm \twoazero v} 
=
(\hskip -2pt \raisebox{-2pt}{\zeroathree x}\hskip -.2em \lboxtimes\hskip -.2em
\raisebox{-2pt}{\zeroathree y}\hskip -2pt) 
\circ \sigma\FuK(3,2) \circ
(\raisebox{-2pt}{\twoazero z}\hskip -.2em 
\lboxtimes  \raisebox{-2pt}{\twoazero u}  \hskip -.2em
\lboxtimes \raisebox{-2pt}{\twoazero v}\hskip -2pt) 
=
{
\unitlength=.5pt
\begin{picture}(110.00,70.00)(0.00,30.00)
\put(90.00,10.00){\makebox(0.00,0.00)[l]{$v$}}
\put(50.00,10.00){\makebox(0.00,0.00)[l]{$u$}}
\put(10.00,10.00){\makebox(0.00,0.00)[l]{$z$}}
\put(90.00,60.00){\makebox(0.00,0.00){$y$}}
\put(20.00,60.00){\makebox(0.00,0.00){$x$}}
\put(110.00,20.00){\line(-1,0){30.00}}
\put(110.00,0.00){\line(0,1){20.00}}
\put(80.00,0.00){\line(1,0){30.00}}
\put(80.00,20.00){\line(0,-1){20.00}}
\put(40.00,0.00){\line(0,1){20.00}}
\put(70.00,0.00){\line(-1,0){30.00}}
\put(70.00,20.00){\line(0,-1){20.00}}
\put(40.00,20.00){\line(1,0){30.00}}
\put(0.00,0.00){\line(0,1){20.00}}
\put(30.00,0.00){\line(-1,0){30.00}}
\put(30.00,20.00){\line(0,-1){20.00}}
\put(0.00,20.00){\line(1,0){30.00}}
\put(60.00,20.00){\line(1,1){30.00}}
\put(20.00,20.00){\line(2,1){60.00}}
\put(30.00,50.00){\line(2,-1){60.00}}
\put(20.00,50.00){\line(1,-1){30.00}}
\put(100.00,50.00){\line(0,-1){30.00}}
\put(70.00,70.00){\line(0,-1){20.00}}
\put(110.00,70.00){\line(-1,0){40.00}}
\put(110.00,50.00){\line(0,1){20.00}}
\put(70.00,50.00){\line(1,0){40.00}}
\put(0.00,70.00){\line(0,-1){20.00}}
\put(40.00,70.00){\line(-1,0){40.00}}
\put(40.00,50.00){\line(0,1){20.00}}
\put(0.00,50.00){\line(1,0){40.00}}
\put(10.00,50.00){\line(0,-1){30.00}}
\end{picture}}\hskip 2pt.
\]
\end{Example}


\begin{thebibliography}{1}

\bibitem{forcey:hull}
S.~Forcey.
\newblock Convex hull realizations of the multiplihedra.
\newblock {\em Topology Appl.}, 156(2):326--347, 2008.


\bibitem{maclane:RiceUniv.Studies63}
S.~Mac~Lane.
\newblock Natural associativity and commutativity.
\newblock {\em Rice Univ. Studies}, 49(4):28--46, 1963.

\bibitem{markl:ba}
M.~Markl.
\newblock A resolution (minimal model) of the {PROP} for bialgebras.
\newblock {\em J. Pure Appl. Algebra}, 205(2):341--374, 2006.

\bibitem{markl:ha}
M.~Markl.
\newblock Homotopy algebras are homotopy algebras.
\newblock {\em Forum Math.}, 16(1):129--160, 2004.

\bibitem{markl:zebrulka}
M.~Markl.
\newblock Models for operads.
\newblock {\em Comm. Algebra}, 24(4):1471--1500, 1996.

\bibitem{markl:handbook}
M.~Markl.
\newblock Operads and {PROP}s.
\newblock In {\em Handbook of algebra. {V}ol. 5}, 
pages 87--140. Elsevier/North-Holland, Amsterdam, 2008.

\bibitem{markl-shnider-stasheff:book}
M.~Markl, S.~Shnider, and J.D. Stasheff.
\newblock {\em Operads in algebra, topology and physics}, volume~96 of {\em
  Mathematical Surveys and Monographs}.
\newblock American Mathematical Society, Providence, RI, 2002.

\bibitem{sanenlidze-umble:HHA11}
S.~Saneblidze and R.~Umble.
\newblock Matrads, biassociahedra, and {$A_\infty$}-bialgebras.
\newblock {\em Homology, Homotopy Appl.}, 13(1):1--57, 2011.

\bibitem{stasheff:TAMS63}
J.D. Stasheff.
\newblock Homotopy associativity of {$H$}-spaces. {I}, {II}.
\newblock {\em Trans. Amer. Math. Soc. 108 (1963), 275-292; ibid.},
  108:293--312, 1963.


\bibitem{jim:book}
J.D. Stasheff.
\newblock {\em H-spaces from a homotopy point of view}, volume 161 of {\em
  Lecture Notes in Math.}
\newblock Springer-Verlag, 1970.


\bibitem{tonks97}
A.~Tonks.
\newblock Relating the associahedron and the permutohedron.
\newblock In J.-L. Loday, J.D. Stasheff, and A.A.~Voronov, editors, {\em Operads:
  Proceedings of Renaissance Conferences}, volume 202 of {\em Contemporary
  Math.}, pages 33--36. Am. Math. Soc., 1997.

\end{thebibliography}

\def\cprime{$'$} \def\cprime{$'$}

\end{document}